\documentclass[a4paper]{amsart}

\usepackage{hyperref}			
\usepackage{amssymb}			
\usepackage{amsmath}			
\usepackage{mathrsfs}			
\usepackage{cite}			
\usepackage{enumerate}			
\usepackage{tikz}[2010/10/13]
\edef\pgfkeysatcode{\the\catcode`\@}
\catcode`\@=11

\ifx\pgfrcsincluded\undefined
\def\pgfrcsincluded{}

\edef\pgfrcsatcode{\the\catcode`\@}
\catcode`\@=11

\catcode`\@=11\relax

\newif\ifpgfutil@format@is@latex
\newif\ifpgfutil@format@is@plain
\newif\ifpgfutil@format@is@context

\def\pgfutil@ifundefined#1{%
  \expandafter\ifx\csname#1\endcsname\relax
    \expandafter\pgfutil@firstoftwo
  \else
    \expandafter\pgfutil@secondoftwo
  \fi}

\def\pgfutil@IfUndefined#1{%
  \begingroup\expandafter\expandafter\expandafter\endgroup
  \expandafter\ifx\csname#1\endcsname\relax
    \expandafter\pgfutil@firstoftwo
  \else
    \expandafter\pgfutil@secondoftwo
  \fi
}
\long\def\pgfutil@firstofone#1{#1}
\long\def\pgfutil@firstoftwo#1#2{#1}
\long\def\pgfutil@secondoftwo#1#2{#2}
\def\pgfutil@empty{}
\def\pgfutil@gobble@until@relax#1\relax{}
\def\pgfutil@gobble#1{}
\def\pgfutil@gobbletwo#1#2{}
\def\pgfutil@namedef#1{\expandafter\def\csname #1\endcsname}
\def\pgfutil@namelet#1{\expandafter\pgfutil@@namelet\csname#1\endcsname}
\def\pgfutil@@namelet#1#2{\expandafter\let\expandafter#1\csname#2\endcsname}
\long\def\pgfutil@g@addto@macro#1#2{%
  \begingroup
    \pgfutil@toks@\expandafter{#1#2}%
    \xdef#1{\the\pgfutil@toks@}%
 \endgroup}
\newif\ifpgfutil@tempswa
\newif\ifpgfutil@tempswb

\long\def\pgfutil@ifnextchar#1#2#3{%
  \let\pgfutil@reserved@d=#1%
  \def\pgfutil@reserved@a{#2}%
  \def\pgfutil@reserved@b{#3}%
  \futurelet\pgfutil@let@token\pgfutil@ifnch}
\def\pgfutil@ifnch{%
  \ifx\pgfutil@let@token\pgfutil@sptoken
    \let\pgfutil@reserved@c\pgfutil@xifnch
  \else
    \ifx\pgfutil@let@token\pgfutil@reserved@d
      \let\pgfutil@reserved@c\pgfutil@reserved@a
    \else
      \let\pgfutil@reserved@c\pgfutil@reserved@b
    \fi
  \fi
  \pgfutil@reserved@c}
{%
  \def\:{\global\let\pgfutil@sptoken= } \:
  \def\:{\pgfutil@xifnch} \expandafter\gdef\: {\futurelet\pgfutil@let@token\pgfutil@ifnch}
}

\newif\ifpgfutil@in@

\def\pgfutil@in@#1#2{%
 \def\pgfutil@in@@##1#1##2##3\pgfutil@in@@{%
  \ifx\pgfutil@in@##2\pgfutil@in@false\else\pgfutil@in@true\fi}%
 \pgfutil@in@@#2#1\pgfutil@in@\pgfutil@in@@}

\def\pgfutil@nnil{\pgfutil@nil}
\def\pgfutil@fornoop#1\@@#2#3{}
\long\def\pgfutil@for#1:=#2\do#3{%
  \expandafter\def\expandafter\pgfutil@fortmp\expandafter{#2}%
  \ifx\pgfutil@fortmp\pgfutil@empty \else
    \expandafter\pgfutil@forloop#2,\pgfutil@nil,\pgfutil@nil\@@#1{#3}\fi}
\long\def\pgfutil@forloop#1,#2,#3\@@#4#5{\def#4{#1}\ifx #4\pgfutil@nnil \else
       #5\def#4{#2}\ifx #4\pgfutil@nnil \else#5\pgfutil@iforloop #3\@@#4{#5}\fi\fi}
\long\def\pgfutil@iforloop#1,#2\@@#3#4{\def#3{#1}\ifx #3\pgfutil@nnil
       \expandafter\pgfutil@fornoop \else
      #4\relax\expandafter\pgfutil@iforloop\fi#2\@@#3{#4}}
\def\pgfutil@tfor#1:={\pgfutil@tf@r#1 }
\long\def\pgfutil@tf@r#1#2\do#3{\def\pgfutil@fortmp{#2}\ifx\pgfutil@fortmp\pgfutil@space\else
    \pgfutil@tforloop#2\pgfutil@nil\pgfutil@nil\@@#1{#3}\fi}
\long\def\pgfutil@tforloop#1#2\@@#3#4{\def#3{#1}\ifx #3\pgfutil@nnil
       \expandafter\pgfutil@fornoop \else
      #4\relax\expandafter\pgfutil@tforloop\fi#2\@@#3{#4}}

\def\pgfutil@space{ }

\def\pgfutil@IfFileExists#1#2#3{%
  \openin\pgfutil@inputcheck=#1 %
  \ifeof\pgfutil@inputcheck
     #3\relax
  \else
    #2\relax
  \fi
  \closein\pgfutil@inputcheck}

\def\pgfutil@InputIfFileExists#1#2#3{\pgfutil@IfFileExists{#1}{\input #1\relax#2}{#3}}%

\def\pgfutil@loop#1\pgfutil@repeat{\def\pgfutil@body{#1}\pgfutil@iterate}
\def\pgfutil@iterate{\pgfutil@body \let\pgfutil@next\pgfutil@iterate \else\let\pgfutil@next\relax\fi \pgfutil@next}
\let\pgfutil@repeat=\fi 

\let\pgfutil@aux@read@hook=\relax

\newtoks\pgfutil@everybye

\def\pgfutil@addpdfresource@extgs#1{\pgf@sys@addpdfresource@extgs@plain{#1}}
\def\pgfutil@addpdfresource@colorspaces#1{\pgf@sys@addpdfresource@colorspaces@plain{#1}}
\def\pgfutil@addpdfresource@patterns#1{\pgf@sys@addpdfresource@patterns@plain{#1}}
\def\pgfutil@setuppdfresources{\pgf@sys@setuppdfresources@plain}

\let\pgfutil@insertatbegincurrentpage=\relax

\def\pgfutil@raggedright{\rightskip\z@ plus2em \spaceskip.3333em \xspaceskip.5em\relax}
\def\pgfutil@raggedleft{\leftskip\z@ plus2em \rightskip\z@ \spaceskip.3333em \xspaceskip.5em\parfillskip0pt\relax}

\def\pgfutil@packageerror#1#2#3{\errhelp{#3}\errmessage{Package #1 Error: #2}}
\def\pgfutil@packagewarning#1#2{\immediate\write17{Package #1: Warning! #2.}}
\def\pgferror#1{\pgfutil@packageerror{pgf}{#1}{}}
\def\pgfwarning#1{\pgfutil@packagewarning{pgf}{#1}}

\def\usepgflibrary{\pgfutil@ifnextchar[{\use@pgflibrary}{\use@@pgflibrary}}
\def\use@pgflibrary[#1]{\use@@pgflibrary{#1}}
\def\use@@pgflibrary#1{%
  \edef\pgf@list{#1}%
  \pgfutil@for\pgf@temp:=\pgf@list\do{%
    \expandafter\pgfkeys@spdef\expandafter\pgf@temp\expandafter{\pgf@temp}%
    \ifx\pgf@temp\pgfutil@empty
    \else
      \expandafter\ifx\csname pgf@library@\pgf@temp @loaded\endcsname\relax%
      \expandafter\let\csname pgf@library@\pgf@temp @loaded\endcsname=\pgfutil@empty%
      \expandafter\edef\csname pgf@library@#1@atcode\endcsname{\the\catcode`\@}
      \expandafter\edef\csname pgf@library@#1@barcode\endcsname{\the\catcode`\|}
      \expandafter\edef\csname pgf@library@#1@dollarcode\endcsname{\the\catcode`\$}
      \catcode`\@=11
      \catcode`\|=12
      \catcode`\$=3
      \pgfutil@InputIfFileExists{pgflibrary\pgf@temp.code.tex}{}{%
          \pgferror{I did not find the pgf library
            '\pgf@temp'. I looked for the file named
            pgflibrary\pgf@temp.code.tex, but could not find it in in
            the current texmf trees.} 
        }%
      \catcode`\@=\csname pgf@library@#1@atcode\endcsname
      \catcode`\|=\csname pgf@library@#1@barcode\endcsname
      \catcode`\$=\csname pgf@library@#1@dollarcode\endcsname
      \fi%
    \fi
  }%
}

\def\usepgfmodule{\pgfutil@ifnextchar[{\use@pgfmodule}{\use@@pgfmodule}}
\def\use@pgfmodule[#1]{\use@@pgfmodule{#1}}
\def\use@@pgfmodule#1{%
  \edef\pgf@list{#1}%
  \pgfutil@for\pgf@temp:=\pgf@list\do{%
    \expandafter\ifx\csname pgf@module@\pgf@temp @loaded\endcsname\relax%
      \expandafter\let\csname pgf@module@\pgf@temp @loaded\endcsname=\pgfutil@empty%
      \expandafter\edef\csname pgf@module@#1@atcode\endcsname{\the\catcode`\@}
      \expandafter\edef\csname pgf@module@#1@barcode\endcsname{\the\catcode`\|}
      \expandafter\edef\csname pgf@module@#1@dollarcode\endcsname{\the\catcode`\$}
      \catcode`\@=11
      \catcode`\|=12
      \catcode`\$=3
      \input pgfmodule\pgf@temp.code.tex
      \catcode`\@=\csname pgf@module@#1@atcode\endcsname
      \catcode`\|=\csname pgf@module@#1@barcode\endcsname
      \catcode`\$=\csname pgf@module@#1@dollarcode\endcsname
    \fi%
  }%
}

\def\pgfutilensuremath#1{%
  \ifmmode#1\else$#1$\fi
}

\begingroup
	\catcode`\"=12
	\edef\pgf@loc@TMPa{"}%
	\gdef\pgfutilpreparefilename#1{%
		\begingroup
			\ifnum\the\catcode`\"=13
				\pgfutilconvertdcolon
			\fi
			\xdef\pgf@temp{#1}%
		\endgroup
		\expandafter\pgfutil@in@\expandafter"\expandafter{\pgf@temp}%
		\ifpgfutil@in@
			\def\pgf@loc@TMPa{\pgfutilstrreplace{"}{}}%
			\expandafter\pgf@loc@TMPa\expandafter{\pgf@temp}%
			\edef\pgfretval{"\pgfretval"}
			\let\pgfretvalquoted=\pgfretval
		\else
			\let\pgfretval=\pgf@temp
			\edef\pgfretvalquoted{"\pgfretval"}%
		\fi
	}%
	\catcode`\"=13
	\xdef\pgfutilconvertdcolon{%
		\noexpand\def\noexpand"{\pgf@loc@TMPa}%
	}%
\endgroup

\long\def\pgfutilstrreplace#1#2#3{%
	\def\pgfretval{}%
	\long\def\pgfutil@search@and@replace@@##1#1##2\pgf@EOI{%
		\expandafter\def\expandafter\pgfretval\expandafter{\pgfretval ##1#2}%
		\pgfutil@search@and@replace@loop{#1}{##2}%
	}%
	\pgfutil@search@and@replace@loop{#1}{#3}%
}
\long\def\pgfutil@search@and@replace@loop#1#2{%
	\pgfutil@in@{#1}{#2}%
	\ifpgfutil@in@
		\def\pgf@loc@TMPa{\pgfutil@search@and@replace@@ #2\pgf@EOI}%
	\else
		\expandafter\def\expandafter\pgfretval\expandafter{\pgfretval #2}%
		\let\pgf@loc@TMPa=\relax
	\fi
	\pgf@loc@TMPa
}%

\def\pgfutilsolvetwotwoleq#1#2{%
	\begingroup
		\dimendef\aa=0
		\dimendef\ab=1
		\dimendef\ba=2
		\dimendef\bb=3
		\dimendef\ra=4
		\dimendef\rb=5
		\dimendef\tmpa=6
		\dimendef\tmpb=7
		\edef\pgf@temp{#1}%
		\expandafter\pgfutilsolvetwotwoleq@A\pgf@temp
		\edef\pgf@temp{#2}%
		\expandafter\pgfutilsolvetwotwoleq@r\pgf@temp
		\pgfutilsolvetwotwoleq@ifislarger\aa\ba{%
			\def\Pa{a}%
			\def\Pb{b}%
		}{%
			\def\Pa{b}%
			\def\Pb{a}%
		}%
		\pgfmathreciprocal@
			{\csname m\Pa a\endcsname}%
		\let\pivot=\pgfmathresult
		\csname \Pb a\endcsname=\pivot\csname \Pb a\endcsname
		\edef\factor{\expandafter\pgf@sys@tonumber\csname \Pb a\endcsname}%
		\tmpa=-\factor\csname \Pa b\endcsname
		\advance\csname \Pb b\endcsname by\tmpa
		\tmpa=-\factor\csname r\Pa\endcsname
		\advance\csname r\Pb\endcsname by\tmpa
		\pgfmathdivide@
			{\expandafter\pgf@sys@tonumber\csname r\Pb\endcsname}
			{\expandafter\pgf@sys@tonumber\csname \Pb b\endcsname}%
		\expandafter\let\csname pgfmathresult\Pb\endcsname=\pgfmathresult
		\tmpa=\csname pgfmathresult\Pb\endcsname\csname \Pa b\endcsname
		\advance\csname r\Pa\endcsname by-\tmpa
		\tmpa=\pivot\csname r\Pa\endcsname
		\expandafter\edef\csname pgfmathresult\Pa\endcsname{\pgf@sys@tonumber\tmpa}%
		\edef\pgfmathresult{%
			{\csname pgfmathresult\Pa\endcsname}%
			{\csname pgfmathresult\Pb\endcsname}%
		}%
		\pgfmath@smuggleone\pgfmathresult
	\endgroup
}%
\def\pgfutilsolvetwotwoleq@ifislarger#1#2#3#4{%
	\tmpa=#1
	\ifdim\tmpa<0pt
		\multiply\tmpa by-1
	\fi
	\tmpb=#2
	\ifdim\tmpb<0pt
		\multiply\tmpb by-1
	\fi
	\ifdim\tmpa>\tmpb
		#3%
	\else
		#4%
	\fi
}%
\def\pgfutilsolvetwotwoleq@A#1#2#3#4{%
	\def\maa{#1}\def\mab{#2}%
	\def\mba{#3}\def\mbb{#3}%
	\aa=#1pt \ab=#2pt
	\ba=#3pt \bb=#4pt
}
\def\pgfutilsolvetwotwoleq@r#1#2{%
	\ra=#1pt \rb=#2pt
}%

\let\pgfutil@write=\write
\let\pgfutil@read=\read

\let\pgfutil@protect\relax

\def\pgfutil@check@rerun#1#2{}

\newdimen\pgfutil@tempdima
\newdimen\pgfutil@tempdimb

\let\pgfutil@ifluatex\iffalse
\begingroup\expandafter\expandafter\expandafter\endgroup
\expandafter\ifx\csname directlua\endcsname\relax
\else
  \expandafter\let\csname pgfutil@ifluatex\expandafter\endcsname
    \csname iftrue\endcsname
\fi

\pgfutil@ifluatex
  \let\pgfutil@directlua\directlua
  \pgfutil@directlua{%
    tex.enableprimitives('pgfutil@',{'luaescapestring'})}
\else
  \def\pgfutil@directlua#1{}
  \def\pgfutil@luaescapestring#1{}
\fi

\def\pgfutil@advancestringcounter#1{%
	\begingroup
		\c@pgf@counta=#1\relax
		\advance\c@pgf@counta by1
		\edef#1{\the\c@pgf@counta}%
		\pgfmath@smuggleone#1%
	\endgroup
}%
\def\pgfapplistnewempty#1{%
	\expandafter\let\csname pgfapp@#1\endcsname=\pgfutil@empty
	\expandafter\let\csname pgfapp@#1@smallbuf\endcsname=\pgfutil@empty
	\expandafter\let\csname pgfapp@#1@bigbuf\endcsname=\pgfutil@empty
	\expandafter\def\csname pgfapp@#1@smallbuf@c\endcsname{0}%
	\expandafter\def\csname pgfapp@#1@bigbuf@c\endcsname{0}%
}%

\long\def\pgfapplistpushback#1\to#2{%
	\begingroup
		\c@pgf@counta=\csname pgfapp@#2@smallbuf@c\endcsname\relax
		\advance\c@pgf@counta by1
		\xdef\pgf@glob@TMPa{\the\c@pgf@counta}%
	\endgroup
	\expandafter\let\csname pgfapp@#2@smallbuf@c\endcsname=\pgf@glob@TMPa
	\ifnum\csname pgfapp@#2@smallbuf@c\endcsname<40
		\t@pgf@toka=\expandafter\expandafter\expandafter{\csname pgfapp@#2@smallbuf\endcsname#1}%
		\expandafter\edef\csname pgfapp@#2@smallbuf\endcsname{\the\t@pgf@toka}%
	\else
		\pgfapplistpushback@smallbufoverfl{#1}{#2}%
	\fi
}%
\long\def\pgfapplistpushback@smallbufoverfl#1#2{%
	\begingroup
		\c@pgf@counta=\csname pgfapp@#2@bigbuf@c\endcsname\relax
		\advance\c@pgf@counta by1
		\xdef\pgf@glob@TMPa{\the\c@pgf@counta}%
	\endgroup
	\expandafter\let\csname pgfapp@#2@bigbuf@c\endcsname=\pgf@glob@TMPa
	\ifnum\csname pgfapp@#2@bigbuf@c\endcsname<30
		\t@pgf@toka=\expandafter\expandafter\expandafter{\csname pgfapp@#2@bigbuf\endcsname}%
		\t@pgf@tokb=\expandafter\expandafter\expandafter{\csname pgfapp@#2@smallbuf\endcsname#1}%
		\expandafter\edef\csname pgfapp@#2@bigbuf\endcsname{\the\t@pgf@toka\the\t@pgf@tokb}%
		\expandafter\let\csname pgfapp@#2@smallbuf\endcsname=\pgfutil@empty
		\expandafter\def\csname pgfapp@#2@smallbuf@c\endcsname{0}%
	\else%
		\t@pgf@toka=\expandafter\expandafter\expandafter{\csname pgfapp@#2\endcsname}%
		\t@pgf@tokb=\expandafter\expandafter\expandafter{\csname pgfapp@#2@bigbuf\endcsname}%
		\t@pgf@tokc=\expandafter\expandafter\expandafter{\csname pgfapp@#2@smallbuf\endcsname#1}%
		\expandafter\edef\csname pgfapp@#2\endcsname{\the\t@pgf@toka\the\t@pgf@tokb\the\t@pgf@tokc}%
		\expandafter\let\csname pgfapp@#2@smallbuf\endcsname=\pgfutil@empty
		\expandafter\def\csname pgfapp@#2@smallbuf@c\endcsname{0}%
		\expandafter\let\csname pgfapp@#2@bigbuf\endcsname=\pgfutil@empty
		\expandafter\def\csname pgfapp@#2@bigbuf@c\endcsname{0}%
	\fi%
}%
\def\pgfapplist@flushbuffers#1{%
	\t@pgf@toka=\expandafter\expandafter\expandafter{\csname pgfapp@#1\endcsname}%
	\t@pgf@tokb=\expandafter\expandafter\expandafter{\csname pgfapp@#1@bigbuf\endcsname}%
	\t@pgf@tokc=\expandafter\expandafter\expandafter{\csname pgfapp@#1@smallbuf\endcsname}%
	\expandafter\edef\csname pgfapp@#1\endcsname{\the\t@pgf@toka\the\t@pgf@tokb\the\t@pgf@tokc}%
	\expandafter\let\csname pgfapp@#1@smallbuf\endcsname=\pgfutil@empty
	\expandafter\def\csname pgfapp@#1@smallbuf@c\endcsname{0}%
	\expandafter\let\csname pgfapp@#1@bigbuf\endcsname=\pgfutil@empty
	\expandafter\def\csname pgfapp@#1@bigbuf@c\endcsname{0}%
}%

\def\pgfapplistlet#1=#2{%
	\pgfapplist@flushbuffers{#2}%
	\expandafter\let\expandafter#1\csname pgfapp@#2\endcsname
}%

\def\pgfprependlistnewempty#1{%
	\expandafter\let\csname pgfpPRP@#1\endcsname=\pgfutil@empty
	\expandafter\let\csname pgfpPRP@#1@smallbuf\endcsname=\pgfutil@empty
	\expandafter\let\csname pgfpPRP@#1@bigbuf\endcsname=\pgfutil@empty
	\expandafter\def\csname pgfpPRP@#1@smallbuf@c\endcsname{0}%
	\expandafter\def\csname pgfpPRP@#1@bigbuf@c\endcsname{0}%
}%

\long\def\pgfprependlistpushfront#1\to#2{%
	\begingroup
		\c@pgf@counta=\csname pgfpPRP@#2@smallbuf@c\endcsname\relax
		\advance\c@pgf@counta by1
		\xdef\pgf@glob@TMPa{\the\c@pgf@counta}%
	\endgroup
	\expandafter\let\csname pgfpPRP@#2@smallbuf@c\endcsname=\pgf@glob@TMPa
	\ifnum\csname pgfpPRP@#2@smallbuf@c\endcsname<40
		\t@pgf@toka=\expandafter\expandafter\expandafter{\csname pgfpPRP@#2@smallbuf\endcsname}%
		\t@pgf@tokb={#1}%
		\expandafter\edef\csname pgfpPRP@#2@smallbuf\endcsname{\the\t@pgf@tokb\the\t@pgf@toka}%
	\else
		\pgfprependlistpushfront@smallbufoverfl{#1}{#2}%
	\fi
}%
\long\def\pgfprependlistpushfront@smallbufoverfl#1#2{%
	\begingroup
		\c@pgf@counta=\csname pgfpPRP@#2@bigbuf@c\endcsname\relax
		\advance\c@pgf@counta by1
		\xdef\pgf@glob@TMPa{\the\c@pgf@counta}%
	\endgroup
	\expandafter\let\csname pgfpPRP@#2@bigbuf@c\endcsname=\pgf@glob@TMPa
	\ifnum\csname pgfpPRP@#2@bigbuf@c\endcsname<30
		\t@pgf@toka=\expandafter\expandafter\expandafter{\csname pgfpPRP@#2@bigbuf\endcsname}%
		\t@pgf@tokb=\expandafter\expandafter\expandafter{\csname pgfpPRP@#2@smallbuf\endcsname}%
		\t@pgf@tokc={#1}%
		\expandafter\edef\csname pgfpPRP@#2@bigbuf\endcsname{\the\t@pgf@tokc\the\t@pgf@tokb\the\t@pgf@toka}%
		\expandafter\let\csname pgfpPRP@#2@smallbuf\endcsname=\pgfutil@empty
		\expandafter\def\csname pgfpPRP@#2@smallbuf@c\endcsname{0}%
	\else%
		\pgfprependlist@flushbuffers{#2}%
		\t@pgf@toka=\expandafter\expandafter\expandafter{\csname pgfpPRP@#2\endcsname}%
		\t@pgf@tokb={#1}%
		\expandafter\edef\csname pgfpPRP@#2\endcsname{\the\t@pgf@tokb\the\t@pgf@toka}%
	\fi%
}%
\def\pgfprependlist@flushbuffers#1{%
	\t@pgf@toka=\expandafter\expandafter\expandafter{\csname pgfpPRP@#1\endcsname}%
	\t@pgf@tokb=\expandafter\expandafter\expandafter{\csname pgfpPRP@#1@bigbuf\endcsname}%
	\t@pgf@tokc=\expandafter\expandafter\expandafter{\csname pgfpPRP@#1@smallbuf\endcsname}%
	\expandafter\edef\csname pgfpPRP@#1\endcsname{\the\t@pgf@tokc\the\t@pgf@tokb\the\t@pgf@toka}%
	\expandafter\let\csname pgfpPRP@#1@smallbuf\endcsname=\pgfutil@empty
	\expandafter\def\csname pgfpPRP@#1@smallbuf@c\endcsname{0}%
	\expandafter\let\csname pgfpPRP@#1@bigbuf\endcsname=\pgfutil@empty
	\expandafter\def\csname pgfpPRP@#1@bigbuf@c\endcsname{0}%
}%
\def\pgfprependlistlet#1=#2{%
	\pgfprependlist@flushbuffers{#2}%
	\expandafter\let\expandafter#1\csname pgfpPRP@#2\endcsname
}%

\input pgfutil-plain.def

\def\pgfversion{3.0.1a}

\begingroup
\catcode`\"=12
\pgfutil@IfUndefined{directlua}{}{%
	\directlua{pgf = {}; pgf.pgfversion = "\pgfversion"}%
}%
\endgroup

\ifx\pgfrcsloaded\undefined
\def\pgfrcsloaded{}

\edef\pgfrcsatcode{\the\catcode`\@}
\catcode`\@=11

\def\pgf@parseid $#1: #2.#3,v #4 #5/#6/#7 #8${%
  \pgf@parsercsfile$#2/$
  \def\pgf@rcssuffix{#3}
  \def\pgf@rcsrevision{#4}
  \def\pgf@rcsdate{#5/#6/#7}
}
\def\pgf@parsercsfile$#1/#2${
  \def\pgf@temp{#2}
  \ifx\pgf@temp\pgfutil@empty
    \def\pgf@rcsfile{#1}
  \else
    \pgf@parsercsfile$#2$
  \fi}

\def\ProvidesFileRCS{%
  \def\pgfrcs@marshal{\ProvidesFile{\pgf@rcsfile.\pgf@rcssuffix}}%
  \pgfrcs@parserest%
}

\def\ProvidesPackageRCS{%
  \def\pgfrcs@marshal{\ProvidesPackage{\pgf@rcsfile}}%
  \pgfrcs@parserest%
}

\def\pgfrcs@parserest{%
  \def\pgf@rcsadditional{}%
  \afterassignment\pgfrcs@checkforoptional\let\next=}

\def\pgfrcs@checkforoptional{%
  \ifx\next[%
    \let\next=\pgfrcs@getoptional%
  \else
    \let\next=\pgfrcs@package%
  \fi%
  \next}

\def\pgfrcs@getoptional#1] ${%
  \def\pgf@rcsadditional{#1}%
  \pgfrcs@package%
}

\def\pgfrcs@package#1${
  \pgf@parseid $#1$
  \pgfrcs@marshal[\pgf@rcsdate\space\pgf@rcsadditional\space(rcs-revision \pgf@rcsrevision)]
}

\def\ProvidesClassRCS $#1$ [#2]{%
  \pgf@parseid $#1$
  \ProvidesClass{\pgf@rcsfile}[\pgf@rcsdate\space#2\space (rcs-revision \pgf@rcsrevision)]
}

\ifx
  \def\ProvidesPackage#1[#2]{\wlog{Loading package #1 version #2.}}
\fi

\ifx
  \def\ProvidesFile#1[#2]{\wlog{Loading file #1 version #2.}}
\fi

\ProvidesPackageRCS[v\pgfversion] $Header: /cvsroot/pgf/pgf/generic/pgf/utilities/pgfrcs.code.tex,v 1.31 2015/08/07 10:17:34 cfeuersaenger Exp $

\catcode`\@=\pgfrcsatcode

\fi

\catcode`\@=\pgfrcsatcode

\fi

\ifx\pgfkeysloaded\undefined
  \let\pgfkeysloaded=\relax
\else
  \expandafter 
\fi
  
\newif\ifpgfkeys@csname@test%

\def\pgfkeys@ifcsname#1\endcsname#2\else#3\fi{%
  \pgfkeys@csname@testfalse%
  \expandafter\ifx\csname#1\endcsname\relax\else\expandafter\ifx\csname#1\endcsname\@undefined\else\pgfkeys@csname@testtrue\fi\fi%
  \ifpgfkeys@csname@test#2\else#3\fi%
}%

\ifx\eTeXrevision\undefined%
\else%
  \expandafter\let\expandafter\pgfkeys@ifcsname\csname ifcsname\endcsname%
\fi

\def\pgfkeys@empty{}

\ifx\PackageError\undefined
  \def\pgfkeys@error#1{\errmessage{Package pgfkeys Error: #1.}}%
\else
  \def\pgfkeys@error#1{\PackageError{pgfkeys}{#1}{}}%
\fi

\long\def\pgfkeyssetvalue#1#2{%
  \pgfkeys@temptoks{#2}\expandafter\edef\csname pgfk@#1\endcsname{\the\pgfkeys@temptoks}%
}

\long\def\pgfkeysaddvalue#1#2#3{%
  {%
    \toks0{#2}%
    \pgfkeysifdefined{#1}
    {\pgfkeys@temptoks\expandafter\expandafter\expandafter{\csname pgfk@#1\endcsname}}%
    {\pgfkeys@temptoks{}}%
    \toks1{#3}%
    \xdef\pgfkeys@global@temp{\the\toks0 \the\pgfkeys@temptoks \the\toks1}
  }%
  \pgfkeyslet{#1}\pgfkeys@global@temp%
}

\def\pgfkeyslet#1#2{%
  \expandafter\let\csname pgfk@#1\endcsname#2%
}

\def\pgfkeysgetvalue#1#2{\expandafter\let\expandafter#2\csname pgfk@#1\endcsname}

\def\pgfkeysvalueof#1{\csname pgfk@#1\endcsname}

\long\def\pgfkeysifdefined#1#2#3{\pgfkeys@ifcsname pgfk@#1\endcsname#2\else#3\fi}

\newtoks\pgfkeys@pathtoks
\def\pgfkeyscurrentpath{\the\pgfkeys@pathtoks}
\newtoks\pgfkeys@temptoks

\def\pgfkeys@root{/}
\let\pgfkeysdefaultpath\pgfkeys@root

\def\pgfkeys{\expandafter\pgfkeys@@set\expandafter{\pgfkeysdefaultpath}}%
\long\def\pgfkeys@@set#1#2{%
  \let\pgfkeysdefaultpath\pgfkeys@root%
  \pgfkeys@parse#2,\pgfkeys@mainstop%
  \def\pgfkeysdefaultpath{#1}}

\def\pgfkeys@parse{\futurelet\pgfkeys@possiblerelax\pgfkeys@parse@main}
\def\pgfkeys@parse@main{%
  \ifx\pgfkeys@possiblerelax\pgfkeys@mainstop%
    \expandafter\pgfkeys@cleanup%
  \else%
    \expandafter\pgfkeys@normal%
  \fi%
}
\newif\ifpgfkeys@syntax@handlers
\def\pgfkeys@normal{%
  \ifpgfkeys@syntax@handlers%
    \expandafter\pgfkeys@syntax@handlers%
  \else%
    \expandafter\pgfkeys@@normal%
  \fi%
}
\def\pgfkeys@syntax@handlers{\pgf@keys@utilifnextchar\relax\pgfkeys@syntax@@handlers\pgfkeys@syntax@@handlers}
\def\pgfkeys@syntax@@handlers{\futurelet\pgfkeys@first@char\pgfkeys@syntax@handlers@test}
\def\pgfkeys@syntax@handlers@test{%
  \pgfkeysgetvalue{/handlers/first char syntax/\meaning\pgfkeys@first@char}\pgfkeys@the@handler%
  \ifx\pgfkeys@the@handler\relax%
    \expandafter\pgfkeys@@normal%
  \else%
    \expandafter\pgfkeys@use@handler%
  \fi%
}
\long\def\pgfkeys@use@handler#1,{%
  \pgfkeys@the@handler{#1}%
  \pgfkeys@parse%
}

\long\def\pgfkeys@@normal#1,{%
  \pgfkeys@unpack#1=\pgfkeysnovalue=\pgfkeys@stop%
  \pgfkeys@parse%
}
\def\pgfkeys@cleanup\pgfkeys@mainstop{}

\def\pgfkeys@mainstop{\pgfkeys@mainstop} 
\def\pgfkeys@novalue{} 
\def\pgfkeysnovalue{\pgfkeys@novalue} 
\def\pgfkeysnovalue@text{\pgfkeysnovalue}
\def\pgfkeysvaluerequired{\pgfkeysvaluerequired} 

\long\def\pgfkeys@unpack#1=#2=#3\pgfkeys@stop{%
  \pgfkeys@spdef\pgfkeyscurrentkey{#1}%
  \edef\pgfkeyscurrentkey{\pgfkeyscurrentkey}%
  \ifx\pgfkeyscurrentkey\pgfkeys@empty%
  \else%
    \pgfkeys@add@path@as@needed%
    \pgfkeys@spdef\pgfkeyscurrentvalue{#2}%
    \ifx\pgfkeyscurrentvalue\pgfkeysnovalue@text
      \pgfkeysifdefined{\pgfkeyscurrentkey/.@def}%
      {\pgfkeysgetvalue{\pgfkeyscurrentkey/.@def}{\pgfkeyscurrentvalue}}
      {}
    \fi%
    \ifx\pgfkeyscurrentvalue\pgfkeysvaluerequired%
      \def\pgf@marshal{\pgfkeysvalueof{/errors/value required/.@cmd}}%
      \expandafter\pgf@marshal\expandafter{\pgfkeyscurrentkey}{}\pgfeov%
    \else%
      \pgfkeys@case@one%
    \fi%
  \fi}

\def\pgfkeys@case@one{%
  \pgfkeysifdefined{\pgfkeyscurrentkey/.@cmd}%
  {\pgfkeysgetvalue{\pgfkeyscurrentkey/.@cmd}{\pgfkeys@code}%
   \expandafter\pgfkeys@code\pgfkeyscurrentvalue\pgfeov}
  {\pgfkeys@case@two}%
}

\def\pgfkeys@case@two{%
  \pgfkeysifdefined{\pgfkeyscurrentkey}%
  {\pgfkeys@case@two@extern}%
  {\pgfkeys@case@three}%
}

\def\pgfkeys@case@two@extern{%
  \ifx\pgfkeyscurrentvalue\pgfkeysnovalue@text%
    \pgfkeysvalueof{\pgfkeyscurrentkey}%
  \else%
    \pgfkeyslet{\pgfkeyscurrentkey}\pgfkeyscurrentvalue%
  \fi%
}

\def\pgfkeys@case@three{%
  \pgfkeys@split@path%
  \pgfkeysifdefined{/handlers/\pgfkeyscurrentname/.@cmd}%
  {\pgfkeysgetvalue{/handlers/\pgfkeyscurrentname/.@cmd}{\pgfkeys@code}%
    \expandafter\pgfkeys@code\pgfkeyscurrentvalue\pgfeov}
  {\pgfkeys@unknown}%
}
\let\pgfkeys@case@three@handleall=\pgfkeys@case@three
\def\pgfkeys@case@three@handle@restricted{%
  \pgfkeys@split@path%
  \pgfkeysifdefined{/handlers/\pgfkeyscurrentname/.@cmd}{%
    \pgfkeys@ifexecutehandler{%
    \pgfkeysgetvalue{/handlers/\pgfkeyscurrentname/.@cmd}{\pgfkeys@code}%
      \expandafter\pgfkeys@code\pgfkeyscurrentvalue\pgfeov
  }{%
    \let\pgfkeys@temp=\pgfkeyscurrentkey
    \let\pgfkeys@tempb=\pgfkeyscurrentname
      \edef\pgfkeyscurrentkey{\pgfkeyscurrentpath}%
      \pgfkeys@split@path%
    \let\pgfkeyscurrentkey=\pgfkeys@temp
    \edef\pgfkeyscurrentname{\pgfkeyscurrentname/\pgfkeys@tempb}%
      \pgfkeys@unknown
  }%
  }{%
    \pgfkeys@unknown
  }%
}

\def\pgfkeys@ifexecutehandler#1#2{#1}%
\let\pgfkeys@ifexecutehandler@handleall=\pgfkeys@ifexecutehandler
\def\pgfkeys@ifexecutehandler@handleonlyexisting#1#2{%
  \pgfkeys@ifcsname pgfk@excpt@\pgfkeyscurrentname\endcsname%
     #1
  \else
     \pgfkeysifdefined{\pgfkeyscurrentpath}{#1}{%
    \pgfkeysifdefined{\pgfkeyscurrentpath/.@cmd}{#1}{#2}%
     }{}%
  \fi%
}%
\def\pgfkeys@ifexecutehandler@handlefullorexisting#1#2{%
  \ifpgfkeysaddeddefaultpath
    \pgfkeys@ifcsname pgfk@excpt@\pgfkeyscurrentname\endcsname%
       #1
    \else
       \pgfkeysifdefined{\pgfkeyscurrentpath}{%
         #1%
       }{%
         \pgfkeysifdefined{\pgfkeyscurrentpath/.@cmd}{%
           #1%
         }{%
           #2%
         }%
       }%
    \fi%
  \else
    #1
  \fi
}%
\def\pgfkeysaddhandleonlyexistingexception#1{\expandafter\def\csname pgfk@excpt@#1\endcsname{1}}%

\def\pgfkeys@unknown{%
  \pgfkeysifdefined{\pgfkeyscurrentpath/.unknown/.@cmd}%
  {%
    \pgfkeysgetvalue{\pgfkeyscurrentpath/.unknown/.@cmd}{\pgfkeys@code}%
    \expandafter\pgfkeys@code\pgfkeyscurrentvalue\pgfeov}
  {%
    \pgfkeysgetvalue{/handlers/.unknown/.@cmd}{\pgfkeys@code}%
    \expandafter\pgfkeys@code\pgfkeyscurrentvalue\pgfeov%
  }%
}

\long\def\pgfkey@argumentisspace#1{%
  \long\def\pgfkeys@spdef##1##2{%
    \futurelet\pgfkeys@possiblespace\pgfkeys@sp@a##2\pgfkeys@stop\pgfkeys@stop#1\pgfkeys@stop\relax##1}%
  \def\pgfkeys@sp@a{%
    \ifx\pgfkeys@possiblespace\pgfkeys@sptoken%
      \expandafter\pgfkeys@sp@b%
    \else%
      \expandafter\pgfkeys@sp@b\expandafter#1%
    \fi}%
  \long\def\pgfkeys@sp@b#1##1 \pgfkeys@stop{\pgfkeys@sp@c##1}%
}
\pgfkey@argumentisspace{ }
\long\def\pgfkeys@sp@c#1\pgfkeys@stop#2\relax#3{\pgfkeys@temptoks{#1}\edef#3{\the\pgfkeys@temptoks}}
{\def\:{\global\let\pgfkeys@sptoken= } \: }

\def\pgfkeys@add@path@as@needed{
  \expandafter\futurelet\expandafter\pgfkeys@possibleslash\expandafter\pgfkeys@check@slash\pgfkeyscurrentkey\relax%
}
\newif\ifpgfkeysaddeddefaultpath
\def\pgfkeys@check@slash{%
  \ifx\pgfkeys@possibleslash/%
    \expandafter\pgfkeys@nevermind%
  \else%
    \expandafter\pgfkeys@addpath%
  \fi%
}

\def\pgfkeys@nevermind#1\relax{%
  \pgfkeysaddeddefaultpathfalse
  \let\pgfkeyscurrentkeyRAW\pgfkeyscurrentkey
}
\def\pgfkeys@addpath#1\relax{%
  \pgfkeysaddeddefaultpathtrue
  \def\pgfkeyscurrentkeyRAW{#1}%
  \edef\pgfkeyscurrentkey{\pgfkeysdefaultpath#1}%
}

\def\pgfkeys@split@path{
  \pgfkeys@pathtoks{}%
  \expandafter\pgfkeys@splitter\pgfkeyscurrentkey//%
}
\def\pgfkeys@splitter#1/#2/{%
  \def\pgfkeys@temp{#2}%
  \ifx\pgfkeys@temp\pgfkeys@empty%
    \def\pgfkeyscurrentname{#1}%
    \expandafter\pgfkeys@gobbletoslash%
  \else%
    \expandafter\pgfkeys@pathtoks\expandafter{\the\pgfkeys@pathtoks#1/}%
  \fi%
  \pgfkeys@splitter#2/%
}
\def\pgfkeys@gobbletoslash\pgfkeys@splitter/{\expandafter\pgfkeys@remove@slash\the\pgfkeys@pathtoks\relax}%
\def\pgfkeys@remove@slash#1/\relax{\pgfkeys@pathtoks{#1}}

\def\pgfqkeys{\expandafter\pgfkeys@@qset\expandafter{\pgfkeysdefaultpath}}%
\long\def\pgfkeys@@qset#1#2#3{\def\pgfkeysdefaultpath{#2/}\pgfkeys@parse#3,\pgfkeys@mainstop\def\pgfkeysdefaultpath{#1}}

\long\def\pgfkeysalso#1{\pgfkeys@parse#1,\pgfkeys@mainstop}

\long\def\pgfqkeysalso#1#2{\def\pgfkeysdefaultpath{#1/}\pgfkeys@parse#2,\pgfkeys@mainstop}

\long\def\pgfkeysdef#1#2{%
  \long\def\pgfkeys@temp##1\pgfeov{#2}%
  \pgfkeyslet{#1/.@cmd}{\pgfkeys@temp}%
  \pgfkeyssetvalue{#1/.@body}{#2}%
}
\long\def\pgfkeysedef#1#2{%
  \long\edef\pgfkeys@temp##1\pgfeov{#2}%
  \pgfkeyslet{#1/.@cmd}{\pgfkeys@temp}%
  \pgfkeyssetvalue{#1/.@body}{#2}%
}

\long\def\pgfkeysdefargs#1#2#3{%
  \long\def\pgfkeys@temp#2\pgfeov{#3}%
  \pgfkeyslet{#1/.@cmd}{\pgfkeys@temp}%
  \pgfkeyssetvalue{#1/.@args}{#2\pgfeov}%
  \pgfkeyssetvalue{#1/.@body}{#3}%
}
\long\def\pgfkeysedefargs#1#2#3{%
  \long\edef\pgfkeys@temp#2\pgfeov{#3}%
  \pgfkeyslet{#1/.@cmd}{\pgfkeys@temp}%
  \pgfkeyssetvalue{#1/.@args}{#2\pgfeov}%
  \pgfkeyssetvalue{#1/.@body}{#3}%
}

\long\def\pgfkeysdefnargs#1#2#3{\pgfkeysdefnargs@{#1}{#2}{#3}{\def}}%
\long\def\pgfkeysedefnargs#1#2#3{\pgfkeysdefnargs@{#1}{#2}{#3}{\edef}}%
\long\def\pgfkeysdefnargs@#1#2#3#4{%
  \ifcase#2\relax
    \pgfkeyssetvalue{#1/.@args}{}%
  \or
    \pgfkeyssetvalue{#1/.@args}{##1}%
  \or
    \pgfkeyssetvalue{#1/.@args}{##1##2}%
  \or
    \pgfkeyssetvalue{#1/.@args}{##1##2##3}%
  \or
    \pgfkeyssetvalue{#1/.@args}{##1##2##3##4}%
  \or
    \pgfkeyssetvalue{#1/.@args}{##1##2##3##4##5}%
  \or
    \pgfkeyssetvalue{#1/.@args}{##1##2##3##4##5##6}%
  \or
    \pgfkeyssetvalue{#1/.@args}{##1##2##3##4##5##6##7}%
  \or
    \pgfkeyssetvalue{#1/.@args}{##1##2##3##4##5##6##7##8}%
  \or
    \pgfkeyssetvalue{#1/.@args}{##1##2##3##4##5##6##7##8##9}%
  \else
    \pgfkeys@error{\string\pgfkeysdefnargs: expected  <= 9 arguments, got #2}%
  \fi
  \pgfkeysgetvalue{#1/.@args}\pgfkeys@tempargs
  \def\pgfkeys@temp{\expandafter\long\expandafter#4\csname pgfk@#1/.@@body\endcsname}%
  \expandafter\pgfkeys@temp\pgfkeys@tempargs{#3}%
  \edef\pgfkeys@tempargs{\noexpand\pgfkeysvalueof{#1/.@@body}}%
  \def\pgfkeys@temp{\pgfkeysdef{#1}}%
  \expandafter\pgfkeys@temp\expandafter{\pgfkeys@tempargs##1}%
  \pgfkeyssetvalue{#1/.@body}{#3}%
}

\pgfkeysdef{/handlers/.code}{\pgfkeysdef{\pgfkeyscurrentpath}{#1}}
\pgfkeysdef{/handlers/.code 2 args}{\pgfkeysdefargs{\pgfkeyscurrentpath}{##1##2}{#1}}
\pgfkeysdef{/handlers/.ecode}{\pgfkeysedef{\pgfkeyscurrentpath}{#1}}
\pgfkeysdef{/handlers/.ecode 2 args}{\pgfkeysedefargs{\pgfkeyscurrentpath}{##1##2}{#1}}
\pgfkeysdefnargs{/handlers/.code args}{2}{\pgfkeysdefargs{\pgfkeyscurrentpath}{#1}{#2}}
\pgfkeysdefnargs{/handlers/.ecode args}{2}{\pgfkeysedefargs{\pgfkeyscurrentpath}{#1}{#2}}
\pgfkeysdefnargs{/handlers/.code n args}{2}{\pgfkeysdefnargs{\pgfkeyscurrentpath}{#1}{#2}}
\pgfkeysdefnargs{/handlers/.ecode n args}{2}{\pgfkeysedefnargs{\pgfkeyscurrentpath}{#1}{#2}}

\pgfkeys{/handlers/.add code/.code 2 args=%
  \pgfkeysifdefined{\pgfkeyscurrentpath/.@args}%
  {
  \pgfkeysaddvalue{\pgfkeyscurrentpath/.@body}{#1}{#2}%
  {%
    \pgfkeysgetvalue{\pgfkeyscurrentpath/.@args}{\pgfkeys@tempargs}%
    \pgfkeysgetvalue{\pgfkeyscurrentpath/.@body}{\pgfkeys@tempbody}%
    \def\pgfkeys@marshal{\expandafter\gdef\expandafter\pgfkeys@global@temp\pgfkeys@tempargs}%
    \expandafter\pgfkeys@marshal\expandafter{\pgfkeys@tempbody}%
  }%
    \pgfkeysifdefined{\pgfkeyscurrentpath/.@@body}{%
    \pgfkeyslet{\pgfkeyscurrentpath/.@@body}{\pgfkeys@global@temp}%
  }{%
      \pgfkeyslet{\pgfkeyscurrentpath/.@cmd}{\pgfkeys@global@temp}%
  }%
  }%
  {%
    \edef\pgf@expanded@path{\pgfkeyscurrentpath}%
    {%
      \toks0{#1}%
      \pgfkeysifdefined{\pgf@expanded@path/.@cmd}%
      {\expandafter\expandafter\expandafter\pgfkeys@temptoks\expandafter\expandafter\expandafter{\csname pgfk@\pgf@expanded@path/.@body\endcsname}}%
      {\expandafter\pgfkeys@temptoks\expandafter{\expandafter\pgfkeyssetvalue\expandafter{\pgf@expanded@path}{##1}}}%
      \toks1{#2}%
      \xdef\pgfkeys@global@temp{\the\toks0 \the\pgfkeys@temptoks \the\toks1 }%
    }%
    \def\pgf@temp{\pgfkeyssetvalue{\pgf@expanded@path/.@body}}%
    \expandafter\pgf@temp\expandafter{\pgfkeys@global@temp}%
    \expandafter\long\expandafter\def\expandafter\pgfkeys@temp\expandafter##\expandafter1\expandafter\pgfeov\expandafter{\pgfkeys@global@temp}%
    \pgfkeyslet{\pgf@expanded@path/.@cmd}\pgfkeys@temp%
  }%
}
\pgfkeys{/handlers/.prefix code/.code=\pgfkeys{\pgfkeyscurrentpath/.add code={#1}{}}}%
\pgfkeys{/handlers/.append code/.code=\pgfkeys{\pgfkeyscurrentpath/.add code={}{#1}}}%

\pgfkeys{/handlers/.style/.code=\pgfkeys{\pgfkeyscurrentpath/.code=\pgfkeysalso{#1}}}
\pgfkeys{/handlers/.estyle/.code=\pgfkeys{\pgfkeyscurrentpath/.ecode=\noexpand\pgfkeysalso{#1}}}
\pgfkeys{/handlers/.style args/.code 2 args=\pgfkeys{\pgfkeyscurrentpath/.code args={#1}{\pgfkeysalso{#2}}}}
\pgfkeys{/handlers/.estyle args/.code 2 args=\pgfkeys{\pgfkeyscurrentpath/.ecode args={#1}{\noexpand\pgfkeysalso{#2}}}}
\pgfkeys{/handlers/.style 2 args/.code=\pgfkeys{\pgfkeyscurrentpath/.code 2 args=\pgfkeysalso{#1}}}
\pgfkeys{/handlers/.style n args/.code 2 args=\pgfkeys{\pgfkeyscurrentpath/.code n args={#1}{\pgfkeysalso{#2}}}}

\pgfkeys{/handlers/.add style/.code 2 args=\pgfkeys{\pgfkeyscurrentpath/.add code={\pgfkeysalso{#1}}{\pgfkeysalso{#2}}}}%
\pgfkeys{/handlers/.prefix style/.code=\pgfkeys{\pgfkeyscurrentpath/.add code={\pgfkeysalso{#1}}{}}}%
\pgfkeys{/handlers/.append style/.code=\pgfkeys{\pgfkeyscurrentpath/.add code={}{\pgfkeysalso{#1}}}}%

\pgfkeys{/handlers/.initial/.code=\pgfkeyssetvalue{\pgfkeyscurrentpath}{#1}}
\pgfkeys{/handlers/.add/.code 2 args=\pgfkeysaddvalue{\pgfkeyscurrentpath}{#1}{#2}}
\pgfkeys{/handlers/.prefix/.code=\pgfkeysaddvalue{\pgfkeyscurrentpath}{#1}{}}
\pgfkeys{/handlers/.append/.code=\pgfkeysaddvalue{\pgfkeyscurrentpath}{}{#1}}
\pgfkeys{/handlers/.get/.code=\pgfkeysgetvalue{\pgfkeyscurrentpath}{#1}}
\pgfkeys{/handlers/.link/.code=\pgfkeyssetvalue{\pgfkeyscurrentpath}{\pgfkeysvalueof{#1}}}

\pgfkeys{/handlers/.default/.code=\pgfkeyssetvalue{\pgfkeyscurrentpath/.@def}{#1}}
\pgfkeys{/handlers/.value required/.code=\pgfkeyssetvalue{\pgfkeyscurrentpath/.@def}{\pgfkeysvaluerequired}}
\pgfkeys{/handlers/.value forbidden/.code=\pgfkeys{\pgfkeyscurrentpath/.add code=%
{%
  \ifx\pgfkeyscurrentvalue\pgfkeysnovalue@text%
  \else%
    \def\pgf@marshal{\pgfkeysvalueof{/errors/value forbidden/.@cmd}}%
    \expandafter\expandafter\expandafter\pgf@marshal\expandafter\expandafter\expandafter{\expandafter\pgfkeyscurrentkey\expandafter}\expandafter{\pgfkeyscurrentvalue}\pgfeov%
  \fi%
}{}}}

\pgfkeys{/handlers/.store in/.code=\pgfkeysalso{\pgfkeyscurrentpath/.code=\def#1{##1}}}
\pgfkeys{/handlers/.estore in/.code=\pgfkeysalso{\pgfkeyscurrentpath/.code=\edef#1{##1}}}

\pgfkeys{/handlers/.is if/.code=\pgfkeysalso{%
    \pgfkeyscurrentpath/.code=\pgfkeys@handle@boolean{#1}{##1},
    \pgfkeyscurrentpath/.default=true%
  }%
}
\def\pgfkeys@handle@boolean#1#2{%
  \pgfkeys@ifcsname#1#2\endcsname%
    \csname#1#2\endcsname%
  \else%
    \def\pgf@marshal{\pgfkeysvalueof{/errors/boolean expected/.@cmd}}%
    \expandafter\pgf@marshal\expandafter{\pgfkeyscurrentkey}{#2}\pgfeov%
  \fi
}

\pgfkeys{/handlers/.is choice/.code=%
  \pgfkeys{%
    \pgfkeyscurrentpath/.cd,%
    .code=\def\pgfkeys@was@choice{##1}\expandafter\pgfkeysalso\expandafter{\pgfkeyscurrentkey/##1},
    .unknown/.code={%
      \def\pgf@marshal{\pgfkeysvalueof{/errors/unknown choice value/.@cmd}}%
      {\expandafter\expandafter\expandafter\pgf@marshal\expandafter\expandafter\expandafter{\expandafter\the\expandafter\pgfkeys@pathtoks\expandafter}\expandafter{\pgfkeys@was@choice}\pgfeov}%
    }%
  }%
}

\pgfkeys{/handlers/.list/.code=%
  {%
    \ifx\foreach\@undefined%
      \pgfkeys@error{You need to load the pgffor package to use the .list key syntax.}%
    \fi%
    \def\pgf@keys@temp{}%
    \foreach \pgf@keys@key in{#1}%
    {\expandafter\expandafter\expandafter\gdef%
      \expandafter\expandafter\expandafter\pgf@keys@temp%
      \expandafter\expandafter\expandafter{\expandafter\pgf@keys@temp\expandafter{\pgf@keys@key}}}%
    \edef\pgf@keys@list@path{\pgfkeyscurrentpath}%
    \expandafter\expandafter\expandafter\pgf@keys@do@list%
    \expandafter\expandafter\expandafter{\expandafter\pgf@keys@list@path\expandafter}\pgf@keys@temp\pgf@stop%
  }%
}
\def\pgf@keys@do@list#1{\pgf@keys@utilifnextchar\bgroup{\pgf@keys@do@list@item{#1}}\pgf@keys@gobble}
\def\pgf@keys@do@list@item#1#2{\pgfkeysalso{#1={#2}}\pgf@keys@do@list{#1}}

\pgfkeys{/handlers/.forward to/.code=%
  \pgfkeys{\pgfkeyscurrentpath/.add code={}{\pgfkeys{#1={##1}}}}
}

\pgfkeys{/handlers/.show value/.code=\pgfkeysgetvalue{\pgfkeyscurrentpath}{\pgfkeysshower}\show\pgfkeysshower} 
\pgfkeys{/handlers/.show code/.code=\pgfkeysgetvalue{\pgfkeyscurrentpath/.@cmd}{\pgfkeysshower}\show\pgfkeysshower} 

\pgfkeys{/handlers/first char syntax/.is if=pgfkeys@syntax@handlers}

\def\pgfkeys@searchalso@prepare@unknown@handler#1{%
  \global\def\pgfkeys@global@temp##1\pgfeov{}%
  \pgfkeys@searchalso@parse#1,\pgfkeys@mainstop
  {%
    \toks0=\expandafter{\pgfkeys@global@temp##1\pgfeov}%
  \toks1={\pgfkeysalso{/handlers/.unknown/.@cmd/.expand once=\pgfkeys@searchalso@temp@value}}%
  \xdef\pgfkeys@global@temp{%
    \noexpand\def\noexpand\pgfkeys@searchalso@temp@value{####1}%
    \noexpand\ifpgfkeysaddeddefaultpath
      \noexpand\pgfkeyssuccessfalse
      \noexpand\let\noexpand\pgfkeys@searchalso@name=\noexpand\pgfkeyscurrentkeyRAW
      \the\toks0 
    \noexpand\else
      \the\toks1 
    \noexpand\fi
  }%
  \expandafter\gdef\expandafter\pgfkeys@global@temp\expandafter##\expandafter1\expandafter\pgfeov\expandafter{\pgfkeys@global@temp}%
  }%
}%

\def\pgfkeys@searchalso@parse{\futurelet\pgfkeys@possiblerelax\pgfkeys@searchalso@parse@main}
\def\pgfkeys@searchalso@parse@main{%
  \ifx\pgfkeys@possiblerelax\pgfkeys@mainstop%
    \expandafter\pgfkeys@cleanup%
  \else%
    \expandafter\pgfkeys@searchalso@appendentry%
  \fi%
}
\def\pgfkeys@searchalso@appendentry#1,#2{%
  \def\pgfkeys@searchalso@nexttok{#2}%
  \pgfkeys@spdef\pgfkeys@temp{#1}%
  {%
    \toks0=\expandafter{\pgfkeys@global@temp##1\pgfeov}%
  \toks1=\expandafter{\pgfkeys@temp}%
  \xdef\pgfkeys@global@temp{%
    \the\toks0 
    \noexpand\ifpgfkeyssuccess\noexpand\else
      \noexpand\pgfqkeys{\the\toks1 }{\noexpand\pgfkeys@searchalso@name
               \ifx\pgfkeys@searchalso@nexttok\pgfkeys@mainstop\else/.try\fi /.expand once=\noexpand\pgfkeys@searchalso@temp@value}%
    \noexpand\fi}%
  \expandafter\gdef\expandafter\pgfkeys@global@temp\expandafter##\expandafter1\expandafter\pgfeov\expandafter{\pgfkeys@global@temp}%
  }%
  \pgfkeys@searchalso@parse#2%
}
\pgfkeys{%
  /handlers/.is family/.code=\pgfkeys{\pgfkeyscurrentpath/.ecode=\edef\noexpand\pgfkeysdefaultpath{\pgfkeyscurrentpath/}},%
  /handlers/.cd/.code=\edef\pgfkeysdefaultpath{\pgfkeyscurrentpath/},%
  /handlers/.search also/.code={%
    \pgfkeys@searchalso@prepare@unknown@handler{#1}%
    \pgfkeyslet{\pgfkeyscurrentpath/.unknown/.@cmd}{\pgfkeys@global@temp}%
  }
}%

\pgfkeys{/handlers/.expand once/.code=\expandafter\pgfkeys@exp@call\expandafter{#1}}
\pgfkeys{/handlers/.expand two once/.code 2 args=\expandafter\expandafter\expandafter\pgfkeys@exp@call\expandafter\expandafter\expandafter{\expandafter#1\expandafter}\expandafter{#2}}
\pgfkeys{/handlers/.expand twice/.code=\expandafter\expandafter\expandafter\pgfkeys@exp@call\expandafter\expandafter\expandafter{#1}}
\pgfkeys{/handlers/.expanded/.code=\edef\pgfkeys@temp{#1}\expandafter\pgfkeys@exp@call\expandafter{\pgfkeys@temp}}

\long\def\pgfkeys@exp@call#1{\pgfkeysalso{\pgfkeyscurrentpath={#1}}}

\newif\ifpgfkeyssuccess
\pgfkeys{/handlers/.try/.code=\pgfkeys@try}
\pgfkeys{/handlers/.retry/.code=\ifpgfkeyssuccess\else\pgfkeys@try\fi}
\def\pgfkeys@try{%
  \edef\pgfkeyscurrentkey{\pgfkeyscurrentpath}
  \ifx\pgfkeyscurrentvalue\pgfkeysnovalue@text
    \pgfkeysifdefined{\pgfkeyscurrentpath/.@def}%
    {\pgfkeysgetvalue{\pgfkeyscurrentpath/.@def}{\pgfkeyscurrentvalue}}
    {}
  \fi%
  \pgfkeysifdefined{\pgfkeyscurrentpath/.@cmd}%
  {%
    \pgfkeysgetvalue{\pgfkeyscurrentpath/.@cmd}{\pgfkeys@code}%
    \expandafter\pgfkeys@code\pgfkeyscurrentvalue\pgfeov%
    \pgfkeyssuccesstrue%
  }%
  {%
    \pgfkeysifdefined{\pgfkeyscurrentpath}%
    {%
      \ifx\pgfkeyscurrentvalue\pgfkeysnovalue@text%
        \pgfkeysvalueof{\pgfkeyscurrentpath}%
      \else%
        \pgfkeyslet{\pgfkeyscurrentpath}\pgfkeyscurrentvalue%
      \fi%
      \pgfkeyssuccesstrue%
    }%
    {%
      \pgfkeys@split@path%
      \pgfkeysifdefined{/handlers/\pgfkeyscurrentname/.@cmd}{%
        \pgfkeys@ifexecutehandler{%
        \pgfkeysgetvalue{/handlers/\pgfkeyscurrentname/.@cmd}{\pgfkeys@code}%
          \expandafter\pgfkeys@code\pgfkeyscurrentvalue\pgfeov
          \pgfkeyssuccesstrue%
      }{%
        \pgfkeyssuccessfalse
      }%
      }{%
      \pgfkeyssuccessfalse
      }%
  }%
  }%
}

\pgfkeys{/utils/exec/.code=#1} 

\pgfkeys{
  /errors/boolean expected/.code 2 args=%
  {
    \toks1={#1}
    \toks2={#2}
    \pgfkeys@error{%
      Boolean parameter of key '\the\toks1' must be 'true' or 'false', not
      '\the\toks2'. I am going to ignore it%
    }
  },
  /errors/value required/.code 2 args=%
  {
    \toks1={#1}
    \pgfkeys@error{%
      The key '\the\toks1' requires a value. I am going to ignore this
      key%
    }
  },
  /errors/value forbidden/.code 2 args=%
  {
    \toks1={#1}
    \toks2={#2}
    \pgfkeys@error{%
      You may not specify a value for the key '\the\toks1'. I am going to ignore
      the value '\the\toks2' that you provided%
    }
  },
  /errors/unknown choice value/.code 2 args=%
  {
    \toks1={#1}
    \toks2={#2}
    \pgfkeys@error{%
      Choice '\the\toks2' unknown in choice key '\the\toks1'. I am
      going to ignore this key%
    }
  },
  /errors/unknown key/.code 2 args=%
  {
    \toks1={#1}
    \toks2={#2}
    \def\pgf@temp{#2}%
    \pgfkeys@error{%
      I do not know the key '\the\toks1'\ifx\pgf@temp\pgfkeysnovalue@text\space\else, to which you passed
      '\the\toks2', \fi and I am going to ignore it. Perhaps you 
      misspelled it%
    }
  }
}

\pgfkeys{/handlers/.unknown/.code=%
  {%
    \def\pgf@marshal{\pgfkeysvalueof{/errors/unknown key/.@cmd}}%
    {\expandafter\expandafter\expandafter\pgf@marshal\expandafter\expandafter\expandafter{\expandafter\pgfkeyscurrentkey\expandafter}\expandafter{\pgfkeyscurrentvalue}\pgfeov}%
  }%
}

\pgfkeys{
  /handler config/.is choice,
  /handler config/all/.code={%
    \let\pgfkeys@case@three=\pgfkeys@case@three@handleall
    \let\pgfkeys@ifexecutehandler=\pgfkeys@ifexecutehandler@handleall
  },
  /handler config/only existing/.code={%
    \let\pgfkeys@case@three=\pgfkeys@case@three@handle@restricted
    \let\pgfkeys@ifexecutehandler=\pgfkeys@ifexecutehandler@handleonlyexisting
  },
  /handler config/full or existing/.code={%
    \let\pgfkeys@case@three=\pgfkeys@case@three@handle@restricted
    \let\pgfkeys@ifexecutehandler=\pgfkeys@ifexecutehandler@handlefullorexisting
  },
  /handler config/only existing/add exception/.code={\pgfkeysaddhandleonlyexistingexception{#1}},
}%
\pgfkeysaddhandleonlyexistingexception{.cd}%
\pgfkeysaddhandleonlyexistingexception{.try}%
\pgfkeysaddhandleonlyexistingexception{.retry}%
\pgfkeysaddhandleonlyexistingexception{.lastretry}%
\pgfkeysaddhandleonlyexistingexception{.unknown}%
\pgfkeysaddhandleonlyexistingexception{.expand once}%
\pgfkeysaddhandleonlyexistingexception{.expand two once}%
\pgfkeysaddhandleonlyexistingexception{.expand twice}%
\pgfkeysaddhandleonlyexistingexception{.expanded}%

\def\pgf@keys@gobble#1{}
\long\def\pgf@keys@utilifnextchar#1#2#3{%
  \let\pgf@keys@utilreserved@d=#1%
  \def\pgf@keys@utilreserved@a{#2}%
  \def\pgf@keys@utilreserved@b{#3}%
  \futurelet\pgf@keys@utillet@token\pgf@keys@utilifnch}
\def\pgf@keys@utilifnch{%
  \ifx\pgf@keys@utillet@token\pgf@keys@utilsptoken
    \let\pgf@keys@utilreserved@c\pgf@keys@utilxifnch
  \else
    \ifx\pgf@keys@utillet@token\pgf@keys@utilreserved@d
      \let\pgf@keys@utilreserved@c\pgf@keys@utilreserved@a
    \else
      \let\pgf@keys@utilreserved@c\pgf@keys@utilreserved@b
    \fi
  \fi
  \pgf@keys@utilreserved@c}
{%
  \def\:{\global\let\pgf@keys@utilsptoken= } \:
  \def\:{\pgf@keys@utilxifnch} \expandafter\gdef\: {\futurelet\pgf@keys@utillet@token\pgf@keys@utilifnch}
}

\let\pgfkeys@orig@case@one=\pgfkeys@case@one
\let\pgfkeys@orig@@set=\pgfkeys@@set
\let\pgfkeys@orig@@qset=\pgfkeys@@qset
\let\pgfkeys@orig@try=\pgfkeys@try
\let\pgfkeys@orig@unknown=\pgfkeys@unknown

\newif\ifpgfkeysfilteringisactive
\newif\ifpgfkeysfiltercontinue
\let\pgfkeys@key@predicate=\pgfkeys@empty
\let\pgfkeys@filtered@handler=\pgfkeys@empty
\newtoks\pgfkeys@tmptoks

\def\pgfkeysfiltered{%
	\expandafter\pgfkeysfiltered@@install\expandafter{\pgfkeysdefaultpath}%
}

\long\def\pgfkeysfiltered@@install#1#2{%
	\pgfkeys@install@filter@and@invoke{%
		\let\pgfkeysdefaultpath\pgfkeys@root%
		\pgfkeys@parse#2,\pgfkeys@mainstop%
		\def\pgfkeysdefaultpath{#1}%
	}%
}

\long\def\pgfkeysalsofrom#1{%
	\expandafter\pgfkeysalso\expandafter{#1}%
}

\long\def\pgfkeysalsofilteredfrom#1{%
	\expandafter\pgfkeysalsofiltered\expandafter{#1}%
}

\long\def\pgfkeysalsofiltered#1{%
	\pgfkeys@install@filter@and@invoke{\pgfkeysalso{#1}}%
}%

\long\def\pgfqkeysfiltered#1{%
	\expandafter\pgfqkeysfiltered@@install\expandafter{\pgfkeysdefaultpath}{#1}%
}

\long\def\pgfqkeysfiltered@@install#1#2#3{%
	\pgfkeys@install@filter@and@invoke{%
		\def\pgfkeysdefaultpath{#2/}\pgfkeys@parse#3,\pgfkeys@mainstop\def\pgfkeysdefaultpath{#1}%
	}%
}

\def\pgfkeysactivatefamily#1{%
	\pgfkeysiffamilydefined
		{#1}%
		{\csname pgfk@#1/familyactivetrue\endcsname}%
		{\pgfkeysvalueof{/errors/family unknown/.@cmd}{#1}\pgfeov}%
}

\def\pgfkeysdeactivatefamily#1{%
	\pgfkeysiffamilydefined
		{#1}%
		{\csname pgfk@#1/familyactivefalse\endcsname}%
		{\pgfkeysvalueof{/errors/family unknown/.@cmd}{#1}\pgfeov}%
}

\def\pgfkeysactivatefamilies#1#2{%
	\pgfkeyssavekeyfilterstateto\pgfkeys@cur@state
	\expandafter\pgfkeysactivatefamilies@impl\expandafter{\pgfkeys@cur@state}{#1}{#2}%
}
\def\pgfkeysactivatefamilies@impl#1#2#3{%
	\pgfkeysinstallkeyfilter{/pgf/key filters/false}{}%
	\let#3=\pgfkeys@empty%
	\def\pgfkeys@filtered@handler{\pgfkeys@family@activate@handler{#3}}%
	\pgfkeysalsofiltered{#2}%
	#1%
}

\def\pgfkeys@family@activate@handler#1{%
	\pgfkeysactivatefamily{\pgfkeyscurrentkey}%
	\pgfkeys@tmptoks=\expandafter\expandafter\expandafter{\expandafter#1\expandafter\pgfkeysdeactivatefamily\expandafter{\pgfkeyscurrentkey}}%
	\edef#1{\the\pgfkeys@tmptoks}%
}

\long\def\pgfkeysiffamilydefined#1#2#3{\pgfkeys@ifcsname ifpgfk@#1/familyactive\endcsname#2\else#3\fi}

\def\pgfkeysisfamilyactive#1{%
	\pgfkeysiffamilydefined{#1}{%
		\expandafter\let\expandafter\ifpgfkeysfiltercontinue\csname ifpgfk@#1/familyactive\endcsname
	}{%
		\pgfkeysvalueof{/errors/family unknown/.@cmd}{#1}\pgfeov%
		\expandafter\expandafter\expandafter\let\csname ifpgfkeysfiltercontinue\endcsname\csname iffalse\endcsname
	}%
}%

\def\pgfkeysgetfamily#1#2{%
	\pgfkeysifdefined{#1/family}{\pgfkeysgetvalue{#1/family}{#2}\pgfkeyssuccesstrue}{\pgfkeyssuccessfalse}%
}

\def\pgfkeyssetfamily#1#2{%
	\pgfkeysiffamilydefined{#2}{%
		\pgfkeyssetvalue{#1/family}{#2}%
	}{%
		\pgfkeysalso{/errors/family unknown=#2}%
	}%
}%

\long\def\pgfkeys@cur@is@descendant@of@errors{%
	\expandafter\pgfkeys@cur@is@descendant@of@errors@impl\pgfkeyscurrentkey/errors\pgf@@eov
}%
\long\def\pgfkeys@cur@is@descendant@of@errors@impl#1/errors#2\pgf@@eov{%
	\def\pgfkeyspred@TMP{#1}%
	\ifx\pgfkeyspred@TMP\pgfkeys@empty
		\pgfkeysfiltercontinuetrue
	\else
		\pgfkeysfiltercontinuefalse
	\fi
}%

% ...
\def\pgfkeysinterruptkeyfilter{%
	\ifpgfkeysfilteringisactive
		\let\pgfkeys@case@one=\pgfkeys@orig@case@one
		\let\pgfkeys@try=\pgfkeys@orig@try
		\let\pgfkeys@unknown=\pgfkeys@orig@unknown
	\fi
}

\def\endpgfkeysinterruptkeyfilter{%
	\ifpgfkeysfilteringisactive
		\let\pgfkeys@case@one=\pgfkeys@case@one@filtered
		\let\pgfkeys@try=\pgfkeys@try@filtered
		\let\pgfkeys@unknown=\pgfkeys@unknown@filtered
	\fi
}

\def\pgfkeysactivatefamiliesandfilteroptions#1#2{%
	\pgfkeysactivatefamilies{#1}{\pgfkeys@family@deactivation}%
	\pgfkeysfiltered{#2}%
	\pgfkeys@family@deactivation
}

\def\pgfqkeysactivatefamiliesandfilteroptions#1#2#3{%
	\pgfkeysactivatefamilies{#1}{\pgfkeys@family@deactivation}%
	\pgfqkeysfiltered{#2}{#3}%
	\pgfkeys@family@deactivation
}

\def\pgfkeyssplitpath{\pgfkeys@split@path}%

\def\pgfkeysinstallkeyfilter#1#2{%
	\pgfkeysifdefined{#1/.@cmd}{%
		\edef\pgfkeyscurrentkeyfilter{#1}%
		\def\pgfkeyscurrentkeyfilterargs{#2}%
		\pgfkeysgetvalue{#1/.@cmd}{\pgfkeys@key@predicate@}%
		\def\pgfkeys@key@predicate{\pgfkeys@key@predicate@#2\pgfeov}%
	}{%
		\pgfkeysvalueof{/errors/no such key filter/.@cmd}{#1}{#2}\pgfeov%
	}%
}

\def\pgfkeysinstallkeyfilterhandler#1#2{%
	\pgfkeysifdefined{#1/.@cmd}{%
		\edef\pgfkeyscurrentkeyfilterhandler{#1}%
		\def\pgfkeyscurrentkeyfilterhandlerargs{#2}%
		\pgfkeysgetvalue{#1/.@cmd}{\pgfkeys@filtered@handler@}%
		\def\pgfkeys@filtered@handler{\pgfkeys@filtered@handler@#2\pgfeov}%
	}{%
		\pgfkeysvalueof{/errors/no such key filter handler/.@cmd}{#1}{#2}\pgfeov%
	}%
}

\def\pgfkeyssavekeyfilterstateto#1{%
	\pgfkeys@tmptoks={\pgfkeysinstallkeyfilter}%
	\pgfkeys@tmptoks=\expandafter\expandafter\expandafter{\expandafter\the\expandafter\pgfkeys@tmptoks\expandafter{\pgfkeyscurrentkeyfilter}}%
	\pgfkeys@tmptoks=\expandafter\expandafter\expandafter{\expandafter\the\expandafter\pgfkeys@tmptoks\expandafter{\pgfkeyscurrentkeyfilterargs}\pgfkeysinstallkeyfilterhandler}%
	\pgfkeys@tmptoks=\expandafter\expandafter\expandafter{\expandafter\the\expandafter\pgfkeys@tmptoks\expandafter{\pgfkeyscurrentkeyfilterhandler}}%
	\pgfkeys@tmptoks=\expandafter\expandafter\expandafter{\expandafter\the\expandafter\pgfkeys@tmptoks\expandafter{\pgfkeyscurrentkeyfilterhandlerargs}}%
	\edef#1{%
		\the\pgfkeys@tmptoks
	}%
}

\pgfkeys{%
	/errors/family unknown/.code=\pgfkeys@error{%
		Sorry, I do not know family '#1' and can't work with any assoicated family handling. Perhaps you misspelled it?},
	/errors/no such key filter/.code 2 args=\pgfkeys@error{Sorry, there is no such key filter '#1'.},
	/errors/no such key filter handler/.code 2 args=\pgfkeys@error{Sorry, there is no such key filter handler '#1'.},
	/handlers/.is family/.append code={%
		\edef\pgfkeyspred@TMP{pgfk@\pgfkeyscurrentpath/familyactive}%
		\expandafter\pgfkeys@non@outer@newif\expandafter{\pgfkeyspred@TMP}%
		\edef\pgfkeyspred@TMP{\pgfkeyscurrentpath/.belongs to family=\pgfkeyscurrentpath}%
		\expandafter\pgfkeysalso\expandafter{\pgfkeyspred@TMP}%
	},%
	/handlers/.activate family/.code=\pgfkeysactivatefamily{\pgfkeyscurrentpath},
	/handlers/.deactivate family/.code=\pgfkeysdeactivatefamily{\pgfkeyscurrentpath},
	/handlers/.belongs to family/.code={\pgfkeyssetfamily{\pgfkeyscurrentpath}{#1}},%
	/handlers/.lastretry/.code={%
		\ifpgfkeyssuccess\else
			\pgfkeys@try
			\ifpgfkeyssuccess\else
				\edef\pgfkeyscurrentkey{\pgfkeyscurrentpath}%
				\pgfkeys@split@path%
				\pgfkeys@unknown
			\fi
		\fi
	},
	/handlers/.install key filter/.code={%
		\pgfkeysinstallkeyfilter{\pgfkeyscurrentpath}{#1}%
	},%
	/handlers/.install key filter handler/.code={%
		\pgfkeysinstallkeyfilterhandler{\pgfkeyscurrentpath}{#1}%
	},%
	/pgf/key filter handlers/append filtered to/.code={%
		\pgfkeys@tmptoks=\expandafter\expandafter\expandafter{\expandafter\pgfkeyscurrentkeyRAW\expandafter=\expandafter{\pgfkeyscurrentvalue}}%
		\ifx#1\pgfkeys@empty
		\else
			\pgfkeys@tmptoks=\expandafter\expandafter\expandafter{\expandafter#1\expandafter,\the\pgfkeys@tmptoks}%
		\fi
		\edef#1{\the\pgfkeys@tmptoks}%
	},%
	/pgf/key filter handlers/ignore/.code={},
	/pgf/key filter handlers/ignore/.install key filter handler,
	/pgf/key filter handlers/log/.code={%
		\immediate\write16{LOG: the option '\pgfkeyscurrentkey' (was originally '\pgfkeyscurrentkeyRAW') (case \pgfkeyscasenumber) has not been processed due to pgfkeysfiltered.}%
	},
	/pgf/key filters/active families/.code={%
		\if\pgfkeyscasenumber0%
			\pgfkeysfiltercontinuetrue
		\else
			\if\pgfkeyscasenumber3%
				\pgfkeysgetfamily\pgfkeyscurrentpath\pgfkeyspred@TMP
			\else
				\pgfkeysgetfamily\pgfkeyscurrentkey\pgfkeyspred@TMP
			\fi
			\ifpgfkeyssuccess
				\pgfkeysisfamilyactive{\pgfkeyspred@TMP}%
			\else
				\pgfkeysfiltercontinuefalse
			\fi
		\fi
	},%
	/pgf/key filters/active families or no family/.code 2 args={%
		\if\pgfkeyscasenumber0%
			\pgfkeysevalkeyfilterwith{#2}%
		\else
			\if\pgfkeyscasenumber3%
				\pgfkeysgetfamily\pgfkeyscurrentpath\pgfkeyspred@TMP
			\else
				\pgfkeysgetfamily\pgfkeyscurrentkey\pgfkeyspred@TMP
			\fi
			\ifpgfkeyssuccess
				\pgfkeysisfamilyactive{\pgfkeyspred@TMP}%
			\else
				\pgfkeysevalkeyfilterwith{#1}%
			\fi
		\fi
	},
	/pgf/key filters/active families or no family DEBUG/.code 2 args={%
		\if\pgfkeyscasenumber0%
			\immediate\write16{[pgfkeyshasactivefamilyornofamily(\pgfkeyscurrentkey, \pgfkeyscasenumber) invoking unknown handler '#2']}%
			\pgfkeysevalkeyfilterwith{#2}%
		\else
			\if\pgfkeyscasenumber3%
				\pgfkeysgetfamily\pgfkeyscurrentpath\pgfkeyspred@TMP
			\else
				\pgfkeysgetfamily\pgfkeyscurrentkey\pgfkeyspred@TMP
			\fi
			\ifpgfkeyssuccess
				\pgfkeysisfamilyactive{\pgfkeyspred@TMP}%
				\ifpgfkeysfiltercontinue
					\immediate\write16{[pgfkeyshasactivefamilyornofamily(\pgfkeyscurrentkey, \pgfkeyscasenumber) family is ACTIVE]}%
				\else
					\immediate\write16{[pgfkeyshasactivefamilyornofamily(\pgfkeyscurrentkey, \pgfkeyscasenumber) family is NOT active.]}%
				\fi
			\else
				\immediate\write16{[pgfkeyshasactivefamilyornofamily(\pgfkeyscurrentkey, \pgfkeyscasenumber) invoking has-no-family-handler '#1']}%
				\pgfkeysevalkeyfilterwith{#1}%
			\fi
		\fi
	},
	/pgf/key filters/active families and known/.code={%
		\if\pgfkeyscasenumber0%
			\pgfkeysfiltercontinuefalse
		\else
			\if\pgfkeyscasenumber3%
				\pgfkeysgetfamily\pgfkeyscurrentpath\pgfkeyspred@TMP
			\else
				\pgfkeysgetfamily\pgfkeyscurrentkey\pgfkeyspred@TMP
			\fi
			\ifpgfkeyssuccess
				\pgfkeysisfamilyactive{\pgfkeyspred@TMP}%
			\else
				\pgfkeysfiltercontinuefalse
			\fi
		\fi
	},
	/pgf/key filters/active families or descendants of/.code={%
		\if\pgfkeyscasenumber0%
			\pgfkeysfiltercontinuefalse
		\else
			\if\pgfkeyscasenumber3%
				\pgfkeysgetfamily\pgfkeyscurrentpath\pgfkeyspred@TMP
			\else
				\pgfkeysgetfamily\pgfkeyscurrentkey\pgfkeyspred@TMP
			\fi
			\ifpgfkeyssuccess
				\pgfkeysisfamilyactive{\pgfkeyspred@TMP}%
			\else
				\def\pgfkeysisdescendantof@impl##1#1##2\pgf@@eov{%
					\def\pgfkeyspred@TMP{##1}%
					\ifx\pgfkeyspred@TMP\pgfkeys@empty
						\pgfkeysfiltercontinuetrue
					\else
						\pgfkeysfiltercontinuefalse
					\fi
				}%
				\expandafter\pgfkeysisdescendantof@impl\pgfkeyscurrentkey#1\pgf@@eov
			\fi
		\fi
	},
	/pgf/key filters/is descendant of/.code={%
		\if\pgfkeyscasenumber0%
			\pgfkeysfiltercontinuetrue
		\else
			\def\pgfkeysisdescendantof@impl##1#1##2\pgf@@eov{%
				\def\pgfkeyspred@TMP{##1}%
				\ifx\pgfkeyspred@TMP\pgfkeys@empty
					\pgfkeysfiltercontinuetrue
				\else
					\pgfkeysfiltercontinuefalse
				\fi
			}%
			\expandafter\pgfkeysisdescendantof@impl\pgfkeyscurrentkey#1\pgf@@eov
		\fi
	},%
	/pgf/key filters/equals/.code={%
		\if\pgfkeyscasenumber0%
			\pgfkeysfiltercontinuetrue
		\else
			\def\pgfkeyspred@TMP{#1}%
			\ifx\pgfkeyscurrentkey\pgfkeyspred@TMP
				\pgfkeysfiltercontinuetrue
			\else
				\pgfkeysfiltercontinuefalse
			\fi
		\fi
	},%
	/pgf/key filters/not/.code={%
		\pgfkeysevalkeyfilterwith{#1}%
		\ifpgfkeysfiltercontinue
			\pgfkeysfiltercontinuefalse
		\else
			\pgfkeysfiltercontinuetrue
		\fi
	},%
	/pgf/key filters/and/.code 2 args={%
		\pgfkeysevalkeyfilterwith{#1}%
		\ifpgfkeysfiltercontinue
			\pgfkeysevalkeyfilterwith{#2}%
		\fi
	},%
	/pgf/key filters/or/.code 2 args={%
		\pgfkeysevalkeyfilterwith{#1}%
		\ifpgfkeysfiltercontinue
		\else
			\pgfkeysevalkeyfilterwith{#2}%
		\fi
	},%
	/pgf/key filters/true/.code={\pgfkeysfiltercontinuetrue},%
	/pgf/key filters/true/.install key filter,
	/pgf/key filters/false/.code={%
		\pgfkeysfiltercontinuefalse
	},%
	/pgf/key filters/defined/.code={%
		\if\pgfkeyscasenumber0%
			\pgfkeysfiltercontinuefalse
		\else
			\pgfkeysfiltercontinuetrue
		\fi
	},
}%

\def\pgfkeys@case@one@filtered{%
	\pgfkeys@cur@is@descendant@of@errors
	\ifpgfkeysfiltercontinue
		\pgfkeys@orig@case@one
	\else
		\pgfkeysfiltercontinuetrue
		\pgfkeysifdefined{\pgfkeyscurrentkey/.@cmd}{%
			\def\pgfkeyscasenumber{1}%
			\pgfkeys@key@predicate%
			\ifpgfkeysfiltercontinue
				\pgfkeysgetvalue{\pgfkeyscurrentkey/.@cmd}{\pgfkeys@code}%
				\expandafter\pgfkeys@code\pgfkeyscurrentvalue\pgfeov
			\else
				\pgfkeys@filtered@handler%
			\fi
		}{%
			\pgfkeysifdefined{\pgfkeyscurrentkey}{%
				\def\pgfkeyscasenumber{2}%
				\pgfkeys@key@predicate%
				\ifpgfkeysfiltercontinue
					\pgfkeys@case@two@extern
				\else
					\pgfkeys@filtered@handler%
				\fi
			}{%
				\pgfkeys@split@path
				\pgfkeysifdefined{/handlers/\pgfkeyscurrentname/.@cmd}{%
					\pgfkeys@ifexecutehandler{%
						\def\pgfkeyscasenumber{3}%
						\pgfkeys@key@predicate%
						\ifpgfkeysfiltercontinue
							\pgfkeysgetvalue{/handlers/\pgfkeyscurrentname/.@cmd}{\pgfkeys@code}%
							\expandafter\pgfkeys@code\pgfkeyscurrentvalue\pgfeov
						\else
							\pgfkeys@filtered@handler%
						\fi
					}{%
						\pgfkeys@unknown
					}%
				}{%
					\pgfkeys@unknown
				}%
			}%
		}%
	\fi
}%

\def\pgfkeys@unknown@filtered{%
	\def\pgfkeyscasenumber{0}%
	\pgfkeys@key@predicate%
	\ifpgfkeysfiltercontinue
		\pgfkeys@orig@unknown
	\else
		\pgfkeys@filtered@handler%
	\fi
}

\def\pgfkeys@try@filtered{%
	\ifpgfkeysfiltercontinue
		\pgfkeys@orig@try
	\else
		\pgfkeysfiltercontinuetrue
		\edef\pgfkeyscurrentkey{\pgfkeyscurrentpath}
		\ifx\pgfkeyscurrentvalue\pgfkeysnovalue@text
			\pgfkeysifdefined{\pgfkeyscurrentpath/.@def}%
			{\pgfkeysgetvalue{\pgfkeyscurrentpath/.@def}{\pgfkeyscurrentvalue}}
			{}
		\fi%
		\pgfkeysifdefined{\pgfkeyscurrentpath/.@cmd}%
		{%
			\def\pgfkeyscasenumber{1}%
			\pgfkeys@key@predicate%
			\ifpgfkeysfiltercontinue
				\pgfkeysgetvalue{\pgfkeyscurrentkey/.@cmd}{\pgfkeys@code}%
				\expandafter\pgfkeys@code\pgfkeyscurrentvalue\pgfeov%
			\else
				\pgfkeys@filtered@handler%
			\fi
			\pgfkeyssuccesstrue%
		}%
		{%
			\pgfkeysifdefined{\pgfkeyscurrentpath}%
			{
				\def\pgfkeyscasenumber{2}%
				\pgfkeys@key@predicate%
				\ifpgfkeysfiltercontinue
					\ifx\pgfkeyscurrentvalue\pgfkeysnovalue@text%
						\pgfkeysvalueof{\pgfkeyscurrentpath}%
					\else%
						\pgfkeyslet{\pgfkeyscurrentpath}\pgfkeyscurrentvalue%
					\fi%
				\else
					\pgfkeys@filtered@handler%
				\fi
				\pgfkeyssuccesstrue%
			}%
			{%
				\pgfkeys@split@path%
				\pgfkeysifdefined{/handlers/\pgfkeyscurrentname/.@cmd}{%
					\pgfkeys@ifexecutehandler{%
						\def\pgfkeyscasenumber{3}%
						\pgfkeys@key@predicate%
						\ifpgfkeysfiltercontinue
							\pgfkeysgetvalue{/handlers/\pgfkeyscurrentname/.@cmd}{\pgfkeys@code}%
							\expandafter\pgfkeys@code\pgfkeyscurrentvalue\pgfeov
						\else
							\pgfkeys@filtered@handler%
						\fi
						\pgfkeyssuccesstrue%
					}{%
						\pgfkeyssuccessfalse
					}%
				}{%
					\pgfkeyssuccessfalse
				}%
			}%
		}%
	\fi
}

\long\def\pgfkeys@install@filter@and@invoke#1{%
	\ifpgfkeysfilteringisactive
		\pgfkeys@error{Sorry, nested calls to key filtering routines are not allowed. (reason: It is not possible to properly restore the previous filtering state after returning from the nested call)}%
	\fi
	\pgfkeysfilteringisactivetrue
	\let\pgfkeys@case@one=\pgfkeys@case@one@filtered
	\let\pgfkeys@try=\pgfkeys@try@filtered
	\let\pgfkeys@unknown=\pgfkeys@unknown@filtered
	#1%
	\let\pgfkeys@case@one=\pgfkeys@orig@case@one
	\let\pgfkeys@try=\pgfkeys@orig@try
	\let\pgfkeys@unknown=\pgfkeys@orig@unknown
	\pgfkeysfilteringisactivefalse
}

\long\def\pgfkeysevalkeyfilterwith#1{%
	\pgfkeys@eval@key@filter@subroutine@unpack#1=\pgfkeysnovalue=\pgfkeys@stop
}%
\long\def\pgfkeys@eval@key@filter@subroutine@unpack#1=#2=#3\pgfkeys@stop{%
	\pgfkeys@spdef\pgfkeyspred@TMP{#1}%
	\edef\pgfkeyspred@TMP{\pgfkeyspred@TMP}%
	\pgfkeys@spdef\pgfkeyspred@TMPB{#2}
	\ifx\pgfkeyspred@TMPB\pgfkeysnovalue@text
		\pgfkeysifdefined{\pgfkeyspred@TMP/.@def}%
		{\pgfkeysgetvalue{\pgfkeyspred@TMP/.@def}{\pgfkeyspred@TMPB}}
		{}
	\fi%
	\ifx\pgfkeyspred@TMPB\pgfkeysvaluerequired%
		\pgfkeysvalueof{/errors/value required/.@cmd}\pgfkeyspred@TMP\pgfkeyspred@TMPB\pgfeov%
	\else%
		\pgfkeysifdefined{\pgfkeyspred@TMP/.@cmd}{%
			\pgfkeysgetvalue{\pgfkeyspred@TMP/.@cmd}{\pgfkeys@code}%
			\expandafter\pgfkeys@code\pgfkeyspred@TMPB\pgfeov
		}{%
			\pgfkeysvalueof{/errors/no such key filter/.@cmd}\pgfkeyspred@TMP\pgfkeyspred@TMPB\pgfeov%
		}%
	\fi%
}

\def\pgfkeys@non@outer@newif@#1#2{%
	\expandafter\edef\csname #2true\endcsname{\noexpand\let\noexpand#1=\noexpand\iftrue}%
	\expandafter\edef\csname #2false\endcsname{\noexpand\let\noexpand#1=\noexpand\iffalse}%
	\csname #2false\endcsname
}%

\def\pgfkeys@non@outer@newif#1{%
	\expandafter\pgfkeys@non@outer@newif@\csname if#1\endcsname{#1}%
}

\catcode`\@=\pgfkeysatcode

\usetikzlibrary{foobaa}
\usetikzlibrary{shapes}

\theoremstyle{definition}
\newtheorem{mydef}{Definition}[section]

\theoremstyle{plain}
\newtheorem{mythm}[mydef]{Theorem}

\newtheorem{mycor}[mydef]{Corollary}
\newtheorem{mylem}[mydef]{Lemma}

\theoremstyle{remark}

\newtheorem{myrem}[mydef]{Remark}

\tikzstyle{my node}=[circle,minimum size = 4pt, fill=black,inner sep=0pt]
\tikzstyle{empty}=[fill=none]
\tikzstyle{my path}=[-stealth, thick, shorten >=1pt, shorten <=1pt]

\newcommand{\id}{\operatorname{id}}					
\DeclareMathOperator*{\colim}{colim}					
\DeclareMathOperator{\sd}{Sd}						

\newcommand{\cat}[1]{\mathcal{#1}}					

\makeatletter
\newcommand*{\xrightrightarrows}[2]{\mathrel{
  \settowidth{\@tempdima}{$\scriptstyle#1$}
  \settowidth{\@tempdimb}{$\scriptstyle#2$}
  \ifdim\@tempdimb>\@tempdima \@tempdima=\@tempdimb\fi
  \mathop{\vcenter{
    \offinterlineskip\ialign{\hbox to\dimexpr\@tempdima+1em{##}\cr
    \rightarrowfill\cr\noalign{\kern.5ex}
    \rightarrowfill\cr}}}\limits^{\!#1}_{\!#2}}}
\makeatother

\newcommand{\thol}{\tau_1\sd^2}						

\newcommand{\trisimp}{\xymatrix@C=0pt@M=0pt@R=.5pt{&\cdot\ar@{-}[dr]&\\\cdot\ar@{-}[ur]\ar@{-}[r]&&\cdot}}

\newcommand{\Ac}{\mathbf{Ac}}						
\newcommand{\sSet}{\mathbf{sSet}}					
\newcommand{\Cat}{\mathbf{Cat}}						
\newcommand{\Pos}{\mathbf{Pos}}						

\newcommand{\ie}{\emph{i.e.\ }}
\newcommand{\eg}{\emph{e.g.\ }}

\newcommand{\tel}{[0]}
\newcommand{\iel}{\emptyset}
\newcommand{\eset}{\varnothing}

\usetikzlibrary{backgrounds}

\pgfdeclarelayer{background}
\pgfsetlayers{background,main}

\begin{document}

\title{Cofibrant Objects in the Thomason Model Structure}
\author{Roman Bruckner \and Christoph Pegel}
\address{Fachbereich Mathematik und Informatik, Universit\"at Bremen,
  Bibliothekstra\ss{}e~1, 28359~Bremen, Bremen, Germany.}
\email[Roman Bruckner]{bruckner@math.uni-bremen.de }
\email[Christoph Pegel]{pegel@math.uni-bremen.de }

\begin{abstract}
There are Quillen equivalent Thomason model structures on the category of small categories, the category of small acyclic categories and the category of posets. These share the property that cofibrant objects are posets. In fact, they share the same class of cofibrant objects. We show that every finite semilattice, every chain, every countable tree, every finite zigzag and every poset with five or less elements is cofibrant in all of those structures.
\end{abstract}

\maketitle

\pagenumbering{arabic}
\setcounter{page}{1}


\newcommand{\pictsdt}{\scalebox{2.0}{						
\begin{tikzpicture}[every node/.style=my node, every path/.style=my path]
 \coordinate (a1) at (0,0); 
 \coordinate (a2) at (7,0);
 \coordinate (a3) at (3.5,6.062);
 
\node (b1) at (barycentric cs:a1=1,a2=0,a3=0){};
\node (b2) at (barycentric cs:a1=0,a2=1,a3=0){};
\node (b3) at (barycentric cs:a1=0,a2=0,a3=1.0){};
   
\node (b123) at (barycentric cs:a1=1,a2=1,a3=1){};

\node (b12) at (barycentric cs:a1=1,a2=1){};
\node (b23) at (barycentric cs:a2=1,a3=1){};
\node (b13) at (barycentric cs:a1=1,a3=1){};

\node (b112) at (barycentric cs:b1=1,b12=1){};
\node (b212) at (barycentric cs:b2=1,b12=1){};

\node (b113) at (barycentric cs:b1=1,b13=1){};
\node (b313) at (barycentric cs:b3=1,b13=1){};

\node (b223) at (barycentric cs:b2=1,b23=1){};
\node (b323) at (barycentric cs:b3=1,b23=1){};


\node (b1123) at (barycentric cs:a1=1,b123=1){};
\node (b2123) at (barycentric cs:a2=1,b123=1){};
\node (b3123) at (barycentric cs:a3=1,b123=1){};

\node (b12123) at (barycentric cs:b12=1,b123=1){};
\node (b23123) at (barycentric cs:b23=1,b123=1){};
\node (b13123) at (barycentric cs:b13=1,b123=1){};

\node (b112123) at (barycentric cs:b1=1,b12=1,b123=1){};
\node (b223123) at (barycentric cs:b2=1,b23=1,b123=1){};
\node (b313123) at (barycentric cs:b3=1,b13=1,b123=1){};

\node (b212123) at (barycentric cs:b2=1,b12=1,b123=1){};
\node (b323123) at (barycentric cs:b3=1,b23=1,b123=1){};
\node (b113123) at (barycentric cs:b1=1,b13=1,b123=1){};

\draw (b1)--(b112);
\draw (b1)--(b113);
\draw (b2)--(b212);
\draw (b2)--(b223);
\draw (b3)--(b313);
\draw (b3)--(b323);
\draw (b1)--(b1123);
\draw (b2)--(b2123);
\draw (b3)--(b3123);

\draw (b123)--(b1123);
\draw (b123)--(b2123);
\draw (b123)--(b3123);
\draw (b123)--(b12123);
\draw (b123)--(b23123);
\draw (b123)--(b13123);

\draw (b12)--(b12123);
\draw (b23)--(b23123);
\draw (b13)--(b13123);
\draw (b12)--(b112);
\draw (b13)--(b113);
\draw (b12)--(b212);
\draw (b23)--(b223);
\draw (b13)--(b313);
\draw (b23)--(b323);

\draw (b223)--(b223123);
\draw (b2123)--(b223123);
\draw (b23123)--(b223123);

 \draw (b212)--(b212123);
 \draw (b2123)--(b212123);
 \draw (b12123)--(b212123);
%
 \draw (b112)--(b112123);
 \draw (b1123)--(b112123);
 \draw (b12123)--(b112123);

\draw (b113)--(b113123);
\draw (b1123)--(b113123);
\draw (b13123)--(b113123);

\draw (b313)--(b313123);
\draw (b3123)--(b313123);
\draw (b13123)--(b313123);

\draw (b323)--(b323123);
\draw (b3123)--(b323123);
\draw (b23123)--(b323123);
\end{tikzpicture}}}

\newcommand{\pictsdtha}{							
\begin{tikzpicture}[every node/.style=my node, every path/.style=my path]
 \coordinate (a1) at (0,0); 
 \coordinate (a2) at (7,0);
 \coordinate (a3) at (3.5,6.062);
 
\node (b1) at (barycentric cs:a1=1,a2=0,a3=0){};
\node (b2) at (barycentric cs:a1=0,a2=1,a3=0){};
   
\node (b123) at (barycentric cs:a1=1,a2=1,a3=1){};

\node (b12) at (barycentric cs:a1=1,a2=1){};
\node (b13) at (barycentric cs:a1=1,a3=1){};

\node (b112) at (barycentric cs:b1=1,b12=1){};
\node (b212) at (barycentric cs:b2=1,b12=1){};

\node (b113) at (barycentric cs:b1=1,b13=1){};


\node (b1123) at (barycentric cs:a1=1,b123=1){};
\node (b2123) at (barycentric cs:a2=1,b123=1){};

\node (b12123) at (barycentric cs:b12=1,b123=1){};
\node (b13123) at (barycentric cs:b13=1,b123=1){};

\node (b112123) at (barycentric cs:b1=1,b12=1,b123=1){};

\node (b212123) at (barycentric cs:b2=1,b12=1,b123=1){};
\node (b113123) at (barycentric cs:b1=1,b13=1,b123=1){};

\draw (b1)--(b112);
\draw (b1)--(b113);
\draw (b2)--(b212);
\draw (b1)--(b1123);
\draw (b2)--(b2123);

\draw (b123)--(b1123);
\draw (b123)--(b12123);
\draw (b123)--(b13123);
\draw (b123)--(b2123);

\draw (b12)--(b12123);
\draw (b13)--(b13123);
\draw (b12)--(b112);
\draw (b13)--(b113);
\draw (b12)--(b212);

 \draw (b212)--(b212123);
 \draw (b2123)--(b212123);
 \draw (b12123)--(b212123);
%
 \draw (b112)--(b112123);
 \draw (b1123)--(b112123);
 \draw (b12123)--(b112123);

\draw (b113)--(b113123);
\draw (b1123)--(b113123);
\draw (b13123)--(b113123);

\end{tikzpicture}}

\newcommand{\pictsdthb}{							
\begin{tikzpicture}[every node/.style=my node, every path/.style=my path]
 \coordinate (a1) at (0,0); 
 \coordinate (a2) at (7,0);
 \coordinate (a3) at (3.5,6.062);
 
\node (b1) at (barycentric cs:a1=1,a2=0,a3=0){};
\node (b2) at (barycentric cs:a1=0,a2=1,a3=0){};
   
\node (b123) at (barycentric cs:a1=1,a2=1,a3=1){};

\node (b12) at (barycentric cs:a1=1,a2=1){};
\node (b23) at (barycentric cs:a2=1,a3=1){};

\node (b112) at (barycentric cs:b1=1,b12=1){};
\node (b212) at (barycentric cs:b2=1,b12=1){};

\node (b223) at (barycentric cs:b2=1,b23=1){};


\node (b1123) at (barycentric cs:a1=1,b123=1){};
\node (b2123) at (barycentric cs:a2=1,b123=1){};

\node (b12123) at (barycentric cs:b12=1,b123=1){};
\node (b23123) at (barycentric cs:b23=1,b123=1){};

\node (b112123) at (barycentric cs:b1=1,b12=1,b123=1){};
\node (b223123) at (barycentric cs:b2=1,b23=1,b123=1){};

\node (b212123) at (barycentric cs:b2=1,b12=1,b123=1){};

\draw (b1)--(b112);

\draw (b2)--(b212);
\draw (b2)--(b223);

\draw (b1)--(b1123);
\draw (b2)--(b2123);

\draw (b123)--(b1123);
\draw (b123)--(b2123);

\draw (b123)--(b12123);
\draw (b123)--(b23123);

\draw (b12)--(b12123);
\draw (b23)--(b23123);

\draw (b12)--(b112);

\draw (b12)--(b212);
\draw (b23)--(b223);

\draw (b223)--(b223123);
\draw (b2123)--(b223123);
\draw (b23123)--(b223123);

 \draw (b212)--(b212123);
 \draw (b2123)--(b212123);
 \draw (b12123)--(b212123);
%
 \draw (b112)--(b112123);
 \draw (b1123)--(b112123);
 \draw (b12123)--(b112123);



\end{tikzpicture}}

\newcommand{\pictsdttt}{							
\begin{tikzpicture}[every node/.style=my node, every path/.style=my path]
 \coordinate (a1) at (0,0); 
 \coordinate (a2) at (7,0);
 \coordinate (a3) at (3.5,6.062);
 
\node (b1) at (barycentric cs:a1=1,a2=0,a3=0){};
   
\node (b123) at (barycentric cs:a1=1,a2=1,a3=1){};

\node (b12) at (barycentric cs:a1=1,a2=1){};
\node (b13) at (barycentric cs:a1=1,a3=1){};

\node (b112) at (barycentric cs:b1=1,b12=1){};
\node (b113) at (barycentric cs:b1=1,b13=1){};


\node (b1123) at (barycentric cs:a1=1,b123=1){};

\node (b12123) at (barycentric cs:b12=1,b123=1){};
\node (b13123) at (barycentric cs:b13=1,b123=1){};

\node (b112123) at (barycentric cs:b1=1,b12=1,b123=1){};
\node (b113123) at (barycentric cs:b1=1,b13=1,b123=1){};

\draw (b1)--(b112);
\draw (b1)--(b113);
\draw (b1)--(b1123);

\draw (b123)--(b1123);
\draw (b123)--(b12123);
\draw (b123)--(b13123);

\draw (b12)--(b12123);
\draw (b13)--(b13123);
\draw (b12)--(b112);
\draw (b13)--(b113);

 \draw (b112)--(b112123);
 \draw (b1123)--(b112123);
 \draw (b12123)--(b112123);

\draw (b113)--(b113123);
\draw (b1123)--(b113123);
\draw (b13123)--(b113123);

\end{tikzpicture}}

\newcommand{\pictsdts}{
\begin{tikzpicture}[every node/.style=my node, every path/.style=my path]
 \coordinate (a1) at (0,0); 
 \coordinate (a2) at (7,0);
 \coordinate (a3) at (3.5,6.062);
 
\node (b1) at (barycentric cs:a1=1,a2=0,a3=0){};
   
\node (b123) at (barycentric cs:a1=1,a2=1,a3=1){};

\node (b12) at (barycentric cs:a1=1,a2=1){};

\node (b112) at (barycentric cs:b1=1,b12=1){};


\node (b12123) at (barycentric cs:b12=1,b123=1){};
\node (b1123) at (barycentric cs:a1=1,b123=1){};

\node (b112123) at (barycentric cs:b1=1,b12=1,b123=1){};


\draw (b1)--(b112);
\draw (b1)--(b1123);

\draw (b123)--(b1123);
\draw (b123)--(b12123);

\draw (b12)--(b12123);
\draw (b12)--(b112);

 \draw (b112)--(b112123);
 \draw (b1123)--(b112123);
 \draw (b12123)--(b112123);


\end{tikzpicture}}

\newcommand{\pictsdtf}{\scalebox{.6}{
\begin{tikzpicture}[every node/.style=my node, every path/.style=my path]
 \coordinate (a1) at (0,0); 
 \coordinate (a2) at (7,0);
 \coordinate (a3) at (3.5,6.062);

\node (b1) at (barycentric cs:a1=1,a2=0,a3=0){};
\node (b112) at (barycentric cs:b1=1,b12=1){};
\node (b1123) at (barycentric cs:a1=1,b123=1){};
\node (b112123) at (barycentric cs:b1=1,b12=1,b123=1){};

\draw (b1)--(b112);
\draw (b1)--(b1123);

\draw (b112)--(b112123);
\draw (b1123)--(b112123);

\end{tikzpicture}}}

\newcommand{\poinsd}{								
\begin{tikzpicture}[every node/.style=my node, every path/.style=my path]
 \coordinate (a1) at (0,0); 
 \coordinate (a2) at (7,0);
 \coordinate (a3) at (3.5,6.062);
 
\node[style=empty] (b1) at (barycentric cs:a1=1,a2=0,a3=0){$x_{2_a}$};
\node[style=empty] (b2) at (barycentric cs:a1=0,a2=1,a3=0){$x_{2_b}$};
\node (b3) at (barycentric cs:a1=0,a2=0,a3=1.0){};
   
\node[style=empty] (b123) at (barycentric cs:a1=1,a2=1,a3=1){$x_3$};

\node (b12) at (barycentric cs:a1=1,a2=1){};
\node (b23) at (barycentric cs:a2=1,a3=1){};
\node[style=empty] (b13) at (barycentric cs:a1=1,a3=1){$x_1$};

\node (b112) at (barycentric cs:b1=1,b12=1){};
\node (b212) at (barycentric cs:b2=1,b12=1){};

\node (b113) at (barycentric cs:b1=1,b13=1){};
\node (b313) at (barycentric cs:b3=1,b13=1){};

\node (b223) at (barycentric cs:b2=1,b23=1){};
\node (b323) at (barycentric cs:b3=1,b23=1){};


\node (b1123) at (barycentric cs:a1=1,b123=1){};
\node[style=empty] (b2123) at (barycentric cs:a2=1,b123=1){$y_2$};
\node (b3123) at (barycentric cs:a3=1,b123=1){};

\node (b12123) at (barycentric cs:b12=1,b123=1){};
\node (b23123) at (barycentric cs:b23=1,b123=1){};
\node (b13123) at (barycentric cs:b13=1,b123=1){};

\node (b112123) at (barycentric cs:b1=1,b12=1,b123=1){};
\node (b223123) at (barycentric cs:b2=1,b23=1,b123=1){};
\node (b313123) at (barycentric cs:b3=1,b13=1,b123=1){};

\node (b212123) at (barycentric cs:b2=1,b12=1,b123=1){};
\node (b323123) at (barycentric cs:b3=1,b23=1,b123=1){};
\node[style=empty] (b113123) at (barycentric cs:b1=1,b13=1,b123=1){$y_1$};

\draw[dashed] (b1) -- (b113123);
\draw[dashed] (b13) -- (b113123);
\draw[dashed] (b123) -- (b113123);

\draw (b1)--(b112);
\draw (b1)--(b113);
\draw (b2)--(b212);
\draw (b2)--(b223);
\draw (b3)--(b313);
\draw (b3)--(b323);
\draw (b1)--(b1123);
\draw[dashed] (b2)--(b2123);
\draw (b3)--(b3123);

\draw (b123)--(b1123);
\draw[dashed] (b123)--(b2123);
\draw (b123)--(b3123);
\draw (b123)--(b12123);
\draw (b123)--(b23123);
\draw (b123)--(b13123);

\draw (b12)--(b12123);
\draw (b23)--(b23123);
\draw (b13)--(b13123);
\draw (b12)--(b112);
\draw (b13)--(b113);
\draw (b12)--(b212);
\draw (b23)--(b223);
\draw (b13)--(b313);
\draw (b23)--(b323);

\draw (b223)--(b223123);
\draw (b2123)--(b223123);
\draw (b23123)--(b223123);

 \draw (b212)--(b212123);
 \draw (b2123)--(b212123);
 \draw (b12123)--(b212123);
%
 \draw (b112)--(b112123);
 \draw (b1123)--(b112123);
 \draw (b12123)--(b112123);

\draw (b113)--(b113123);
\draw (b1123)--(b113123);
\draw (b13123)--(b113123);

\draw (b313)--(b313123);
\draw (b3123)--(b313123);
\draw (b13123)--(b313123);

\draw (b323)--(b323123);
\draw (b3123)--(b323123);
\draw (b23123)--(b323123);
\end{tikzpicture}}

\newcommand{\poinsdf}{								
\begin{tikzpicture}[every node/.style=my node, every path/.style=my path]
 \coordinate (a1) at (0,0); 
 \coordinate (a2) at (7,0);
 \coordinate (a3) at (3.5,6.062);
 
\node[style=empty] (b1) at (barycentric cs:a1=1,a2=0,a3=0){$x_{2_a}$};
\node[style=empty] (b2) at (barycentric cs:a1=0,a2=1,a3=0){$x_{2_b}$};
   
\node[style=empty] (b123) at (barycentric cs:a1=1,a2=1,a3=1){$x_3$};

\node (b12) at (barycentric cs:a1=1,a2=1){};
\node[style=empty] (b13) at (barycentric cs:a1=1,a3=1){$x_1$};

\node (b112) at (barycentric cs:b1=1,b12=1){};
\node (b212) at (barycentric cs:b2=1,b12=1){};

\node (b113) at (barycentric cs:b1=1,b13=1){};


\node (b1123) at (barycentric cs:a1=1,b123=1){};
\node[style=empty] (b2123) at (barycentric cs:a2=1,b123=1){$y_2$};

\node (b12123) at (barycentric cs:b12=1,b123=1){};
\node (b13123) at (barycentric cs:b13=1,b123=1){};

\node (b112123) at (barycentric cs:b1=1,b12=1,b123=1){};

\node (b212123) at (barycentric cs:b2=1,b12=1,b123=1){};
\node[style=empty] (b113123) at (barycentric cs:b1=1,b13=1,b123=1){$y_1$};

\draw[dashed] (b1) -- (b113123);
\draw[dashed] (b13) -- (b113123);
\draw[dashed] (b123) -- (b113123);

\draw (b1)--(b112);
\draw (b1)--(b113);
\draw (b2)--(b212);
\draw (b1)--(b1123);
\draw[dashed] (b2)--(b2123);

\draw (b123)--(b1123);
\draw (b123)--(b12123);
\draw (b123)--(b13123);
\draw[dashed] (b123)--(b2123);

\draw (b12)--(b12123);
\draw (b13)--(b13123);
\draw (b12)--(b112);
\draw (b13)--(b113);
\draw (b12)--(b212);

 \draw (b212)--(b212123);
 \draw (b2123)--(b212123);
 \draw (b12123)--(b212123);
%
 \draw (b112)--(b112123);
 \draw (b1123)--(b112123);
 \draw (b12123)--(b112123);

\draw (b113)--(b113123);
\draw (b1123)--(b113123);
\draw (b13123)--(b113123);

\begin{pgfonlayer}{background}
	\tikzstyle{every path}=[-,gray!50,draw=gray!50,line width=15pt, line join=round, line cap=round]
	\path[fill]
	(b112.center) --
	(b112123.center) --
	(b1123.center)
	(b113123.center) -- cycle;
	\path
	(b1123.center) --
	(b113123.center) --
	(b113.center);
	\path
	(b113123.center) --
	(b13123.center);
	\path
	(b123.center) --
	(b12123.center) --
	(b12.center);
	\path
	(b212.center) --
	(b212123.center)--
	(b2123.center) -- cycle;
\end{pgfonlayer}

\end{tikzpicture}}

\newcommand{\ptoinsdfl}{
\begin{tikzpicture}[every node/.style=my node, every path/.style=my path]
 \coordinate (a1) at (0,0); 
 \coordinate (a2) at (7,0);
 \coordinate (a3) at (3.5,6.062);
 
\node (b1) at (barycentric cs:a1=1,a2=0,a3=0){};
\node[style=empty] (b2) at (barycentric cs:a1=0,a2=1,a3=0){$x_{1_a}$};
   
\node[style=empty] (b123) at (barycentric cs:a1=1,a2=1,a3=1){$x_2$};

\node (b12) at (barycentric cs:a1=1,a2=1){};
\node[style=empty] (b13) at (barycentric cs:a1=1,a3=1){$x_{1_b}$};

\node (b112) at (barycentric cs:b1=1,b12=1){};
\node (b212) at (barycentric cs:b2=1,b12=1){};

\node (b113) at (barycentric cs:b1=1,b13=1){};


\node (b1123) at (barycentric cs:a1=1,b123=1){};
\node[style=empty] (b2123) at (barycentric cs:a2=1,b123=1){$y_1$};

\node (b12123) at (barycentric cs:b12=1,b123=1){};
\node[style=empty] (b13123) at (barycentric cs:b13=1,b123=1){$y_2$};

\node (b112123) at (barycentric cs:b1=1,b12=1,b123=1){};

\node[style=empty] (b212123) at (barycentric cs:b2=1,b12=1,b123=1){$z$};
\node (b113123) at (barycentric cs:b1=1,b13=1,b123=1){};

\draw[dashed] (b123) -- (b212123);

\draw (b1)--(b112);
\draw (b1)--(b113);
\draw (b2)--(b212);
\draw (b1)--(b1123);
\draw[dashed] (b2)--(b2123);

\draw (b123)--(b1123);
\draw (b123)--(b12123);
\draw[dashed] (b123)--(b13123);
\draw (b123)--(b2123);

\draw (b12)--(b12123);
\draw[dashed] (b13)--(b13123);
\draw (b12)--(b112);
\draw (b13)--(b113);
\draw (b12)--(b212);

 \draw (b212)--(b212123);
 \draw[dashed] (b2123)--(b212123);
 \draw (b12123)--(b212123);
%
 \draw (b112)--(b112123);
 \draw (b1123)--(b112123);
 \draw (b12123)--(b112123);

\draw (b113)--(b113123);
\draw (b1123)--(b113123);
\draw (b13123)--(b113123);

\end{tikzpicture}}

\newcommand{\ptoinsd}{
\begin{tikzpicture}[every node/.style=my node, every path/.style=my path]
 \coordinate (a1) at (0,0); 
 \coordinate (a2) at (7,0);
 \coordinate (a3) at (3.5,6.062);

\node[style=empty] (b1) at (barycentric cs:a1=1,a2=0,a3=0){$x_{1_a}$};
\node[style=empty] (b2) at (barycentric cs:a1=0,a2=1,a3=0){$x_{1_b}$};
\node (b3) at (barycentric cs:a1=0,a2=0,a3=1.0){};
   
\node[style=empty] (b123) at (barycentric cs:a1=1,a2=1,a3=1){$x_2$};

\node (b12) at (barycentric cs:a1=1,a2=1){};
\node (b23) at (barycentric cs:a2=1,a3=1){};
\node (b13) at (barycentric cs:a1=1,a3=1){};

\node (b112) at (barycentric cs:b1=1,b12=1){};
\node (b212) at (barycentric cs:b2=1,b12=1){};

\node[style=empty] (b113) at (barycentric cs:b1=1,b13=1){$y_1$};
\node (b313) at (barycentric cs:b3=1,b13=1){};

\node (b223) at (barycentric cs:b2=1,b23=1){};
\node (b323) at (barycentric cs:b3=1,b23=1){};


\node (b1123) at (barycentric cs:a1=1,b123=1){};
\node (b2123) at (barycentric cs:a2=1,b123=1){};
\node (b3123) at (barycentric cs:a3=1,b123=1){};

\node (b12123) at (barycentric cs:b12=1,b123=1){};
\node (b23123) at (barycentric cs:b23=1,b123=1){};
\node (b13123) at (barycentric cs:b13=1,b123=1){};

\node (b112123) at (barycentric cs:b1=1,b12=1,b123=1){};
\node[style=empty] (b223123) at (barycentric cs:b2=1,b23=1,b123=1){$y_2$};
\node (b313123) at (barycentric cs:b3=1,b13=1,b123=1){};

\node (b212123) at (barycentric cs:b2=1,b12=1,b123=1){};
\node (b323123) at (barycentric cs:b3=1,b23=1,b123=1){};
\node[style=empty] (b113123) at (barycentric cs:b1=1,b13=1,b123=1){$z$};


\draw[dashed] (b123) -- (b113123);
\draw[dashed] (b123) -- (b223123);
\draw[dashed] (b2) -- (b223123);

\draw (b1)--(b112);
\draw[dashed] (b1)--(b113);
\draw (b2)--(b212);
\draw (b2)--(b223);
\draw (b3)--(b313);
\draw (b3)--(b323);
\draw (b1)--(b1123);
\draw (b2)--(b2123);
\draw (b3)--(b3123);

\draw (b123)--(b1123);
\draw (b123)--(b2123);
\draw (b123)--(b3123);
\draw (b123)--(b12123);
\draw (b123)--(b23123);
\draw (b123)--(b13123);

\draw (b12)--(b12123);
\draw (b23)--(b23123);
\draw (b13)--(b13123);
\draw (b12)--(b112);
\draw (b13)--(b113);
\draw (b12)--(b212);
\draw (b23)--(b223);
\draw (b13)--(b313);
\draw (b23)--(b323);

\draw (b223)--(b223123);
\draw (b2123)--(b223123);
\draw (b23123)--(b223123);

 \draw (b212)--(b212123);
 \draw (b2123)--(b212123);
 \draw (b12123)--(b212123);
%
 \draw (b112)--(b112123);
 \draw (b1123)--(b112123);
 \draw (b12123)--(b112123);

\draw[dashed] (b113)--(b113123);
\draw (b1123)--(b113123);
\draw (b13123)--(b113123);

\draw (b313)--(b313123);
\draw (b3123)--(b313123);
\draw (b13123)--(b313123);

\draw (b323)--(b323123);
\draw (b3123)--(b323123);
\draw (b23123)--(b323123);

\begin{pgfonlayer}{background}
	\tikzstyle{every path}=[-,gray!50,draw=gray!50,line width=15pt, line join=round, line cap=round]
	\path[fill]
	(b112.center) --
	(b112123.center) --
	(b1123.center) --
	(b113123.center) -- cycle;
	\path
	(b113123.center) --
	(b13123.center) --
	(b313123.center) -- 
	(b313.center);
	\path
	(b13.center) --
	(b113.center);
	\path
	(b3.center) --
	(b3123.center) --
	(b123.center) --
	(b12123.center) --
	(b12.center);
 	\path[fill]
 	(b212.center) --
 	(b212123.center) --
	(b2123.center) --
	(b223.center) --
	(b23.center) --
	(b323.center) --
	(b323123.center) --
	(b23123.center) -- 
	(b223123.center) --cycle;
\end{pgfonlayer}

\end{tikzpicture}}

\newcommand{\pttinsdf}{							
\begin{tikzpicture}[every node/.style=my node, every path/.style=my path]
 \coordinate (a1) at (0,0); 
 \coordinate (a2) at (7,0);
 \coordinate (a3) at (3.5,6.062);
 
\node[style=empty] (b1) at (barycentric cs:a1=1,a2=0,a3=0){$x_{2_a}$};
\node[style=empty] (b2) at (barycentric cs:a1=0,a2=1,a3=0){$x_{2_b}$};
   
\node[style=empty] (b123) at (barycentric cs:a1=1,a2=1,a3=1){$x_1$};

\node (b12) at (barycentric cs:a1=1,a2=1){};
\node (b13) at (barycentric cs:a1=1,a3=1){};

\node (b112) at (barycentric cs:b1=1,b12=1){};
\node (b212) at (barycentric cs:b2=1,b12=1){};

\node (b113) at (barycentric cs:b1=1,b13=1){};


\node (b1123) at (barycentric cs:a1=1,b123=1){};
\node[style=empty] (b2123) at (barycentric cs:a2=1,b123=1){$y_2$};

\node (b12123) at (barycentric cs:b12=1,b123=1){};
\node[style=empty] (b13123) at (barycentric cs:b13=1,b123=1){$y_1$};

\node (b112123) at (barycentric cs:b1=1,b12=1,b123=1){};

\node (b212123) at (barycentric cs:b2=1,b12=1,b123=1){};
\node[style=empty] (b113123) at (barycentric cs:b1=1,b13=1,b123=1){$z$};

\draw[dashed] (b1) -- (b113123);
\draw (b1)--(b112);
\draw (b1)--(b113);
\draw (b2)--(b212);
\draw (b1)--(b1123);
\draw[dashed] (b2)--(b2123);

\draw (b123)--(b1123);
\draw (b123)--(b12123);
\draw[dashed] (b123)--(b13123);
\draw[dashed] (b123)--(b2123);

\draw (b12)--(b12123);
\draw (b13)--(b13123);
\draw (b12)--(b112);
\draw (b13)--(b113);
\draw (b12)--(b212);

 \draw (b212)--(b212123);
 \draw (b2123)--(b212123);
 \draw (b12123)--(b212123);
%
 \draw (b112)--(b112123);
 \draw (b1123)--(b112123);
 \draw (b12123)--(b112123);

\draw (b113)--(b113123);
\draw (b1123)--(b113123);
\draw[dashed] (b13123)--(b113123);

\begin{pgfonlayer}{background}
	\tikzstyle{every path}=[-,gray!50,draw=gray!50,line width=15pt, line join=round, line cap=round]
	\path[fill]
	(b112.center) --
	(b112123.center) --
	(b1123.center) -- 
	(b113123.center) --
	(b113.center) -- cycle;
	\path
	(b13.center) --
	(b13123.center);
	\path
	(b123.center) --
	(b12123.center) --
	(b12.center);
	\path[fill]
	(b2123.center) --
	(b212123.center) --
	(b212.center) -- cycle;
\end{pgfonlayer}

\end{tikzpicture}}

\newcommand{\pfinsdf}{							
\begin{tikzpicture}[every node/.style=my node, every path/.style=my path]
 \coordinate (a1) at (0,0); 
 \coordinate (a2) at (7,0);
 \coordinate (a3) at (3.5,6.062);
 
\node[style=empty] (b1) at (barycentric cs:a1=1,a2=0,a3=0){$y_{2_a}$};
\node[style=empty] (b2) at (barycentric cs:a1=0,a2=1,a3=0){$y_{2_b}$};
   
\node[style=empty] (b123) at (barycentric cs:a1=1,a2=1,a3=1){$x$};

\node (b12) at (barycentric cs:a1=1,a2=1){};
\node (b23) at (barycentric cs:a2=1,a3=1){};

\node (b112) at (barycentric cs:b1=1,b12=1){};
\node (b212) at (barycentric cs:b2=1,b12=1){};

\node (b223) at (barycentric cs:b2=1,b23=1){};


\node (b1123) at (barycentric cs:a1=1,b123=1){};
\node (b2123) at (barycentric cs:a2=1,b123=1){};

\node[style=empty] (b12123) at (barycentric cs:b12=1,b123=1){$y_1$};
\node (b23123) at (barycentric cs:b23=1,b123=1){};

\node[style=empty] (b112123) at (barycentric cs:b1=1,b12=1,b123=1){$z_1$};
\node (b223123) at (barycentric cs:b2=1,b23=1,b123=1){};

\node[style=empty] (b212123) at (barycentric cs:b2=1,b12=1,b123=1){$z_2$};

\draw[dashed] (b1) -- (b112123);
\draw[dashed] (b2) -- (b212123);

\draw (b1)--(b112);

\draw (b2)--(b212);
\draw (b2)--(b223);

\draw (b1)--(b1123);
\draw (b2)--(b2123);

\draw (b123)--(b1123);
\draw (b123)--(b2123);

\draw[dashed] (b123)--(b12123);
\draw (b123)--(b23123);

\draw (b12)--(b12123);
\draw (b23)--(b23123);

\draw (b12)--(b112);

\draw (b12)--(b212);
\draw (b23)--(b223);

\draw (b223)--(b223123);
\draw (b2123)--(b223123);
\draw (b23123)--(b223123);

 \draw (b212)--(b212123);
 \draw (b2123)--(b212123);
 \draw[dashed] (b12123)--(b212123);
%
 \draw (b112)--(b112123);
 \draw (b1123)--(b112123);
 \draw[dashed] (b12123)--(b112123);

\begin{pgfonlayer}{background}
	\tikzstyle{every path}=[-,gray!50,draw=gray!50,line width=15pt, line join=round, line cap=round]
	\path[fill]
	(b112.center) --
	(b112123.center) --
	(b1123.center) -- cycle;
	\path
	(b23.center) --
	(b23123.center) --
	(b123.center);
	\path[fill]
	(b223.center) --
	(b223123.center) --
	(b2123.center) --
	(b212123.center) --
	(b212.center) -- cycle;
	\path
	(b12.center) --
	(b12123.center);
\end{pgfonlayer}

\end{tikzpicture}}

\newcommand{\psqinsd}{								
\begin{tikzpicture}[every node/.style=my node, every path/.style=my path]
 \coordinate (a1) at (0,0); 
 \coordinate (a2) at (7,0);
 \coordinate (a3) at (3.5,6.062);
 
\node[style=empty] (b1) at (barycentric cs:a1=1,a2=0,a3=0){$y_{1_a}$};
\node[style=empty] (b2) at (barycentric cs:a1=0,a2=1,a3=0){$y_{1_b}$};
\node[style=empty] (b3) at (barycentric cs:a1=0,a2=0,a3=1.0){$y_{1_c}$};
   
\node[style=empty] (b123) at (barycentric cs:a1=1,a2=1,a3=1){$x_2$};

\node (b12) at (barycentric cs:a1=1,a2=1){};
\node (b23) at (barycentric cs:a2=1,a3=1){};
\node (b13) at (barycentric cs:a1=1,a3=1){};

\node (b112) at (barycentric cs:b1=1,b12=1){};
\node (b212) at (barycentric cs:b2=1,b12=1){};

\node (b113) at (barycentric cs:b1=1,b13=1){};
\node (b313) at (barycentric cs:b3=1,b13=1){};

\node (b223) at (barycentric cs:b2=1,b23=1){};
\node (b323) at (barycentric cs:b3=1,b23=1){};


\node[style=empty] (b1123) at (barycentric cs:a1=1,b123=1){$x_1$};
\node[style=empty] (b2123) at (barycentric cs:a2=1,b123=1){$x_3$};
\node[style=empty] (b3123) at (barycentric cs:a3=1,b123=1){$x_2$};

\node (b12123) at (barycentric cs:b12=1,b123=1){};
\node (b23123) at (barycentric cs:b23=1,b123=1){};
\node (b13123) at (barycentric cs:b13=1,b123=1){};

\node (b112123) at (barycentric cs:b1=1,b12=1,b123=1){};
\node (b223123) at (barycentric cs:b2=1,b23=1,b123=1){};
\node (b313123) at (barycentric cs:b3=1,b13=1,b123=1){};

\node (b212123) at (barycentric cs:b2=1,b12=1,b123=1){};
\node (b323123) at (barycentric cs:b3=1,b23=1,b123=1){};
\node (b113123) at (barycentric cs:b1=1,b13=1,b123=1){};

\draw (b1)--(b112);
\draw (b1)--(b113);
\draw (b2)--(b212);
\draw (b2)--(b223);
\draw (b3)--(b313);
\draw (b3)--(b323);
\draw[dashed] (b1)--(b1123);
\draw[dashed] (b2)--(b2123);
\draw[dashed] (b3)--(b3123);

\draw[dashed] (b123)--(b1123);
\draw[dashed] (b123)--(b2123);
\draw[dashed] (b123)--(b3123);
\draw (b123)--(b12123);
\draw (b123)--(b23123);
\draw (b123)--(b13123);

\draw (b12)--(b12123);
\draw (b23)--(b23123);
\draw (b13)--(b13123);
\draw (b12)--(b112);
\draw (b13)--(b113);
\draw (b12)--(b212);
\draw (b23)--(b223);
\draw (b13)--(b313);
\draw (b23)--(b323);

\draw (b223)--(b223123);
\draw (b2123)--(b223123);
\draw (b23123)--(b223123);

 \draw (b212)--(b212123);
 \draw (b2123)--(b212123);
 \draw (b12123)--(b212123);
%
 \draw (b112)--(b112123);
 \draw (b1123)--(b112123);
 \draw (b12123)--(b112123);

\draw (b113)--(b113123);
\draw (b1123)--(b113123);
\draw (b13123)--(b113123);

\draw (b313)--(b313123);
\draw (b3123)--(b313123);
\draw (b13123)--(b313123);

\draw (b323)--(b323123);
\draw (b3123)--(b323123);
\draw (b23123)--(b323123);

\begin{pgfonlayer}{background}
	\tikzstyle{every path}=[-,gray!50,draw=gray!50,line width=15pt, line join=round, line cap=round]
	\path[fill]
		(b113.center) --
		(b113123.center) --
		(b1123.center) --
		(b112123.center) --
		(b112.center) -- cycle;
	\path[fill]
		(b223.center) --
		(b223123.center) --
		(b2123.center) --
		(b212123.center) --
		(b212.center) -- cycle;
	\path[fill]
		(b313.center) --
		(b313123.center) --
		(b3123.center) --
		(b323123.center) --
		(b323.center) -- cycle;
 	\path
		(b13.center) --
		(b13123.center) --
		(b123.center) --
		(b12123.center) --
		(b12.center);
	\path	
		(b23.center) --
		(b23123.center) --
		(b123.center);
\end{pgfonlayer}
\end{tikzpicture}}

\newcommand{\pninsdf}{							
\begin{tikzpicture}[every node/.style=my node, every path/.style=my path]
 \coordinate (a1) at (0,0); 
 \coordinate (a2) at (7,0);
 \coordinate (a3) at (3.5,6.062);
 
\node[style=empty] (b1) at (barycentric cs:a1=1,a2=0,a3=0){$x_{1_a}$};
\node[style=empty] (b2) at (barycentric cs:a1=0,a2=1,a3=0){$x_{1_b}$};
   
\node[style=empty] (b123) at (barycentric cs:a1=1,a2=1,a3=1){$x_2$};

\node (b12) at (barycentric cs:a1=1,a2=1){};
\node (b13) at (barycentric cs:a1=1,a3=1){};

\node (b112) at (barycentric cs:b1=1,b12=1){};
\node (b212) at (barycentric cs:b2=1,b12=1){};

\node (b113) at (barycentric cs:b1=1,b13=1){};


\node[style=empty] (b1123) at (barycentric cs:a1=1,b123=1){$y_1$};
\node[style=empty] (b2123) at (barycentric cs:a2=1,b123=1){$y_2$};

\node (b12123) at (barycentric cs:b12=1,b123=1){};
\node (b13123) at (barycentric cs:b13=1,b123=1){};

\node (b112123) at (barycentric cs:b1=1,b12=1,b123=1){};

\node (b212123) at (barycentric cs:b2=1,b12=1,b123=1){};
\node[style=empty] (b113123) at (barycentric cs:b1=1,b13=1,b123=1){$z$};

\draw (b1)--(b112);
\draw (b1)--(b113);
\draw (b2)--(b212);
\draw[dashed] (b1)--(b1123);
\draw[dashed] (b2)--(b2123);

\draw[dashed] (b123)--(b1123);
\draw (b123)--(b12123);
\draw (b123)--(b13123);
\draw[dashed] (b123)--(b2123);

\draw (b12)--(b12123);
\draw (b13)--(b13123);
\draw (b12)--(b112);
\draw (b13)--(b113);
\draw (b12)--(b212);

 \draw (b212)--(b212123);
 \draw (b2123)--(b212123);
 \draw (b12123)--(b212123);
%
 \draw (b112)--(b112123);
 \draw (b1123)--(b112123);
 \draw (b12123)--(b112123);

\draw (b113)--(b113123);
\draw[dashed] (b1123)--(b113123);
\draw (b13123)--(b113123);

\begin{pgfonlayer}{background}
	\tikzstyle{every path}=[-,gray!50,draw=gray!50,line width=15pt, line join=round, line cap=round]
	\path[fill]
		(b113123.center) --
		(b113.center) --
		(b13.center) --
		(b13123.center) -- cycle;
	\path[fill]
		(b1123.center) --
		(b112123.center) --
		(b112.center) -- cycle;
	\path[fill]
		(b123.center) --
		(b12123.center) --
		(b12.center) -- cycle;
	\path[fill]
		(b2123.center) --
		(b212123.center) --
		(b212.center) -- cycle;
\end{pgfonlayer}

\end{tikzpicture}}

\newcommand{\pictsdthf}{							
\begin{tikzpicture}[every node/.style=my node, every path/.style=my path]
 \coordinate (a1) at (0,0); 
 \coordinate (a2) at (7,0);
 \coordinate (a3) at (3.5,6.062);
 
\node (b1) at (barycentric cs:a1=1,a2=0,a3=0){};
\node (b2) at (barycentric cs:a1=0,a2=1,a3=0){};
   
\node[ellipse,minimum height=2em,minimum width=2pt, style=empty, scale=0.7] (b123) at (barycentric cs:a1=1,a2=1,a3=1){${\{\{0,1,2\}\}}$};

\node[rectangle, style=empty, scale=.7] (b12) at (barycentric cs:a1=1,a2=1){$\{\{0,1\}\}$};
\node (b13) at (barycentric cs:a1=1,a3=1){};

\node (b112) at (barycentric cs:b1=1,b12=1){};
\node (b212) at (barycentric cs:b2=1,b12=1){};

\node (b113) at (barycentric cs:b1=1,b13=1){};


\node (b1123) at (barycentric cs:a1=1,b123=1){};
\node (b2123) at (barycentric cs:a2=1,b123=1){};

\node (b12123) at (barycentric cs:b12=1,b123=1){};
\node (b13123) at (barycentric cs:b13=1,b123=1){};

\node (b112123) at (barycentric cs:b1=1,b12=1,b123=1){};

\node (b212123) at (barycentric cs:b2=1,b12=1,b123=1){};
\node (b113123) at (barycentric cs:b1=1,b13=1,b123=1){};

\draw (b1)--(b112);
\draw (b1)--(b113);
\draw (b2)--(b212);
\draw (b1)--(b1123);
\draw (b2)--(b2123);

\draw (b123)--(b1123);
\draw[dotted] (b123)--(b12123);
\draw (b123)--(b13123);
\draw (b123)--(b2123);

\draw[dotted] (b12)--(b12123);
\draw (b13)--(b13123);
\draw (b12)--(b112);
\draw (b13)--(b113);
\draw (b12)--(b212);

 \draw (b212)--(b212123);
 \draw (b2123)--(b212123);
 \draw (b12123)--(b212123);
%
 \draw (b112)--(b112123);
 \draw (b1123)--(b112123);
 \draw (b12123)--(b112123);

\draw (b113)--(b113123);
\draw (b1123)--(b113123);
\draw (b13123)--(b113123);

\end{tikzpicture}}

\newcommand{\perqinsdf}{							
\begin{tikzpicture}[every node/.style=my node, every path/.style=my path]
 \coordinate (a1) at (0,0); 
 \coordinate (a2) at (7,0);
 \coordinate (a3) at (3.5,6.062);
 
\node (b1) at (barycentric cs:a1=1,a2=0,a3=0){};
\node (b2) at (barycentric cs:a1=0,a2=1,a3=0){};
   
\node (b123) at (barycentric cs:a1=1,a2=1,a3=1){};

\node[style=empty] (b12) at (barycentric cs:a1=1,a2=1){$x$};
\node (b23) at (barycentric cs:a2=1,a3=1){};

\node[style=empty] (b112) at (barycentric cs:b1=1,b12=1){$y_1$};
\node[style=empty] (b212) at (barycentric cs:b2=1,b12=1){$y_3$};

\node (b223) at (barycentric cs:b2=1,b23=1){};


\node (b1123) at (barycentric cs:a1=1,b123=1){};
\node (b2123) at (barycentric cs:a2=1,b123=1){};

\node[style=empty] (b12123) at (barycentric cs:b12=1,b123=1){$y_2$};
\node (b23123) at (barycentric cs:b23=1,b123=1){};

\node[style=empty] (b112123) at (barycentric cs:b1=1,b12=1,b123=1){$z_1$};
\node (b223123) at (barycentric cs:b2=1,b23=1,b123=1){};

\node[style=empty] (b212123) at (barycentric cs:b2=1,b12=1,b123=1){$z_2$};

\draw (b1)--(b112);

\draw (b2)--(b212);
\draw (b2)--(b223);

\draw (b1)--(b1123);
\draw (b2)--(b2123);

\draw (b123)--(b1123);
\draw (b123)--(b2123);

\draw (b123)--(b12123);
\draw (b123)--(b23123);

\draw[dashed] (b12)--(b12123);
\draw (b23)--(b23123);

\draw[dashed] (b12)--(b112);

\draw[dashed] (b12)--(b212);
\draw (b23)--(b223);

\draw (b223)--(b223123);
\draw (b2123)--(b223123);
\draw (b23123)--(b223123);

 \draw[dashed] (b212)--(b212123);
 \draw (b2123)--(b212123);
 \draw[dashed] (b12123)--(b212123);
%
 \draw[dashed] (b112)--(b112123);
 \draw (b1123)--(b112123);
 \draw[dashed] (b12123)--(b112123);

\begin{pgfonlayer}{background}
	\tikzstyle{every path}=[-,gray!50,draw=gray!50,line width=15pt, line join=round, line cap=round]
	\path
		(b1.center) --
		(b112.center);
 	\path
 		(b2.center) --
 		(b212.center);
 	\path
 		(b1123.center) --
 		(b112123.center);
 	\path
		(b12123.center) --
 		(b123.center) --
 		(b23123.center) --
 		(b23.center);
	\path[fill]
		(b212123.center) --
		(b2123.center) --
		(b223123.center) -- 
		(b223.center) --cycle;
\end{pgfonlayer}

\end{tikzpicture}}


\newcommand{\latfive}{
\scalebox{.6}{\begin{tikzpicture}[every node/.style=my node, every path/.style=my path, rotate=180]
\node (a) at (0.5,0) {};
\node (b) at (1.5,0) {};
\node (c) at (0,0) {};
\node (d) at (2,0) {};
\node (e) at (1,-1) {};
\draw (a)--(e);
\draw (b)--(e);
\draw (c)--(e);
\draw (d)--(e);
\end{tikzpicture}}
\hspace{2em}
\scalebox{.6}{\begin{tikzpicture}[every node/.style=my node, every path/.style=my path, rotate=180]
\node (a) at (2,1) {};
\node (b) at (1,0) {};
\node (c) at (0,0) {};
\node (d) at (2,0) {};
\node (e) at (1,-1) {};
\draw (a)--(d);
\draw (b)--(e);
\draw (c)--(e);
\draw (d)--(e);
\end{tikzpicture}}
\hspace{2em}
\scalebox{.6}{\begin{tikzpicture}[every node/.style=my node, every path/.style=my path, rotate=180]
\node (a) at (0,1) {};
\node (b) at (1,1) {};
\node (c) at (2,1) {};
\node (d) at (1,0) {};
\node (e) at (1,-1) {};
\draw (a)--(d);
\draw (b)--(d);
\draw (c)--(d);
\draw (d)--(e);
\end{tikzpicture}}
\hspace{2em}
\scalebox{.6}{\begin{tikzpicture}[every node/.style=my node, every path/.style=my path, rotate=180]
\node (a) at (1,1) {};
\node (b) at (1,0) {};
\node (c) at (0,0) {};
\node (d) at (2,0) {};
\node (e) at (1,-1) {};
\draw (a)--(c);
\draw (a)--(d);
\draw (b)--(e);
\draw (c)--(e);
\draw (d)--(e);
\end{tikzpicture}}
\hspace{2em}
\scalebox{.6}{\begin{tikzpicture}[every node/.style=my node, every path/.style=my path, rotate=180]
\node (a) at (1,1) {};
\node (b) at (2,1) {};
\node (c) at (0,0) {};
\node (d) at (2,0) {};
\node (e) at (1,-1) {};
\draw (a)--(d);
\draw (b)--(d);
\draw (c)--(e);
\draw (d)--(e);
\end{tikzpicture}}
\vspace{2em}

\scalebox{.6}{\begin{tikzpicture}[every node/.style=my node, every path/.style=my path, rotate=180]
\node (a) at (0,1) {};
\node (b) at (2,1) {};
\node (c) at (0,0) {};
\node (d) at (2,0) {};
\node (e) at (1,-1) {};
\draw (a)--(c);
\draw (b)--(d);
\draw (c)--(e);
\draw (d)--(e);
\end{tikzpicture}}
\hspace{2em}
\scalebox{.6}{\begin{tikzpicture}[every node/.style=my node, every path/.style=my path, rotate=180]
\node (a) at (0,1) {};
\node (b) at (2,0) {};
\node (c) at (.5,.5) {};
\node (d) at (1,0) {};
\node (e) at (1.5,-.5) {};
\draw (a)--(c);
\draw (c)--(d);
\draw (b)--(e);
\draw (d)--(e);
\end{tikzpicture}}
\hspace{2em}
\scalebox{.6}{\begin{tikzpicture}[every node/.style=my node, every path/.style=my path, rotate=180]
\node (a) at (1,1) {};
\node (b) at (2,1) {};
\node (c) at (0,0) {};
\node (d) at (2,0) {};
\node (e) at (1,-1) {};
\draw (a)--(c);
\draw (a)--(d);
\draw (b)--(d);
\draw (c)--(e);
\draw (d)--(e);
\end{tikzpicture}}
\hspace{2em}
\scalebox{.6}{\begin{tikzpicture}[every node/.style=my node, every path/.style=my path, rotate=180]
\node (a) at (0,1) {};
\node (b) at (2,1) {};
\node (c) at (0,.5) {};
\node (d) at (0,0) {};
\node (e) at (1,-1) {};
\draw (a)--(c);
\draw (b)--(d);
\draw (c)--(d);
\draw (d)--(e);
\end{tikzpicture}}
\hspace{2em}
\scalebox{.6}{\begin{tikzpicture}[every node/.style=my node, every path/.style=my path, rotate=180]
\node (a) at (0,1) {};
\node (b) at (1,1) {};
\node (c) at (0,0) {};
\node (d) at (0,-.5) {};
\node (e) at (1,-1) {};
\draw (a)--(c);
\draw (b)--(c);
\draw (c)--(d);
\draw (d)--(e);
\end{tikzpicture}}
\vspace{2em}

\scalebox{.6}{\begin{tikzpicture}[every node/.style=my node, every path/.style=my path, rotate=180]
\node (a) at (0,1) {};
\node (b) at (2,1) {};
\node (c) at (0,0) {};
\node (d) at (2,0) {};
\node (e) at (1,-1) {};
\draw (a)--(b);
\draw (a)--(c);
\draw (a)--(d);
\draw (a)--(e);
\end{tikzpicture}}
\hspace{2em}
\scalebox{.6}{\begin{tikzpicture}[every node/.style=my node, every path/.style=my path, rotate=180]
\node (a) at (0,1) {};
\node (b) at (2,0) {};
\node (c) at (0,0) {};
\node (d) at (1,0) {};
\node (e) at (1,-1) {};
\draw (a)--(b);
\draw (a)--(c);
\draw (a)--(d);
\draw (b)--(e);
\end{tikzpicture}}
\hspace{2em}
\scalebox{.6}{\begin{tikzpicture}[every node/.style=my node, every path/.style=my path, rotate=180]
\node (a) at (1,1) {};
\node (b) at (2,0) {};
\node (c) at (0,0) {};
\node (d) at (1,0) {};
\node (e) at (1,-1) {};
\draw (a)--(b);
\draw (a)--(c);
\draw (a)--(d);
\draw (b)--(e);
\draw (c)--(e);
\draw (d)--(e);
\end{tikzpicture}}
\hspace{2em}
\scalebox{.6}{\begin{tikzpicture}[every node/.style=my node, every path/.style=my path, rotate=180]
\node (a) at (0,1) {};
\node (b) at (1,0) {};
\node (c) at (0,0) {};
\node (d) at (1.5,-1) {};
\node (e) at (.5,-1) {};
\draw (a)--(b);
\draw (a)--(c);
\draw (b)--(d);
\draw (b)--(e);
\end{tikzpicture}}
\hspace{2em}
\scalebox{.6}{\begin{tikzpicture}[every node/.style=my node, every path/.style=my path, rotate=180]
\node (a) at (0,1) {};
\node (b) at (1,0) {};
\node (c) at (0,0) {};
\node (d) at (1,-1) {};
\node (e) at (0,-1) {};
\draw (a)--(b);
\draw (a)--(c);
\draw (b)--(d);
\draw (c)--(e);
\end{tikzpicture}}
\vspace{2em}

\scalebox{.6}{\begin{tikzpicture}[every node/.style=my node, every path/.style=my path, rotate=180]
\node (a) at (0,1) {};
\node (b) at (1,0) {};
\node (c) at (0,0) {};
\node (d) at (1,-.5) {};
\node (e) at (1,-1) {};
\draw (a)--(b);
\draw (a)--(c);
\draw (b)--(d);
\draw (d)--(e);
\end{tikzpicture}}
\hspace{2em}
\scalebox{.6}{\begin{tikzpicture}[every node/.style=my node, every path/.style=my path, rotate=180]
\node (a) at (1,1) {};
\node (b) at (2,0) {};
\node (c) at (0,0) {};
\node (d) at (1,0) {};
\node (e) at (1,-1) {};
\draw (a)--(b);
\draw (a)--(c);
\draw (b)--(d);
\draw (b)--(e);
\draw (c)--(e);
\end{tikzpicture}}
\hspace{2em}
\scalebox{.6}{\begin{tikzpicture}[every node/.style=my node, every path/.style=my path, rotate=180]
\node (a) at (0,1) {};
\node (b) at (2,1) {};
\node (c) at (0,0) {};
\node (d) at (2,0) {};
\node (e) at (1,-1) {};
\draw (a)--(b);
\draw (a)--(c);
\draw (b)--(d);
\draw (c)--(e);
\draw (d)--(e);
\end{tikzpicture}}
\hspace{2em}
\scalebox{.6}{\begin{tikzpicture}[every node/.style=my node, every path/.style=my path, rotate=180]
\node (a) at (1,1) {};
\node (b) at (2,0) {};
\node (c) at (0,0) {};
\node (d) at (1,-1) {};
\node (e) at (0,-1) {};
\draw (a)--(b);
\draw (a)--(c);
\draw (b)--(d);
\draw (c)--(d);
\draw (d)--(e);
\end{tikzpicture}}
\hspace{2em}
\scalebox{.6}{\begin{tikzpicture}[every node/.style=my node, every path/.style=my path, rotate=180]
\node (a) at (1,1) {};
\node (b) at (1,0) {};
\node (c) at (0,-1) {};
\node (d) at (1,-1) {};
\node (e) at (2,-1) {};
\draw (a)--(b);
\draw (b)--(c);
\draw (b)--(d);
\draw (b)--(e);
\end{tikzpicture}}
\vspace{2em}

\scalebox{.6}{\begin{tikzpicture}[every node/.style=my node, every path/.style=my path, rotate=180]
\node (a) at (1,1) {};
\node (b) at (2,0) {};
\node (c) at (0,0) {};
\node (d) at (1,0) {};
\node (e) at (1,-1) {};
\draw (a)--(b);
\draw (a)--(c);
\draw (a)--(d);
\draw (b)--(e);
\draw (c)--(e);
\end{tikzpicture}}
\hspace{2em}
\scalebox{.6}{\begin{tikzpicture}[every node/.style=my node, every path/.style=my path, rotate=180]
\node (a) at (0,1) {};
\node (b) at (2,1) {};
\node (c) at (0,0) {};
\node (d) at (2,0) {};
\node (e) at (1,-1) {};
\draw (a)--(b);
\draw (b)--(c);
\draw (b)--(d);
\draw (c)--(e);
\end{tikzpicture}}
\hspace{2em}
\scalebox{.6}{\begin{tikzpicture}[every node/.style=my node, every path/.style=my path, rotate=180]
\node (a) at (0,1) {};
\node (b) at (2,1) {};
\node (c) at (0,0) {};
\node (d) at (2,0) {};
\node (e) at (1,-1) {};
\draw (a)--(b);
\draw (b)--(c);
\draw (b)--(d);
\draw (c)--(e);
\draw (d)--(e);
\end{tikzpicture}}
\hspace{2em}
\scalebox{.6}{\begin{tikzpicture}[every node/.style=my node, every path/.style=my path, rotate=180]
\node (a) at (0,1) {};
\node (b) at (2,.5) {};
\node (c) at (0,0) {};
\node (d) at (2,-.5) {};
\node (e) at (1,-1) {};
\draw (a)--(b);
\draw (b)--(c);
\draw (c)--(d);
\draw (c)--(e);
\end{tikzpicture}}
\hspace{2em}
\scalebox{.6}{\begin{tikzpicture}[every node/.style=my node, every path/.style=my path, rotate=180]
\node (a) at (0,1) {};
\node (b) at (2,.5) {};
\node (c) at (0,0) {};
\node (d) at (2,-.5) {};
\node (e) at (1,-1) {};
\draw (a)--(b);
\draw (b)--(c);
\draw (c)--(d);
\draw (d)--(e);
\end{tikzpicture}}}

\newcommand{\pushfive}{
\scalebox{.6}{\begin{tikzpicture}[every node/.style=my node, every path/.style=my path, rotate=180]		
\node (a) at (1.5,0) {};
\node (b) at (0.5,0) {};
\node (c) at (0,0) {};
\node (d) at (2,-1) {};
\node (e) at (1,-1) {};
\draw (a)--(d);
\draw (a)--(e);
\draw (b)--(e);
\draw (c)--(e);
\end{tikzpicture}}
\hspace{2em}
\scalebox{.6}{\begin{tikzpicture}[every node/.style=my node, every path/.style=my path, rotate=180]
\node (a) at (0,0) {};
\node (b) at (2,0) {};
\node (c) at (0,-1) {};
\node (d) at (1,-1) {};
\node (e) at (2,-1) {};
\draw (a)--(c);
\draw (a)--(d);
\draw (a)--(e);
\draw (b)--(e);
\end{tikzpicture}}
\hspace{2em}
\scalebox{.6}{\begin{tikzpicture}[every node/.style=my node, every path/.style=my path, rotate=180]
\node (a) at (0,1) {};
\node (b) at (2,0) {};
\node (c) at (0,0) {};
\node (d) at (1,0) {};
\node (e) at (1,-1) {};
\draw (a)--(c);
\draw (a)--(d);
\draw (b)--(e);
\draw (c)--(e);
\end{tikzpicture}}
\hspace{2em}
\scalebox{.6}{\begin{tikzpicture}[every node/.style=my node, every path/.style=my path, rotate=180]
\node (a) at (0,1) {};
\node (b) at (2,1) {};
\node (c) at (0,0) {};
\node (d) at (2,0) {};
\node (e) at (1,-1) {};
\draw (a)--(c);
\draw (b)--(d);
\draw (a)--(e);
\draw (b)--(e);
\end{tikzpicture}}
\hspace{2em}
\scalebox{.6}{\begin{tikzpicture}[every node/.style=my node, every path/.style=my path, rotate=180]
\node (a) at (0,1) {};
\node (b) at (2,1) {};
\node (c) at (0,0) {};
\node (d) at (2,0) {};
\node (e) at (1,-1) {};
\draw (a)--(c);
\draw (b)--(d);
\draw (a)--(e);
\draw (d)--(e);
\end{tikzpicture}}
\vspace{2em}

\scalebox{.6}{\begin{tikzpicture}[every node/.style=my node, every path/.style=my path, rotate=180] 
\node (a) at (0,1) {};
\node (b) at (2,1) {};
\node (c) at (0,0) {};
\node (d) at (0,-1) {};
\node (e) at (1,-1) {};
\draw (a)--(c);
\draw (c)--(d);
\draw (a)--(e);
\draw (b)--(e);
\end{tikzpicture}}
\hspace{2em}
\scalebox{.6}{\begin{tikzpicture}[every node/.style=my node, every path/.style=my path, rotate=180]
\node (a) at (1,1) {};
\node (b) at (2,1) {};
\node (c) at (0,0) {};
\node (d) at (2,0) {};
\node (e) at (2,-1) {};
\draw (a)--(c);
\draw (a)--(d);
\draw (b)--(d);
\draw (d)--(e);
\end{tikzpicture}}
\hspace{2em}
\scalebox{.6}{\begin{tikzpicture}[every node/.style=my node, every path/.style=my path, rotate=180]
\node (a) at (1,1) {};
\node (b) at (1,-1) {};
\node (c) at (1,0) {};
\node (d) at (2,0) {};
\node (e) at (0,0) {};
\draw (a)--(c);
\draw (a)--(d);
\draw (b)--(d);
\draw (a)--(e);
\draw (b)--(e);
\end{tikzpicture}}
\hspace{2em}
\scalebox{.6}{\begin{tikzpicture}[every node/.style=my node, every path/.style=my path, rotate=180]
\node (a) at (0,1) {};
\node (b) at (2,1) {};
\node (c) at (0,0) {};
\node (d) at (0,-1) {};
\node (e) at (1,-1) {};
\draw (a)--(c);
\draw (c)--(d);
\draw (b)--(e);
\draw (c)--(e);
\end{tikzpicture}}}

\newcommand{\potpn}{
\scalebox{.6}{\begin{tikzpicture}[every node/.style=my node, every path/.style=my path, rotate=180]
\node (a) at (1,1) {};
\node (b) at (1,-1) {};
\node (c) at (1,0) {};
\node (d) at (2,0) {};
\node (e) at (0,0) {};
\draw (a)--(d);
\draw (b)--(d);
\draw (a)--(e);
\draw (b)--(e);
\draw (c)--(e);
\end{tikzpicture}}
\hspace{2em}
\scalebox{.6}{\begin{tikzpicture}[every node/.style=my node, every path/.style=my path, rotate=180]
\node (a) at (2,0) {};
\node (b) at (1,0) {};
\node (c) at (0,0) {};
\node (d) at (1,1) {};
\node (e) at (1,-1) {};
\draw (a)--(d);
\draw (b)--(d);
\draw (c)--(d);
\draw (a)--(e);
\draw (b)--(e);
\draw (c)--(e);
\end{tikzpicture}}
\hspace{2em}
\scalebox{.6}{\begin{tikzpicture}[every node/.style=my node, every path/.style=my path, rotate=180]
\node (a) at (0,1) {};
\node (b) at (1.5,0) {};
\node (c) at (0,0) {};
\node (d) at (1,.5) {};
\node (e) at (1,-1) {};
\draw (a)--(c);
\draw (a)--(d);
\draw (b)--(d);
\draw (b)--(e);
\draw (c)--(e);
\end{tikzpicture}}
\hspace{2em}
\scalebox{.6}{\begin{tikzpicture}[every node/.style=my node, every path/.style=my path, rotate=180]
\node (a) at (0,1) {};
\node (b) at (2,0) {};
\node (c) at (0,0) {};
\node (d) at (1,1) {};
\node (e) at (1,-1) {};
\draw (a)--(c);
\draw (b)--(d);
\draw (c)--(d);
\draw (b)--(e);
\draw (c)--(e);
\end{tikzpicture}}
\vspace{2em}

\scalebox{.6}{\begin{tikzpicture}[every node/.style=my node, every path/.style=my path, rotate=180]
\node (a) at (0,1) {};
\node (b) at (2,1) {};
\node (c) at (1,0) {};
\node (d) at (2,-1) {};
\node (e) at (0,-1) {};
\draw (a)--(c);
\draw (b)--(c);
\draw (c)--(d);
\draw (c)--(e);
\end{tikzpicture}}
\hspace{2em}
\scalebox{.6}{\begin{tikzpicture}[every node/.style=my node, every path/.style=my path, rotate=180]
\node (a) at (0,0) {};
\node (b) at (2,0) {};
\node (c) at (1,1) {};
\node (d) at (1,0) {};
\node (e) at (1,-1) {};
\draw (a)--(c);
\draw (b)--(c);
\draw (a)--(d);
\draw (b)--(d);
\draw (a)--(e);
\draw (b)--(e);
\end{tikzpicture}}
\hspace{2em}
\scalebox{.6}{\begin{tikzpicture}[every node/.style=my node, every path/.style=my path, rotate=180]
\node (a) at (1,1) {};
\node (b) at (1,0) {};
\node (c) at (0,0) {};
\node (d) at (2,0) {};
\node (e) at (1,-1) {};
\draw (a)--(c);
\draw (b)--(c);
\draw (a)--(d);
\draw (b)--(d);
\draw (c)--(e);
\draw (d)--(e);
\end{tikzpicture}}
\hspace{2em}
\scalebox{.6}{\begin{tikzpicture}[every node/.style=my node, every path/.style=my path, rotate=180]
\node (a) at (1,1) {};
\node (b) at (2,0) {};
\node (c) at (0,0) {};
\node (d) at (1,-1) {};
\node (e) at (1,0) {};
\draw (a)--(b);
\draw (a)--(c);
\draw (b)--(d);
\draw (c)--(d);
\draw (b)--(e);
\draw (c)--(e);
\end{tikzpicture}}
\vspace{2em}

\scalebox{.6}{\begin{tikzpicture}[every node/.style=my node, every path/.style=my path, rotate=180]		
\node (a) at (0,0) {};
\node (b) at (2,0) {};
\node (c) at (1,1) {};
\node (d) at (1,0) {};
\node (e) at (1,-1) {};
\draw (a)--(c);
\draw (b)--(c);
\draw (c)--(d);
\draw (a)--(e);
\draw (b)--(e);

\end{tikzpicture}}}

\newcommand{\posthree}{
\scalebox{.6}{\begin{tikzpicture}[every node/.style=my node, every path/.style=my path, rotate=180]
\node (a) at (0,1) {};
\node (b) at (.5,0) {};
\node (c) at (1,1) {};
\draw (a)--(b);
\draw (c)--(b);
\end{tikzpicture}}
\hspace{2em}
\scalebox{.6}{\begin{tikzpicture}[every node/.style=my node, every path/.style=my path, rotate=180]
\node (a) at (0,0) {};
\node (b) at (.5,1) {};
\node (c) at (1,0) {};
\draw (a)--(b);
\draw (c)--(b);
\end{tikzpicture}}
\hspace{2em}
\scalebox{.6}{\begin{tikzpicture}[every node/.style=my node, every path/.style=my path, rotate=180]
\node (a) at (0,2) {};
\node (b) at (0,1) {};
\node (c) at (0,0) {};
\draw (a)--(b);
\draw (b)--(c);
\end{tikzpicture}}}

\newcommand{\latfour}{
\scalebox{.6}{\begin{tikzpicture}[every node/.style=my node, every path/.style=my path, rotate=180]
\node (a) at (0,1) {};
\node (b) at (1,1) {};
\node (c) at (2,1) {};
\node (d) at (1,0) {};
\draw (a)--(d);
\draw (b)--(d);
\draw (c)--(d);
\end{tikzpicture}}
\hspace{2em}
\scalebox{.6}{\begin{tikzpicture}[every node/.style=my node, every path/.style=my path, rotate=180]
\node (a) at (0,2) {};
\node (b) at (1.5,1) {};
\node (c) at (.5,1) {};
\node (d) at (1,0) {};
\draw (a)--(c);
\draw (b)--(d);
\draw (c)--(d);
\end{tikzpicture}}
\hspace{2em}
\scalebox{.6}{\begin{tikzpicture}[every node/.style=my node, every path/.style=my path, rotate=180]
\node (a) at (0,2) {};
\node (b) at (1,2) {};
\node (c) at (.5,1) {};
\node (d) at (.5,0) {};
\draw (a)--(c);
\draw (b)--(c);
\draw (c)--(d);
\end{tikzpicture}}
\hspace{2em} 
\scalebox{.6}{\begin{tikzpicture}[every node/.style=my node, every path/.style=my path, rotate=180]
\node (a) at (1,1) {};
\node (b) at (0,0) {};
\node (c) at (1,0) {};
\node (d) at (2,0) {};
\draw (a)--(b);
\draw (a)--(c);
\draw (a)--(d);
\end{tikzpicture}}
\vspace{2em}

\scalebox{.6}{\begin{tikzpicture}[every node/.style=my node, every path/.style=my path, rotate=180]
\node (a) at (.5,2) {};
\node (b) at (1,1) {};
\node (c) at (0,1) {};
\node (d) at (1.5,0) {};
\draw (a)--(b);
\draw (a)--(c);
\draw (b)--(d);
\end{tikzpicture}}
\hspace{2em}
\scalebox{.6}{\begin{tikzpicture}[every node/.style=my node, every path/.style=my path, rotate=180]
\node (a) at (1,2) {};
\node (b) at (0,1) {};
\node (c) at (2,1) {};
\node (d) at (1,0) {};
\draw (a)--(b);
\draw (a)--(c);
\draw (b)--(d);
\draw (c)--(d);
\end{tikzpicture}}
\hspace{2em}
\scalebox{.6}{\begin{tikzpicture}[every node/.style=my node, every path/.style=my path, rotate=180]
\node (a) at (.5,2) {};
\node (b) at (.5,1) {};
\node (c) at (0,0) {};
\node (d) at (1,0) {};
\draw (a)--(b);
\draw (b)--(c);
\draw (b)--(d);
\end{tikzpicture}}
\hspace{2em}
\scalebox{.6}{\begin{tikzpicture}[every node/.style=my node, every path/.style=my path, rotate=180]
\node (a) at (0,2.1) {};
\node (b) at (1,1.4) {};
\node (c) at (0,.7) {};
\node (d) at (1,0) {};
\draw (a)--(b);
\draw (b)--(c);
\draw (c)--(d);
\end{tikzpicture}}}

\newcommand{\posfour}{
\scalebox{.6}{\begin{tikzpicture}[every node/.style=my node, every path/.style=my path, rotate=180]
\node (a) at (0,1) {};
\node (b) at (1,1) {};
\node (c) at (0,0) {};
\node (d) at (1,0) {};
\draw (a)--(c);
\draw (b)--(c);
\draw (a)--(d);
\draw (b)--(d);
\end{tikzpicture}}
\hspace{2em}
\scalebox{.6}{\begin{tikzpicture}[every node/.style=my node, every path/.style=my path, rotate=180]
\node (a) at (0,1) {};
\node (b) at (1,1) {};
\node (c) at (0,0) {};
\node (d) at (1,0) {};
\draw (a)--(c);
\draw (a)--(d);
\draw (b)--(d);
\end{tikzpicture}}}

\section{Introduction}
In \cite{thomason:ccmc}, Thomason showed that there is a model structure on $\Cat$---the category of small categories---which is Quillen equivalent to the Quillen model structure on $\sSet$,  the category of simplicial sets. This model structure has been lifted to $\Pos$, the category of posets in \cite{raptis:hotopo}, and to $\Ac$, the category of small acyclic categories in \cite{bruckner:amsosac}. In his original paper, Thomason already showed that every cofibrant object in his model structure on $\Cat$ is a poset \cite[Proposition~5.7]{thomason:ccmc}. Since the model structures on $\Pos$, $\Ac$ and $\Cat$ are cofibrantly generated by the same classes of generating cofibrations and trivial cofibrations, and pushouts and colimits along cofibrations yield the same objects in all of those, this implies that all three categories feature the same class of cofibrant objects. 

There have been many attempts to find Thomason model structures on various categories, \eg the category of $G$--categories \cite{bmoopy:gcat}, the category of $G$--posets \cite{msz:htep}, the category of strict $n$-categories \cite{am:ncat} or the category of small $n$-fold categories \cite{fiorepaoli:nfcat}. However, to the knowledge of the authors, there have been no attempts to identify cofibrant objects in the Thomason model structure, except for one chapter in \cite{msz:htep}. 

In \cite[Proposition~6.5]{msz:htep} it was proved, that every finite one-dimensional poset is cofibrant, \ie every poset that has a $1$-skeletal nerve. In this paper we identify various other classes of cofibrant posets. In Section~\ref{sec:tcl}, we show that every finite semilattice, every countable tree, every chain and every finite zigzag is cofibrant. In Section~\ref{sec:finpos} we show that every poset with five or less elements is cofibrant. Moreover, we prove that every inclusion of a minimum into any of the cofibrant posets we identified is a cofibration.
\section{Preliminaries}
We assume that the reader has some general knowledge on model categories and the Thomason model structure and use this section primarily to fix some notation and terminology.
We denote by $\Delta^n$ the standard $n$--simplex, given by $\Delta^n = N([n])$ where $N\colon\Cat\to\sSet$ is the usual nerve functor, and 
\begin{align*}
	[n]=0\to1\to\dotsb\to n
\end{align*}
the finite ordinal with $n+1$ elements. Given a simplicial set $X$, the barycentric subdivision is denoted by $\sd X$. In particular, given the barycentric subdivision of a standard $n$--simplex, we denote the vertices of $\sd\Delta^n$ by sets of vertices of $\Delta^n$. For example,
\begin{align*}
	(\sd\Delta^2)_0 = \left\{\left\{0\right\},\left\{1\right\},\left\{2\right\},\left\{0,1\right\},\left\{1,2\right\},\left\{0,2\right\},\left\{0,1,2\right\}\right\}\text{.}
\end{align*}
By $\tau_1\colon\sSet\to\Cat$ we denote the fundamental category functor, \ie the left adjoint of the nerve.

Given any poset $P$, we treat it as a category in the sense that there is an arrow from $x$ to $y$ if and only if $x\le y$. In particular, when talking about the natural numbers $\mathbb{N}$, we are talking about the category
\begin{align*}
	0\to 1\to 2\to 3\to\dotsb\text{.}
\end{align*}
 Furthermore, we denote by $\iel$ the initial, and by $\tel$ the terminal object in $\Pos$, $\Cat$, and $\Ac$, respectively.

When we talk about the Thomason model structure, we mean the Thomason model structure on either $\Pos$, $\Ac$, or $\Cat$. In particular, when we use the term category, we refer to an object in any Thomason model structure. Given any category $\cat{C}$, we denote by $\cat{C}^{(0)}$ its class of objects and by $\cat{C}^{(1)}$ its class of morphisms. Given two categories $\cat{C}$ and $\cat{D}$, we say that $\cat{D}$ is a \emph{retract} of $\cat{C}$, if $\iel\to\cat{D}$ is a retract of $\iel\to\cat{C}$.

We need a few theorems, that are used throughout this paper.

\begin{mythm}\label{thm:sdiscof}
	The category $\tau_1\sd\Delta^n$ is cofibrant, and every inclusion
	\begin{align*}
		i_k\colon[0]&\to\tau_1\sd\Delta^n\text{,}\\
		0&\mapsto\left\{k\right\}
	\end{align*}
	is a cofibration.
\end{mythm}

\begin{proof}
	This is found in the proof of \cite[Lemma~1]{cisinski:dnpr}.
\end{proof}

\begin{mylem}\label{lem:imapycof}
	Let $\cat{C}$ be a category. If $\tel\to\cat{C}$ is a cofibration, $\cat{C}$ is cofibrant.
\end{mylem}

\begin{proof}
	Since $\tel$ lies in the image of $\thol$, $\tel$ is cofibrant, and since cofibrations are closed under composition, $\iel\to\tel\to\cat{C}$ is a cofibration and thus, $\cat{C}$ is cofibrant.
\end{proof}

Moreover, note that since $\thol$ is a left Quillen functor, and monomorphisms in $\sSet$ are cofibrations, the image of a monomorphism under $\thol$ is a cofibration in the Thomason model structure. We will use this fact quite often throughout this paper.
\section{Trees, Zigzags, Chains and Semiattices}
\label{sec:tcl}
In this section, we  proof that every countable tree, every finite zigzag, every chain and every finite semilattice is cofibrant in the Thomason model structure.

\subsection{Finite Semiattices}
We first show that the semilattices constructed from a Boolean lattice by removing either the top, or the bottom element are cofibrant in Lemma~\ref{lem:bopcof} and \ref{lem:bnmiscof}, respectively. In Theorem~\ref{thm:sliscof} we use the resulting semilattices to construct arbitrary, finite semilattices as retracts of those. We start by giving some definitions.

\begin{mydef}
	Let $\cat{C}$ be a small category, $A\subseteq\cat{C}^{(0)}$. We denote by $\cat{C}\setminus A$ the full subcategory of $\cat{C}$ with object set $\cat{C}^{(0)}\setminus A$. In particular, if $A=\{x\}$, we simply write $\cat{C}\setminus x$.
\end{mydef}

\begin{mydef}
	Let $\cat{C}$ be a small category. We denote by $\cat{P}(\cat{C})$ the category which is given by the power set lattice of $\cat{C}^{(0)}$.
\end{mydef}

In the following, we will denote the minimal element of $\cat{P}(\cat{C})$ by $\eset$, as opposed to $\iel$ to avoid any confusion on whether we are talking about the minimal element of $\cat{P}(\cat{C})$, or the initial object in the ambient category.

\begin{mylem}\label{lem:bopcof}
	The category $\cat{P}([n])\setminus[n]$ is cofibrant and the inclusion $\tel\to\cat{P}([n])\setminus[n]$ of the minimum is a cofibration.
\end{mylem}

\begin{proof}
	Let $\xi\colon\mathbf{Pos}\to\Cat$ be the functor which maps a poset $P$ to the lattice of non-empty chains in $P$, ordered by inclusion. Then $\xi = \tau_1\sd N$ and $\xi^2 = \tau_1\sd^2 N$ (cf. \cite{cisinski:dnpr}). Consider the diagram
	\begin{equation*}
		\begin{tikzcd}
			\ar[d, "i_\eset"][0]\ar[r]&\ar[d, "\iota_\eset"]\xi[0]\ar[r]&\ar[d, "i_\eset"][0]\\
 			\cat{P}([n])\setminus[n]\ar[r, "i"]&\xi(\cat{P}([n])\setminus[n])\ar[r, "p"]&\cat{P}([n])\setminus[n]
		\end{tikzcd}\text{,}
	\end{equation*}
	where $i_\eset$ and $\iota_\eset$ are the minima inclusions, and $i$ and $p$ are given as follows: let $A=\{m_1,m_2,\dotsc,m_k\}\in\cat{P}([n])\setminus[n]$, such that $m_1\le m_2\le\dotsb\le m_k$. We set
	\begin{equation*}
		i(A) = \left\{\{\eset\},\{m_1\},\{m_1,m_2\},\dotsc,\{m_1,m_2,\dotsc,m_k\}\right\}
	\end{equation*}
	and given $B\in\cat{P}([n])\setminus[n]$, we set $p(B)=\bigcup B$. 

	It is easy to see, that $p\circ i=\id$, hence $i_\eset$ is a retract of $\iota_\eset$. If we apply $\xi$ to the whole diagram, we get that $\xi(i_\eset)=\iota_\eset$ is a retract of $\xi(\iota_\eset)$, which is a cofibration since $\xi(\iota_\eset) = \thol N(i_\eset)$ and $N(i_\eset)$ is a monomorphism in $\sSet$. Hence $i_\eset$ is a cofibration and $\cat{P}([n])\setminus[n]$ is cofibrant.
\end{proof}

\begin{mylem}\label{lem:bnmiscof}
	The category $\cat{P}([n])\setminus\eset$ is cofibrant and the inclusions of the respective minima are cofibrations.
\end{mylem}

\begin{proof}
	Since $\cat{P}([n])\setminus\eset\cong\tau_1\sd\Delta^{n}$, this follows from \ref{thm:sdiscof}.
\end{proof}

\begin{mythm}\label{thm:sliscof}
	Every finite semilattice is cofibrant.
\end{mythm}

\begin{proof}
	Let $L$ be a finite join--semilattice. Define
	\begin{align*}
		i\colon L&\to\cat{P}(L)\setminus \eset\text{,}\\
		x&\mapsto\left\{\,y\in L\,\middle|\, y\le x\,\right\}
	\intertext{and}
		p\colon\cat{P}(L)\setminus\eset&\to L\text{,}\\
		A&\mapsto\bigvee A\text{,}
	\end{align*}
	where $\bigvee A$ denotes the join over all elements of $A$ in $L$. Then $p\circ i = \id$, hence $L$ is a retract of $\cat{P}(L)\setminus\eset$ and since $\cat{P}(L)\setminus\eset$ is cofibrant by Lemma~\ref{lem:bnmiscof}, so is $L$. Thus every join--semilattice is cofibrant.

	Let $M$ be a finite meet--semilattice. Since $M^\mathrm{op}$ is a join--semilattice and $\cat{P}(M) = \cat{P}(M^\mathrm{op})$ we obtain a retract diagram
	\begin{equation*}
		\begin{tikzcd}
			M^\mathrm{op}\ar[r,"i"]&{\cat{P}}(M)\setminus\eset\ar[r,"p"]& M^\mathrm{op}
		\end{tikzcd}\text{,}
	\end{equation*}
	where $i$ and $p$ are given as before. Dualizing every object we obtain a retract
	\begin{equation*}
		\begin{tikzcd}
			M\ar[r,"i^\mathrm{op}"]&(\cat{P}(M)\setminus\eset)^\mathrm{op}\ar[r,"p^\mathrm{op}"]& M
		\end{tikzcd}\text{.}
	\end{equation*}
	But $(\cat{P}(M)\setminus\eset)^\mathrm{op}$ is isomorphic to $\cat{P}(M)\setminus M$, which is cofibrant by Lemma~\ref{lem:bopcof} and hence so is $M$.
\end{proof}

\begin{mycor}\label{cor:slinccof}
	Let $S$ be a finite semilattice, and $i_m\colon\tel\to S$ an inclusion that maps the single element of $\tel$ to a local minimum $m$ in $S$, then $i_m$ is a cofibration.
\end{mycor}

\begin{proof}
	Let $L$ be a finite join--semilattice, $m\in L$ be a local minimum. Define
	\begin{align*}
		i_m\colon\tel&\to L\text{,}\\
		0&\mapsto m
	\intertext{and}
		\iota_m\colon\tel&\to\cat{P}(L)\setminus\eset\text{,}\\
		0&\mapsto\left\{m\right\}\text{.}
	\end{align*}		
	We obtain a diagram
	\begin{equation*}
		\begin{tikzcd}
			\tel\ar[r]\ar[d,"i_m"]&\tel\ar[r]\ar[d,"\iota_m"]&\tel\ar[d,"i_m"]\\
			L\ar[r,"i"]&{\cat{P}}(L)\setminus\eset\ar[r,"p^\mathrm{op}"]& L
		\end{tikzcd}\text{,}
	\end{equation*}
	where $i$ and $p$ are given as in the proof of Theorem~\ref{thm:sliscof}. It is easy to see that every square commutes and since $\iota_m$ is a cofibration by Theorem~\ref{thm:sdiscof}, so is $i_m$.

	Now let $M$ be a finite meet--semilattice, $m\in M$ be the minimal element. Define
	\begin{align*}
		i_m\colon\tel&\to M\text{,} \\
		0&\mapsto m
	\intertext{as before, and}
		\iota_\eset\colon\tel&\to\cat{P}(M)\setminus M\text{,}\\
		0&\mapsto\eset\text{.}
	\end{align*}
	Consider the isomorphisms
	\begin{align*}
		(\cat{P}(M)\setminus\eset)^\mathrm{op}\xrightarrow{\varphi}{\cat{P}(M)\setminus M}\xrightarrow{\psi}{(\cat{P}(M)\setminus\eset)^\mathrm{op}}\text{,}
	\end{align*}
	both of which are given by mapping a subset $A\subseteq M$ to its complement. We obtain a retract diagram
		\begin{equation*}
		\begin{tikzcd}
			\tel\ar[rr]\ar[d,"i_m"]&&\tel\ar[rr]\ar[d,"\iota_\eset"]&&\tel\ar[d,"i_m"]\\
			M\ar[r,"i^\mathrm{op}"]&(\cat{P}(M)\setminus\eset)^\mathrm{op}\ar[r,"\varphi"]&\cat{P}(M)\setminus M\ar[r,"\psi"]&(\cat{P}(M)\setminus\eset)^\mathrm{op}\ar[r,"p"]& M
		\end{tikzcd}\text{,}
	\end{equation*}
	and since $\iota_\eset$ is a cofibration by Lemma~\ref{lem:bopcof}, so is $i_m$.
\end{proof}

\subsection{Chains}

\begin{mydef}
	Let $\cat{C}$ be a category. We call $\cat{C}$ a \emph{chain} if $\cat{C}$ is either isomorphic to a finite ordinal $[n]$, or to the natural numbers $\mathbb{N}$.
\end{mydef}

\begin{mythm}\label{thm:cacof}
	Every chain is cofibrant and the inclusion of the minimum is a cofibration.
\end{mythm}

\begin{proof}
	Let $\cat{C}$ be a chain. If $\cat{C}$ is finite, then it is cofibrant by Theorem~\ref{thm:sliscof} and the inclusion of the minimum is a cofibration by Corollary~\ref{cor:slinccof}. Assume that $\cat{C}$ is isomorphic to $\mathbb{N}$. We will construct a sequence
	\begin{align*}
		X\colon\mathbb{N}&\to\Cat\text{,}\\
			i&\mapsto X_i\text{,}\\
		(i\to i+1)&\mapsto F_i
	\end{align*}
	such that $X_0$ is cofibrant, $\colim_\mathbb{N} X\cong\cat{C}$, and the map $X_0\to\colim_\mathbb{N} X$ is a cofibration and the inclusion of the minimum, which yields that $\colim_\mathbb{N} X$ is cofibrant.
	Let $X_0 = {x_0}$. Assume that 
	\begin{align*}
		X_i = x_0\xrightarrow{f_0} x_1\xrightarrow{f_1}x_2\xrightarrow{f_2}\cdots\xrightarrow{f_{i-1}}x_i
	\end{align*}
	and let $\cat{D} = x\to y$. We construct $X_{i+1}$ from $X_i$ via the pushout
	\begin{equation*}
		\begin{tikzcd}
			\tel\ar[d, "h"']\ar[r, "f"] &\ar[d]\cat{D}\\
			X_i \ar[r, "F_i"] & X_{i+1}
		\end{tikzcd}\text{,}
	\end{equation*}
	where $f$ and $h$ are given by $f(0)=x$ and $h(0)=x_i$. Since $f$ is a cofibration by Corollary~\ref{cor:slinccof}, so is $F_i$. Moreover, since the class of cofibrations is closed under transfinite composition, the map $X_0\to\colim_\mathbb{N} X$ is a cofibration and thus $\colim_\mathbb{N} X$ is cofibrant. Furthermore $\cat{C}$ is a universal co-cone for $X$ by construction and thus $\colim_\mathbb{N} X\cong\cat{C}$.
\end{proof}

\subsection{Finite Zigzags}
We will proof that every finite zigzag is cofibrant by showing that a certain class of zigzags is cofibrant (Lemma~\ref{lem:chaincof}) and then glue together arbitrary finite zigzags from elements of this class (Theorem~\ref{thm:zziscof}). At first, we should give a definition of what we mean by zigzag.

\begin{mydef}
	A \emph{zigzag} is a category $\cat{Z}$, which is generated by a (possibly infinite) directed graph
	\begin{align*}
		\dotsb\leftrightarrow x_{i-1} \leftrightarrow x_{i} \leftrightarrow x_{i+1} \leftrightarrow\dotsb
	\end{align*}
	where $\leftrightarrow$ denotes either an arrow pointing to the left, or to the right.
\end{mydef}

Alternatively one might say a zigzag is a category generated by a total order
\begin{align*}
	\dotsb\to x_{i-2}\to x_{i-1}\to x_i\to x_{i+1}\to x_{i+2}\to\dotsb
\end{align*}
where some of the generating arrows are flipped:
\begin{align*}
	\dotsb\leftarrow x_{i-2}\to x_{i-1}\leftarrow x_i\leftarrow x_{i+1}\to x_{i+2}\to\dotsb
\end{align*}

\begin{mylem}\label{lem:chaincof}
	Let $\cat{Z}$ be a finite zigzag with a global maximum, then $\cat{Z}$ is cofibrant and the minimum inclusions are cofibrations.
\end{mylem}

\begin{proof}
	The claim is trivial for zigzags with one or less elements. Hence let $\cat{Z}$ be a zigzag with $n+2$ elements and a global maximum. If $\cat{Z}$ is isomorphic to $[n+1]$, $\cat{Z}$ is cofibrant by Theorem~\ref{thm:sliscof} and we are done. Otherwise, we can write $\cat{Z}$ as 
	\begin{equation*}
		\cat{Z} = x_0\xrightarrow{f_0}\dotsb \xrightarrow{f_{i-2}}x_{i-1}\xrightarrow{f_{i-1}}x_i\xleftarrow{f_i}\dotsb \xleftarrow{f_n}x_{n+1}\text{.}
	\end{equation*}
	Assume without loss of generality that $i>n/2$. By Theorem~\ref{thm:sdiscof}, the inclusions 
	\begin{align*}
		\iota_k\colon[0]&\hookrightarrow \tau_1\sd \Delta^n\text{,}\\
		0&\mapsto \left\{k\right\}
	\end{align*}
	are cofibrations. We construct $\cat{Z}$ as a retract of $\tau_1\sd\Delta^n$ and the inclusions $i_0, i_{n+1}\colon \tel\to\cat{Z}$ of the minima as retracts of $\iota_0$ and $\iota_n$, respectively. Define
	\begin{align*}
		i\colon \cat{Z}&\to\tau_1\sd \Delta^n\text{,}\\
		x_k &\mapsto \begin{cases}
		           	\{0,1,\dotsc,k\}			& \text{if } k<i \\
				\{0,1,\dotsc,n\}			& \text{if } k=i \\
				\{k-1,k,\dotsc, n\}	& \text{if } k>i \\
		           \end{cases}
	\intertext{and}
		p\colon\tau_1\sd\Delta^n&\to\cat{Z}\text{,}\\
		\sigma = \{k_1,k_2,\dotsc,k_p\} &\mapsto 	\begin{cases}
								x_{\left|\sigma\right|-1}& \text{if } 0\le k_1,\dotsc, k_p < i \\
								x_{n+2-\left|\sigma\right|} &\text{if } i\le k_1,\dotsc, k_p\le n \\
								x_i & \text{else} \\
								\end{cases}\text{.}
	\end{align*}
	Then $p\circ i = \id$ and we obtain a retract diagram
 	\begin{equation*}
 		\begin{tikzcd}
 		\tel\ar[r]\ar[d,"i_k"]&\tel\ar[r]\ar[d, "\iota_l"]&\tel\ar[d,"i_k"]\\
 		\cat{Z}\ar[r, "i"]&\tau_1\sd N[n]\ar[r, "p"]&\cat{Z}
 		\end{tikzcd}
 	\end{equation*}
	where $k=l=0$ or $k=n+1$, $l=n$. Since $\iota_0$ and $\iota_{n}$ are cofibrations, so are $i_0$ and $i_{n+1}$. Hence $\cat{Z}$ is cofibrant by Lemma~\ref{lem:imapycof}. 
\end{proof}

\begin{mythm}\label{thm:zziscof}
	Every finite zigzag is cofibrant, and every inclusion of a minimum into a zigzag is a cofibration.
\end{mythm}

\begin{proof}
	Given a zigzag $\cat{Z}_n$, we prove the claim by induction over the number $n$ of local minima. If $n=1$, $\cat{Z}_n$ is a either a chain, hence cofibrant and the inclusion of the minimum is a cofibration by Theorem~\ref{thm:cacof}, or can be obtained by gluing together two chains $\cat{C}_1$, $\cat{C}_2$ along their minima via the pushout
	\begin{equation*}
		\begin{tikzcd}
			\ar[d,"i"]\tel\ar[r,"j"]&\ar[d,"\alpha"]\cat{C}_1\\
			\cat{C}_2\ar[r, "\beta"]&\cat{Z}_n
		\end{tikzcd}
	\end{equation*}
	where $i$ and $j$ are the respective minimum inclusions. Since $i$ and $j$ are cofibrations by Theorem~\ref{thm:cacof}, so are $\alpha$ and $\beta$ and thus, in particular, the compositions $\alpha\circ j$ or $\beta\circ i$.

	Now assume the for every $n<N$, a zigzag with $n$ local minima is cofibrant, and every inclusion of a minimum is a cofibration. Let $\cat{Z}_N$ be a zigzag with $N$ local minima. There is a zigzag $\cat{Z}_{N-1}$ with $N-1$ local minima, and a zigzag $\cat{Z}$ with a global maximum, such that we can construct $\cat{Z}_N$ via a pushout
 	\begin{equation*}	
 		\begin{tikzcd}
 		\ar[d, "i_0"']\tel\ar[r, "k_{N-1}"] &  \cat{Z}_{N-1}\ar[d, "\iota"]\\
 		\cat{Z}\ar[r, "\kappa"] & \cat{Z}_N
 		\end{tikzcd}\text{,}
 	\end{equation*}
 	where $i_0$ and $k_{N-1}$ are inclusions of local minima, hence cofibrations. Thus, so are $\iota$ and $\kappa$ and therefore also the compositions 
	\begin{align*}
		\iel\to\cat{Z}\xrightarrow{\kappa}\cat{Z}_N\text{,}&\\
		\tel\xrightarrow{i_0}\cat{Z}\xrightarrow{\kappa}\cat{Z}_N\text{,}&\\
		\tel\xrightarrow{i_1}\cat{Z}\xrightarrow{\kappa}\cat{Z}_N\text{,}&
	\intertext{and}\\
		\tel\xrightarrow{k_j}\cat{Z}_{N-1}\xrightarrow{\kappa}\cat{Z}_N\text{,}&\qquad j=1,\dotsc, N-1\text{,}
	\end{align*}
	where $i_1$ is the other minimum inclusion of $\cat{Z}$, and $k_j$ are the minimum inclusions of $\cat{Z}_{N-1}$.
\end{proof}

\begin{myrem}
	Since the map $\iota$ from the previous proof is a cofibration, we can construct some infinite cofibrant zigzags with a countable number of objects via transfinite composition, similar to the methods used in proof of Lemma~\ref{lem:chaincof}. There are however zigzags, where this will not work. 

	Assume that $\cat{Z}$ is a countable infinite zigzag, with a finite number $N$ of local maxima. If 
	\begin{align*}
		\cat{Z} = \dotsb\leftrightarrows x_{i-1}\to x_{i}\leftarrow x_{i+1}\leftarrow\dotsb\text{,}
	\end{align*}
	such that there are no more local extrema to the right of $x_i$, then $\cat{Z}$ is not constructable as a pushout along a cofibration using the methods from the previous proof.
\end{myrem}

\subsection{Posets generated by directed trees}
By a directed tree, we mean a directed graph which is also a rooted tree, satisfying that every edge points away from the root. There is a natural poset structure on the set of vertices: given two vertices $x,y$, we say that $x\le y$ if and only if the unique path from the root to $y$ passes through $x$. We call the resulting poset the \emph{poset generated by the tree}.

Note that there is a natural grading on the resulting poset structure, where the rank of an element $x$ is given by the length of the minimal chain including $x$ and the root. We will now show, that any such poset is cofibrant in the Thomason model structure, and that the inclusion of the root is a cofibration:

\begin{mythm}
	Every countable poset generated by a directed tree is cofibrant, and the inclusion of the root is a cofibration.
\end{mythm}

\begin{proof}
	Let $\cat{T}$ be a countable poset generated by a directed tree, and $\operatorname{rk}\colon\cat{T}\to\mathbb{N}$ the associated rank function. We define a sequence
	\begin{align*}
		X\colon\mathbb{N}&\to\Cat\text{,}\\
			i&\mapsto X_i\text{,}\\
		(i\to i+1)&\mapsto F_i\text{,}
	\end{align*}
	such that $X_0$ is cofibrant, $\colim_\mathbb{N} X\cong \cat{T}$ and the map $X_0\to\colim_\mathbb{N} X$ is cofibrant and the inclusion of the root, which yields that $\cat{T}$ is cofibrant.
	For that purpose, let $X_0$ be the category with a single object $r$, and no non--identity morphisms. Assume that $X_i$ is a subposet of $\cat{T}$ that contains all elements with rank lower or equal to $i$. If the rank of $\cat{T}$ is lower than $i$, we set $X_{i+1}=X_i$. Otherwise we construct $X_{i+1}$ from $X_i$ as follows: Let $J\subseteq\cat{T}$ be the set of all elements with rank $i$. Given $j\in J$, we define 
	\begin{align*}
		K_j := \left\{\,t\in\cat{T}\,\middle|\,\operatorname{rk}(t)=i+1\text{ and } t\ge j\,\right\}\text{,}
	\end{align*}
	\ie the set of all rank $i+1$ elements that are smaller than $j$. Let $\cat{D}=x\to y$. We obtain $X_{i+1}$ via the pushout
	\begin{equation*}
		\begin{tikzcd}
		\ar[dd, "h"]\coprod\limits_{j\in J}\left(\coprod\limits_{k\in K_j}\tel\right)\ar[rrr,"\coprod\limits_{j\in J}\left(\coprod\limits_{k\in K_j}f\right) "]&&&\coprod\limits_{j\in J}\left(\coprod\limits_{k\in K_j}\cat{D}\right)\ar[dd]\\\\
		X_i\ar[rrr, "F_i"]&&&X_{i+1}
		\end{tikzcd}
	\end{equation*}
	where the maps $f$ and $h$ are given by
	\begin{align*}
		f\colon\tel&\to\cat{D}\text{,}\\
			0&\mapsto x\text{,}
	\intertext{and}
		h\colon \coprod\limits_{j\in J}\left(\coprod\limits_{k\in K_j}\tel\right)&\to X_i\\
			(0_k)_j &\mapsto j\text{.}
	\end{align*}
	Since $f$ is a cofibration, so is $\coprod_{j\in J}\left(\coprod_{k\in K_j}f\right)$ and hence $F_i$. Thus the transfinite composition $X_0\to\colim_\mathbb{N} X$---which is also the inclusion of the root---is a cofibration. Hence by Lemma~\ref{lem:imapycof}, $\colim_\mathbb{N} X$ is cofibrant. Furthermore $\cat{T}$ is a universal co-cone for $X$ by construction, which finishes the proof.
\end{proof}
\section{Cofibrant objects with a fixed number of elements}
\label{sec:finpos}
In this section, we proof that every poset with five or less elements is cofibrant and that their respective inclusions of minima are cofibrations. Most of those posets are already covered by previous theorems, so there are only a handful of posets which we have to construct by hand. Before starting with the proofs, we need a small lemma which we use throughout this section.
\begin{mylem}\label{lem:retpush} Let $Q$ be a poset, $A\subseteq Q$ a subset and $i\colon A\to[n]$ an injective map of sets. Let furthermore 
	\begin{align*}
		m\colon N(\coprod_{a\in A}\tel)&\to\Delta^n\\
		0_a&\mapsto i(a)
	\end{align*}
	be a map in $\sSet$. If there is a retract $h\colon\coprod_{a\in A}\tel\to Q$ of $\thol m$ satisfying $h(0_a)=a$, then the poset $P$ obtained by the pushout
	\begin{equation*}
		\begin{tikzcd}
			\ar[d, "h"]\coprod\limits_{a\in A}\tel\ar[r]&\ar[d,"\alpha"]\tel\\
			Q\ar[r]&P
		\end{tikzcd}
	\end{equation*}
	is cofibrant, and the inclusion $\alpha$ is a cofibration.
\end{mylem}

\begin{proof}
	Since $m$ is a monomorphism, $\thol m$ is a cofibration and thus, so is $h$. Hence $\alpha$ is a cofibration and by Lemma~\ref{lem:imapycof}, $P$ is cofibrant.
\end{proof}
\subsection{Posets with four or less elements}

\begin{mythm}\label{thm:lt3el}
	Every poset with three or less elements is cofibrant, and every inclusion of a local minimum into one of those posets is a cofibration.
\end{mythm}

\begin{proof}
	Every connected poset with three or less elements is a semilattice, hence cofibrant by Theorem~\ref{thm:sliscof} and the inclusions are cofibrations by the Corollary~\ref{cor:slinccof}. Furthermore, since coproducts of cofibrant objects are cofibrant, every poset with three or less elements is cofibrant.
\end{proof}

\begin{mythm}\label{thm:posf}
	Every poset with four elements is cofibrant, and the respective inclusions of minima are cofibrations.
\end{mythm}

\begin{proof}
	Every poset with four elements that is a coproduct of cofibrant posets is cofibrant, hence every disconnected poset with four or less elements is cofibrant by Theorem~\ref{thm:lt3el}. Up to isomorphism, there are ten connected posets with four elements. Eight of those are semilattices, hence cofibrant by Theorem~\ref{thm:sliscof}, and their respective inclusions of minima are cofibrant by Corollary~\ref{cor:slinccof}.

	The only two posets that are not semilattices are
	\begin{align*}
		P_1 = \begin{tikzcd}[ampersand replacement=\&]
			y_1\&y_2\\
			x_1\ar[u]\ar[ur]\&x_2\ar[u]
		\end{tikzcd}
		&&\text{and}&&
		P_2 = \begin{tikzcd}[ampersand replacement=\&]
			y_1\&y_2\\
			x_1\ar[u]\ar[ur]\&x_2\ar[ul]\ar[u]
		\end{tikzcd}\text{.}
	\end{align*}
	Let
	\begin{align*}
		\cat{D} = \begin{tikzcd}[ampersand replacement=\&]
			y\\
			\ar[u]x	
		\end{tikzcd}
		&&\text{and}&&
		\cat{E} = \begin{tikzcd}[ampersand replacement=\&]
			\&a\&\\
			b_1\ar[ur]\&\&\ar[ul]b_2
		\end{tikzcd}\text{.}
	\end{align*}
	We construct $P_1$ from $\cat{D}$ and $\cat{E}$ via the pushout
	\begin{equation*}
		\begin{tikzcd}
			\ar[d, "\iota_{b_1}"]\tel\ar[r, "\iota_x"]&\cat{D}\ar[d, "\alpha"]\\
			\cat{E}\ar[r, "\beta"]&P_1
		\end{tikzcd}\text{,}
	\end{equation*}
	where $\iota_{b_1}$ and $\iota_x$ are given by $\iota_{b_1}(0)=b_1$ and $\iota_x(0)=x$ respectively. Since $\iota_{b_1}$ and $\iota_x$ are cofibrations by Theorem~\ref{thm:lt3el}, so are $\alpha$ and $\beta$ and hence---in particular---the compositions with the respective inclusions of minima into $\cat{D}$ and $\cat{E}$ and every inclusion of a minimum into $P_1$ can be written as such a composition.

	Regarding $P_2$, consider the pushout
	\begin{equation*}
		\begin{tikzcd}
			\thol\partial\Delta^1\ar[r]\ar[d,"h"]&\tel\ar[d, "\alpha"]\\
			\thol\Delta^1\ar[r]& P_2
		\end{tikzcd}\text{,}
	\end{equation*}
	where $h$ is the usual boundary inclusion. Since $h$ is a cofibration, so is $\alpha$. Hence by Lemma~\ref{lem:imapycof}, $P_2$ is cofibrant. Note that $\alpha$ is one of the inclusion $0\mapsto x_k$, and since $P_2$ is symmetric, the other inclusion has to be a cofibration as well.
\end{proof}
\subsection{Posets with five Elements}\label{sec:fivecof}
In this section, we will show that every poset with five or less elements is cofibrant, and that every inclusion of a minimum into one of those is a cofibration. As before, we only have to consider posets that are connected. According to the \emph{Chapel Hill Poset Atlas}\footnote{\url{http://www.unc.edu/~rap/Posets/index.html}}, there are up to isomorphism 44 connected posets with five elements. 25 of those are semilattices. Of the remaining 19, nine can be constructed via simple pushouts in a similar fashion to $P_1$  in the previous proof, \ie by gluing a category with two elements and a single non-identity morphism between those two to a cofibrant poset with four elements. Of the remaining ten, one is $\thol\Delta^1$. To see that the inclusion of $\{\{0,1\}\}$ is a cofibration, we have to construct it by gluing two copies of the poset $\cat{E}$ from the previous section together at one of their respective local minima.

There are nine posets left that have to be considered separately. Those are
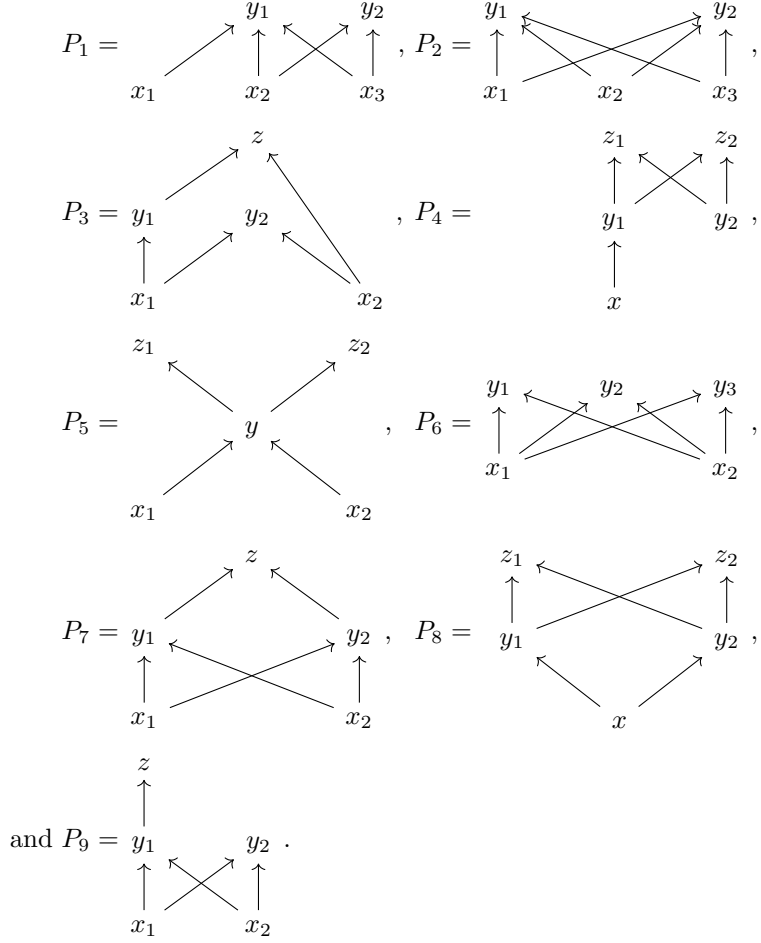
\begin{alignat*}{3}
	P_1 =& \begin{tikzcd}[ampersand replacement=\&]
	      	\& y_1\& y_2\\
		x_1\ar[ur]\&x_2\ar[u]\ar[ur]\&x_3\ar[u]\ar[ul]
	      \end{tikzcd}\text{, }&&
	P_2 =& \begin{tikzcd}[ampersand replacement=\&]
	      	y_1\&\&y_2\\
		x_1\ar[u]\ar[urr]\&\ar[ul]\ar[ur]x_2\&\ar[ull]\ar[u]x_3
	      \end{tikzcd}\text{, }\\
	P_3 =& \begin{tikzcd}[ampersand replacement=\&]
	      	\& z\\
		y_1\ar[ur]\&y_2\\
		\ar[u]x_1\ar[ur]\&\&\ar[ul]\ar[uul]x_2
	      \end{tikzcd}\text{, }&&
	P_4 =& \begin{tikzcd}[ampersand replacement=\&]
		z_1\& z_2\\
		y_1\ar[u]\ar[ur]\&\ar[u]\ar[ul]y_2\\
		x\ar[u]
	      \end{tikzcd}\text{, }\\
	P_5 =& \begin{tikzcd}[ampersand replacement=\&]
	      	z_1\&\& z_2\\
		\& \ar[ul]y\ar[ur]\\
		x_1\ar[ur]\&\&\ar[ul]x_2
	      \end{tikzcd}\text{, }&&
	P_6 =& \begin{tikzcd}[ampersand replacement=\&]
	      	y_1\&y_2\& y_3\\
		\ar[u]x_1\ar[ur]\ar[urr]\&\&\ar[ull]\ar[ul]x_2\ar[u]
	      \end{tikzcd}\text{, }\\
	P_7 =& \begin{tikzcd}[ampersand replacement=\&]
	      	\& z\\
		y_1\ar[ur]\&\& \ar[ul]y_2\\
		x_1\ar[u]\ar[urr]\&\&\ar[ull]\ar[u]x_2
	      \end{tikzcd}\text{, }&&
	P_8 =& \begin{tikzcd}[ampersand replacement=\&]
		z_1\&\& z_2\\
		y_1\ar[u]\ar[urr]\&\&\ar[ull]\ar[u]y_2\\
		\&\ar[ul]x\ar[ur]
	      \end{tikzcd}\text{, }\\
	\text{and }P_9 =& \begin{tikzcd}[ampersand replacement=\&]
	       	z\\
		y_1\ar[u]\&y_2\\
		x_1\ar[u]\ar[ur]\&x_2\ar[u]\ar[ul]
	       \end{tikzcd}\text{.}
	\end{alignat*}

\begin{mylem}\label{lem:poiscof}
	The poset $P_1$ is cofibrant, and every inclusion of a minimum into $P_1$ is a cofibration
\end{mylem}

\begin{proof}
	Let
	\begin{equation*}
		\tilde{P} = \begin{tikzcd}
		           	&y_1&&y_2\\
				x_{2_a}\ar[ur]&x_1\ar[u]&\ar[ul]x_3\ar[ur]&\ar[u]x_{2_b}
		           \end{tikzcd}\text{,}
	\end{equation*}
	and
	\begin{align*}
		h\colon\thol\partial\Delta^1&\to\tilde P\text{,}\\
		\{\{0\}\}&\mapsto x_{2_a}\text{,}\\
		\{\{1\}\}&\mapsto x_{2_b}\text{.}
	\end{align*}
	Then $P_1$ is given by the pushout diagram
	\begin{equation*}
		\begin{tikzcd}
			\ar[d, "h"]\thol\partial\Delta^1\ar[r]&\tel\ar[d, "i_{x_2}"]\\
			\tilde P\ar[r]&P_1
		\end{tikzcd}\text{,}
	\end{equation*}
	where $i_{x_2}(0)=x_2$. We have to show that $h$ is a cofibration. For that purpose, let
	\begin{align*}
		m\colon\partial\Delta^1&\to\Delta^2\text{,}\\
		0&\mapsto 0\text{,}\\
		1&\mapsto 2
	\end{align*}
	in $\sSet$. Since $m$ is a monomorphism, $\thol m$ is a cofibration. We will construct $h$ as a retract of $\thol m$. Consider the embedding of $\tilde P$ into $\thol\Delta^2$ given as follows:
	\begin{equation*}
	\poinsd\text{.}
	\end{equation*}
	Folding this along the axis between $x_1$ and $x_{2_b}$, we obtain
	\begin{equation*}
	\poinsdf\text{.}
	\end{equation*}
	From here, it is easy to see that we can obtain $h$ as a retract of $\thol m$ as indicated in the diagram above. Hence $i_{x_2}$ is a cofibration and by symmetry, so is $i_{x_3}\colon\tel\to P_1$, given by $i_{x_3}(0)=x_3$. Thus by Lemma~\ref{lem:imapycof}, $P_1$ is cofibrant.

	To show that the inclusion
	\begin{align*}
		i_{x_1}\colon\tel&\to P_1\text{,}\\
		0&\mapsto x_1
	\end{align*}
	is a cofibration, we need to construct $P_1$ differently. Consider the poset $Q$ given by
	\begin{equation*}
		Q = \begin{tikzcd}
			y_1&y_2\\
		    	\ar[u]x_1&\ar[ul]\ar[u]x_2
		    \end{tikzcd}\text{,}
	\end{equation*}
	and let
	\begin{align*}
		f\colon\thol\partial\Delta^1&\to Q\text{,}\\
		\{\{0\}\}&\mapsto y_1\text{,}\\
		\{\{1\}\}&\mapsto y_2\text{.}
	\end{align*}
	Taking the pushout
	\begin{equation*}
		\begin{tikzcd}
			\ar[d, "h"]\thol\partial\Delta^1\ar[r, "f"]&\ar[d,"\alpha"]Q\\
				\thol\Delta^1\ar[r]&\tilde P
		\end{tikzcd}\text{,}
	\end{equation*}
	where $h$ is the usual boundary inclusion, we obtain the poset $\tilde P$ given by
	\begin{equation*}
		\tilde P = \begin{tikzcd}
				z_1&z_2\\
				y_1\ar[u]&y_2\ar[u]\ar[ul]&y_3\ar[ul]\\
				&x_2\ar[ul]\ar[ur]&x_1\ar[u]
			\end{tikzcd}
	\end{equation*}
	The map $\iota_{x_1}\colon\tel\to Q$, given by $\iota_{x_1}(0)=x_1$ is a cofibration by Theorem~\ref{thm:posf}, and since $h$ is a cofibration, so is $\alpha$ and in particular the composition $\alpha\circ\iota_{x_1}$. We will construct $i_{x_1}$ as a retract of $\alpha\circ\iota_{x_1}$. For that purpose, let
	\begin{align*}
		i\colon P_1&\to\tilde P\\
		x_1&\mapsto x_1\text{,}&x_3&\mapsto y_2\text{,}\\
		x_2&\mapsto x_2\text{,}&y_k&\mapsto z_k
	\intertext{and}
		p\colon\tilde P&\to P_1\\
		x_1&\mapsto x_1\text{,}	&y_1&\mapsto y_1\text{,}	&z_1&\mapsto y_1\text{,}\\
		x_2&\mapsto x_2\text{,}	&y_2&\mapsto x_3\text{,}	&z_2&\mapsto y_2\text{,}\\
					&&y_3&\mapsto y_2\text{.}
	\end{align*}
	It is easy to see that $p\circ i=\id$ and that this gives $i_{x_1}$ as a retract of $\alpha\circ\iota_{x_1}$.
\end{proof}

\begin{mylem}
	The poset $P_2$ is cofibrant, and every inclusion of a minimum into $P_2$ is a cofibration.
\end{mylem}

\begin{proof}
	 Let
	\begin{align*}
		m\colon\partial\Delta^1&\to\Delta^1\text{,}\\
		0&\mapsto 0\text{,}\\
		1&\mapsto 1
	\end{align*}
	in $\sSet$. Since $m$ is a monomorphism, $\thol m$ is a cofibration. Let $Q$ be the poset
	\begin{equation*}
		\begin{tikzcd}
			b_1 & b_2 \\
			a_1\ar[u]\ar[ur] & a_2\ar[u]\ar[ul]
		\end{tikzcd}\text{.}
	\end{equation*}
	Consider the pushout diagram
	\begin{equation*}
	\begin{tikzcd}
		\ar[d,"\thol m"]\thol\partial\Delta^1 \ar[r,"i_b"]& Q\ar[d, "\alpha"]\\
		\thol\Delta^1\ar[r, "\beta"]&\tilde P
	\end{tikzcd}\text{,}
	\end{equation*}
	where $i_b$ is given by $i_b(\{\{k\}\}) = b_k$. The poset $\tilde P$ is given by 
	\begin{equation*}
		\tilde P = 
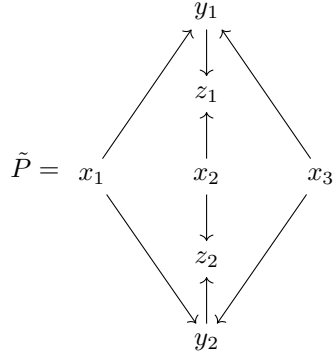
\begin{tikzcd}
		    	& y_1\ar[d]\\
			& z_1\\
			x_1\ar[uur]\ar[ddr] & \ar[u]x_2\ar[d] &\ar[uul]\ar[ddl]x_3\\
			& z_2\\
			& y_2\ar[u]
		    \end{tikzcd}
	\end{equation*}
	and the map $\alpha$ by
	\begin{align*}
		\alpha\colon Q&\to\tilde P\text{,}\\
			a_1&\mapsto x_1\text{,}\\
			a_2&\mapsto x_3\text{,}\\
			b_k&\mapsto y_k
	\end{align*}
	Since $\thol m$ is a cofibration, so is $\alpha$ and in particular the compositions $\alpha\circ \iota_{a_k}$, where
	\begin{align*}
		\iota_{a_k}\colon\tel&\to Q\text{,}\\
		0&\mapsto a_k
	\end{align*}
	for $k=1,3$. Let
	\begin{align*}
		i_{x_k}\colon\tel&\to P_2\text{,}\\
		0&\mapsto x_k\text{,}
	\end{align*}
	where $k=1,2,3$. We will construct $i_{x_1}$ as a retract of $\iota_{b_1}$ via the retract diagram
	\begin{equation*}
		\begin{tikzcd}
			\ar[d, "i_{x_1}"]\tel\ar[r]& \ar[d,"\iota_{b_1}"]\tel\ar[r]&\tel\ar[d,"i_{x_k}"]\\
			P_2\ar[r, "i"]& \tilde P\ar[r, "p"]& P_2 
		\end{tikzcd}\text{,}
	\end{equation*}
	where $i$ and $p$ are given as
	\begin{align*}
		i\colon P_2&\to\tilde P\text{,}\\
		x_k&\mapsto x_k\text{,}\\
		y_k&\mapsto y_k\text{,}
	\intertext{and}
		p\colon \tilde P&\to P_2\text{,}\\
		x_k&\mapsto x_k\text{,}\\
		y_k&\mapsto y_k\text{,}\\
		z_k&\mapsto y_k\text{.}
	\end{align*}
	It is easy to see that $p\circ i=\id$. Hence $i_{x_1}$ is a cofibration and by symmetry, so are $i_{x_2}$ and $i_{x_3}$.
\end{proof}

\begin{mylem}
	The poset $P_3$ is cofibrant, and every inclusion of a minimum into $P_3$ is a cofibration.
\end{mylem}

\begin{proof}
	We will give two constructions of $P_3$. One for each of the inclusions of $x_1$ and $x_2$ respectively. Let 
	\begin{align*}
		\tilde P_1 = \begin{tikzcd}[ampersand replacement=\&]
 	             	z\& y_2\\
 			y_1\ar[u]\& x_2\ar[ul]\ar[u]\&x_{1_b}\ar[ul]\\
 			x_{1_a}\ar[u]
 			\end{tikzcd} && \text{and}&&
		\tilde P_2 = \begin{tikzcd}[ampersand replacement=\&] 
		       	\&z\\
			x_{2_a}\ar[ur]\&y_1\ar[u]\&x_1\\
			\&y_2\ar[u]\ar[ur]\&x_{2_b}\ar[u]
		        \end{tikzcd}\text{.}
	\end{align*}
	Similar to the proof of $P_1$ being cofibrant, we construct cofibrant maps 
	\begin{align*}
		h_k\colon\thol\Delta^1&\to\tilde P_k\text{,}\\
		\{\{0\}\}&\mapsto x_{k_a}\text{,}\\
		\{\{1\}\}&\mapsto x_{k_b}
	\end{align*}
	as retracts of an inclusion $\thol m\colon\thol\partial\Delta^1\to\thol\Delta^2$, and then obtain $P_3$ via pushouts
	\begin{equation*}
		\begin{tikzcd}
			\ar[d, "h_k"]\thol\partial\Delta^1\ar[r, "f"]&\tel\ar[d, "\alpha"]\\
			\tilde P_k\ar[r,"\beta"]&P_3
		\end{tikzcd}\text{.}
	\end{equation*}
	Let
	\begin{align*}
		m\colon\partial\Delta^1&\to\Delta^2\text{,}\\
		0&\mapsto 0\text{,}\\
		1&\mapsto 2\text{.}
	\end{align*}
	We obtain $h_1$ as a retract of $\thol m$ as indicated in the diagram below
	\begin{equation*}
		\ptoinsd\text{,}
	\end{equation*}
	and $h_2$ as a retract of $\thol m$ as follows:
	\begin{equation*}
		\pttinsdf\text{.}
	\end{equation*}
	Note that we skipped the step where we apply the folding from the proof of Theorem~\ref{lem:poiscof}. Since $m$ is a monomorphism in $\sSet$, $\thol m$ is a cofibration and thus, so is $h_k$ for $k=1,2$.
\end{proof}

\begin{mylem}
	The poset $P_4$ is cofibrant, and every inclusion of a minimum into $P_4$ is a cofibration.
\end{mylem}

\begin{proof}
	Again, we have to construct $P_4$ twice, once to show that the inclusion 
	\begin{align*}
		i_x\colon\tel&\to P_4\text{,}\\
		0&\mapsto x
	\end{align*}
	is a cofibration, and once to show that 
	\begin{align*}
		i_{y_2}\colon\tel&\to P_4\text{,}\\
		0&\mapsto y_2
	\end{align*}
	is a cofibration.

	We will start with $i_x$. Let $\cat{D} = x\to y$, and $Q$ be the poset given by
	\begin{equation*}
		Q = \begin{tikzcd}
		    	&z_2\\
			y_1\ar[ur]\ar[dr]&&\ar[ul]\ar[dl]y_2\\
			&z_1
		    \end{tikzcd}\text{.}
	\end{equation*}
	We obtain $P_4$ via the pushout
	\begin{equation*}
		\begin{tikzcd}
			\ar[d, "\iota_{y_1}"]\tel\ar[r, "\iota_y"]&\cat{D}\ar[d,"\alpha"]\\
			Q\ar[r, "\beta"]& P_4
		\end{tikzcd}\text{,}
	\end{equation*}
	where $\iota_y(0)=y$ and $\iota_{y_1}(0)=y_1$. Since $\iota_{y_1}$ is a cofibration by Theorem~\ref{thm:posf}, so is $\alpha$ and hence in particular the composition $i_x = \alpha\circ\iota_x$ where $\iota_x\colon\tel\to\cat{D}$ is given by $\iota_x(0)=x$.

	To show that $i_{y_2}$ is a cofibration, let $\tilde P$ be the poset 
	\begin{equation*}
		\tilde P = \begin{tikzcd}
			z_1&&z_2\\
			y_{2_a}\ar[u]&y_1\ar[ul]\ar[ur]&y_{2_b}\ar[u]\\
			&x\ar[u]
		\end{tikzcd}\text{.}
	\end{equation*}
	We will use the same procedure as before, \ie construct the two--point embedding 
	\begin{align*}
		h\colon\thol\Delta^1&\to \tilde P\text{,}\\
		\{\{0\}\}&\mapsto y_{2_a}\text{,}\\
		\{\{1\}\}&\mapsto y_{2_b}
	\end{align*}
	as a retract of a two--point embedding into the folded $\thol\Delta^2$ as indicated by the following diagram:
	\begin{equation*}
		\pfinsdf\text{.}
	\end{equation*}
	Then apply Lemma~\ref{lem:retpush} to glue $\tilde P$ together at $y_{2_a}$ and $y_{2_b}$.
\end{proof}

\begin{mylem}
	The poset $P_5$ is cofibrant, and every inclusion of a minimum into $P_5$ is a cofibration.
\end{mylem}

\begin{proof}
	Let 
	\begin{align*}
		Q_1 = \begin{tikzcd}[ampersand replacement=\&]
		      	z_1\&\&z_2\\
			\&\ar[ur]y_a\ar[ul]
		      \end{tikzcd}
			&&\text{and}&&
		Q_2 = \begin{tikzcd}[ampersand replacement=\&]
		      	\&y_b\\
			x_1\ar[ur]\&\&x_2\ar[ul]
		      \end{tikzcd}
	\end{align*}
	and
	\begin{align*}
		\iota_{x_k}\colon\tel&\to Q_2\text{,}\\
		0&\mapsto x_k
	\end{align*}
	for $k=1,2$. Then we can construct $P_5$ via the pushout
	\begin{equation*}
		\begin{tikzcd}
			\tel\ar[r, "\iota_{y_b}"]\ar[d, "\iota_{y_a}"]& Q_2\ar[d, "\alpha"]\\
			Q_1\ar[r]& P_5
		\end{tikzcd}\text{,}
	\end{equation*}
	where $\iota_{y_a}$ and $\iota_{y_b}$ are given by $\iota_{y_a}(0)=y_a$ and $\iota_{y_b}(0)=y_b$. Since $\iota_{y_a}$ is a cofibration by Corollary~\ref{cor:slinccof}, so is $\alpha$ and thus in particular the compositions $\alpha\circ i_{x_k}$ for $k=1,2$. By Lemma~\ref{lem:imapycof}, $P_5$ is cofibrant and the minima inclusions are cofibrations.
\end{proof}

\begin{mylem}
	The poset $P_6$ is cofibrant, and every inclusion of a minimum into $P_6$ is a cofibration
\end{mylem}

\begin{proof}
	Let $Q$ be the poset
	\begin{equation*}
		Q = 
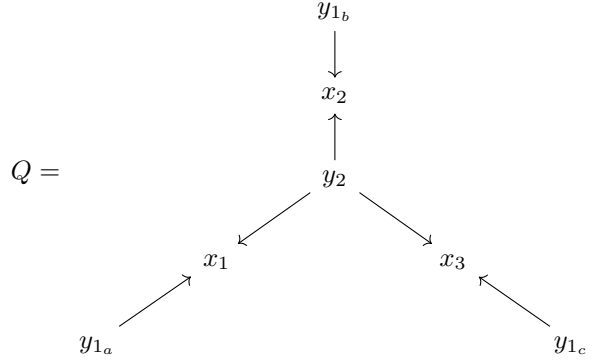
\begin{tikzcd}
		    	&&y_{1_b}\ar[d]\\
			&&x_2\\
			&&y_2\ar[u]\ar[dl]\ar[dr]\\
			&x_1&&x_3\\
			y_{1_a}\ar[ur]&&&&\ar[ul]y_{1_c}
	    \end{tikzcd}\text{,}
	\end{equation*}
	and let $h\colon\coprod_{i=1}^3\tel\to Q$ be the inclusion with image $\{y_{1_a},y_{1_b},y_{1_c}\}$. We obtain $h$ as a retract of
	\begin{align*}
		m\colon N(\coprod_{i=1}^3)\tel&\to\thol\Delta^2\text{,}\\
			0_k&\mapsto \{\{k\}\}
	\end{align*}
	as indicated by the following diagram:
	\begin{equation*}
		\psqinsd
	\end{equation*}
	We can construct $P_6$ via the pushout
	\begin{equation*}
		\begin{tikzcd}
			\ar[d,"h"]\coprod_{i=1}^3\tel\ar[r]&\tel\ar[d,"\alpha"]\\
			Q\ar[r]&P_6
		\end{tikzcd}\text{.}
	\end{equation*}
	Hence by Lemma~\ref{lem:retpush}, $P_6$ is cofibrant and the inclusion of $y_1$ into $P_6$ is a cofibration and by symmetry, so is the inclusion of $y_2$.
\end{proof}

\begin{mythm}
	The poset $P_7$ is cofibrant, and every inclusion of a minimum into $P_7$ is a cofibration.
\end{mythm}

\begin{proof}
	Let $Q$ be the poset
	\begin{equation*}
		Q = 
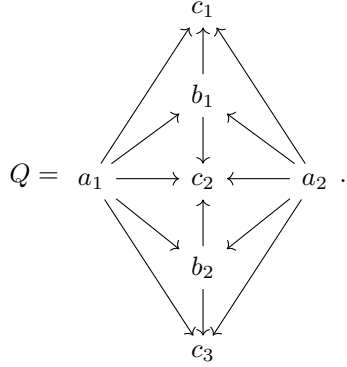
\begin{tikzcd}
		    	&c_1\\
			&b_1\ar[u]\ar[d]\\
			a_1\ar[uur]\ar[ur]\ar[r]\ar[dr]\ar[ddr]&c_2&\ar[uul]\ar[ul]\ar[l]\ar[dl]\ar[ddl]a_2\\
			&b_2\ar[u]\ar[d]\\
			&c_3
		    \end{tikzcd}\text{.}
	\end{equation*}
	We will obtain a cofibrant embedding 
	\begin{align*}	
		i_{x_1}\colon\tel&\to P_7\text{,}\\
			0&\mapsto x_1
	\end{align*}
	as a retract of the inclusion
	\begin{align*}
		\iota_{a_1}\colon\tel&\to Q\text{,}\\
		0&\mapsto a_1\text{,}
	\end{align*}
	and $\iota_{a_1}$ as a pushout of a retract of the boundary inclusion $\thol\partial\Delta^2\to\thol\Delta^2$. We will start with the embedding of the folded boundary into the folded $\thol\Delta^2$:
	\begin{equation*}
		\pictsdthf
	\end{equation*}
	Folding again at the axis between $\{\{0,1\}\}$ and $\{\{0,1,2\}\}$ as indicated in the diagram above, we obtain
	\begin{equation*}
		\pictsdttt\text{.}
	\end{equation*}
	We will denote the resulting inclusion by $m\colon B\to S$. We obtain the poset $Q$ by taking the pushout
	\begin{equation*}
		\begin{tikzcd}
			B\ar[r]\ar[d,"m"]&\ar[d, "\iota_{a_1}"]\tel\\
			S\ar[r]& Q
		\end{tikzcd}\text{,}		
	\end{equation*}
	and since $m$ is a cofibration, so is $\iota_{a_1}$. Let 
	\begin{align*}
		i\colon P_7&\to Q\text{,}\\
		x_k&\mapsto a_k\text{,}\\
		y_k&\mapsto b_k\text{,}\\
		z&\mapsto c_2
	\intertext{and}
		p\colon Q&\to P_7\text{,}\\
		a_k&\mapsto x_k\text{,}	&	c_1&\mapsto y_1\text{,}\\
		b_k&\mapsto y_k\text{,}	&	c_2&\mapsto z\text{,}\\
		&			&	c_3&\mapsto y_2\text{.}
	\end{align*}
	Then $p\circ i=\id$ and the diagram 
 	\begin{equation*}
 	\begin{tikzcd}
 		\ar[d,"i_{x_1}"]\tel\ar[r]&\ar[d, "\iota_{a_1}"]\tel\ar[r]&\tel\ar[d,"i_{x_1}"]\\
 		P_7\ar[r, "i"]& Q\ar[r,"p"]&P_7
 	\end{tikzcd}
 	\end{equation*}
	commutes. Thus, $i_{x_1}$ is a cofibration and by symmetry, so is $i_{x_2}\colon\tel\to P_7$, given by $i_{x_2}(0)=x_2$. Applying Lemma~\ref{lem:imapycof} yields that $P_7$ is cofibrant.
\end{proof}

\begin{mylem}\label{lem:peiscof}
	The poset $P_8$ is cofibrant, and the inclusion of the minimum into $P_8$ is a cofibration.
\end{mylem}

\begin{proof}
	Let $Q$ and $R$ be the posets given by
	\begin{align*}
		Q = \begin{tikzcd}[ampersand replacement=\&]
				z_1\&\&z_2\\
				y_1\ar[u]\&\ar[ul]y_2\ar[ur]\& y_3\ar[u]\\
				\&\ar[ul]\ar[u]x\ar[ur]
			\end{tikzcd}&&\text{and}&&
		R = \begin{tikzcd}[ampersand replacement=\&]
		    	y_1\&\&y_2\\
			\&\ar[ul]x\ar[ur]
		    \end{tikzcd}
	\end{align*}
	and let
	\begin{align*}
		h\colon R&\to Q\text{,}\\
		y_1&\mapsto y_1\text{,}\\
		y_2&\mapsto y_3\text{,}\\
		x&\mapsto x\text{.}
	\end{align*}
	Consider the embedding $m\colon\Delta^1\to\Delta^2$ that maps $\Delta^1$ to the $1$-simplex between $0$ and $1$. We obtain $h$ as a retract of $\thol m$ as indicated by the following diagram (where we skipped the usual folding and the image of $\thol m$ is located at the bottom of the diagram):
	\begin{equation*}
		\perqinsdf\text{.}
	\end{equation*}
	Now let $\cat{D}=x\to y$, and
	\begin{align*}
		f\colon R&\to\cat{D}\text{,}\\
		y_k&\mapsto y\text{,}\\
		x&\mapsto x\text{.}
	\end{align*}
	We obtain $P_8$ via the pushout
	\begin{equation*}
		\begin{tikzcd}
			\ar[d,"h"]R\ar[r,"f"]&\cat{D}\ar[d,"\alpha"]\\
			Q\ar[r]&P_8
		\end{tikzcd}\text{.}
	\end{equation*}
	Since $m$ is a monomorphism in $\sSet$, $\thol m$ is a cofibration and hence, so is $h$. Thus $\alpha$ is a cofibration, and since the inclusion $\iota_x\colon\tel\to\cat{D}$ given by $\iota_x(0)=x$ is a cofibration, so is the composition $i_x=\alpha\circ\iota_x$, which is the inclusion of the minimum into $P_8$. Hence by Lemma~\ref{lem:imapycof}, $P_8$ is cofibrant.
\end{proof}

\begin{mythm}
	The poset $P_9$ is cofibrant, and every inclusion of a minimum is a cofibration
\end{mythm}

\begin{proof}
	Let 
	\begin{equation*}
		\tilde P = \begin{tikzcd} 
		           	z\\
				y_1\ar[u]&y_2\\
				x_{1_a}\ar[u]&x_2\ar[ul]\ar[u]&x_{1_b}\ar[ul]
		           \end{tikzcd}
	\end{equation*}
	and
	\begin{align*}
		h\colon\thol\partial\Delta^1&\to\tilde P\text{,}\\
		\{\{0\}\}&\mapsto x_{1_a}\text{,}\\
		\{\{1\}\}&\mapsto x_{1_b}\text{.}
	\end{align*}
	Then $P_9$ is given by the pushout diagram
	\begin{equation*}
		\begin{tikzcd}
			\ar[d, "h"]\thol\partial\Delta^1\ar[r]&\tel\ar[d, "i_{x_1}"]\\
			\tilde P\ar[r]&P_9
		\end{tikzcd}\text{,}
	\end{equation*}
	where $i_{x_1}$ is the inclusion given by $i_{x_1}(0)=x_1$. Let furthermore
	\begin{align*}
		h\colon\partial\Delta^1&\to\Delta^2\text{,}\\
		0&\mapsto 0\text{,}\\
		1&\mapsto 1
	\end{align*}
	in $\sSet$. We obtain $h$ as a retract of $\thol m$ as indicated by the following diagram:
	\begin{equation*}
		\pninsdf\text{.}
	\end{equation*}
	Hence $h$ is a cofibration and thus, so is $i_{x_1}$ (and by symmetry, so is $i_{x_2}\colon\tel\to P_9$, given by $i_{x_2}(0)=x_2$). So by Lemma~\ref{lem:imapycof}, $P_9$ is cofibrant.
\end{proof}

Combining all of our previous results, we have proved the following theorem:

\begin{mythm}
	Every poset with five or less elements is cofibrant, and the respective inclusions of minima are cofibrations.
\end{mythm}


\section{Appendix}
The posets on the following pages are mostly computer generated with the data pulled from \emph{Chapel Hill Poset Atlas}. We denote the posets by their generating graphs. That means that we will only use anonymous nodes, instead of named objects and will only draw the minimal generating set of arrows, \ie those arrows, that are indecomposable.

Moreover, when arranging the objects, we put our focus on avoiding crossing arrows.
\subsection{Posets with 4 or less elements}
There is exactly one poset with one element, and one connected poset with two. 

There are three connected posets with three elements, and all three are semilattices:

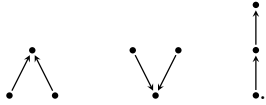
\begin{center}
	\posthree.
\end{center}

There are ten connected posets with four elements, eight of those are semilattices, namely:

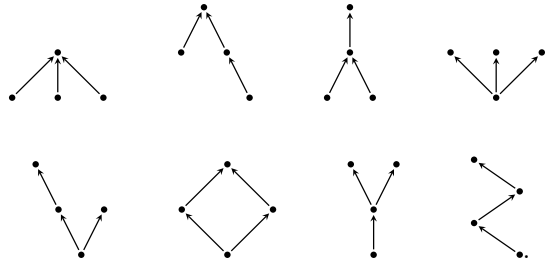
\begin{center}
	\latfour.
\end{center}

and the remaining two are the posets $P_1$ and $P_2$ from Theroem~\ref{thm:posf}:

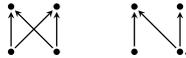
\begin{center}
	\posfour.
\end{center}

\subsection{Posets with 5 Elements}
There are 44 connected posets with five elements, 25 of those are semilattices and listed below:

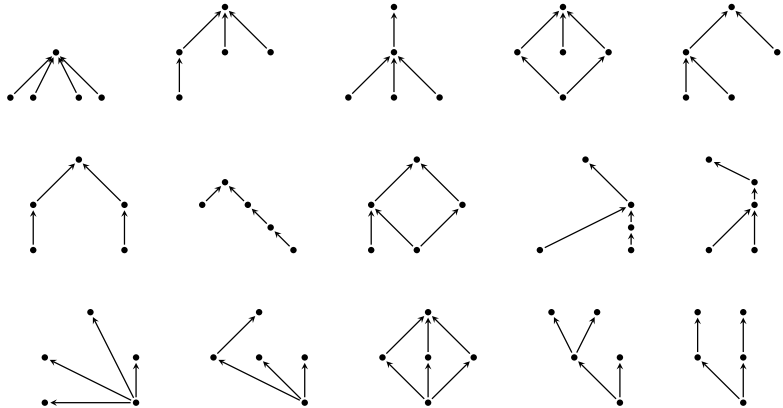
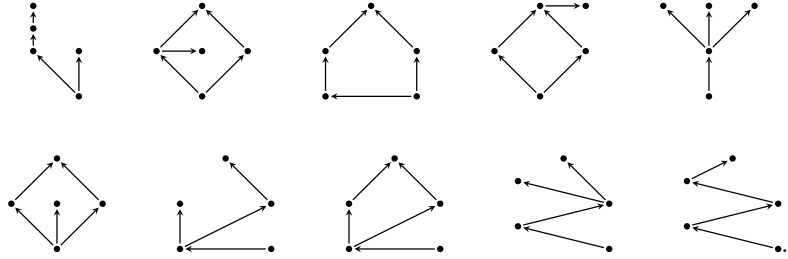
\begin{center}
\latfive.
\end{center}

Of the remaining 19, nine can be constructed by gluing the connected poset with two elements along its minimum to a cofibrant poset with four elements. Those are 

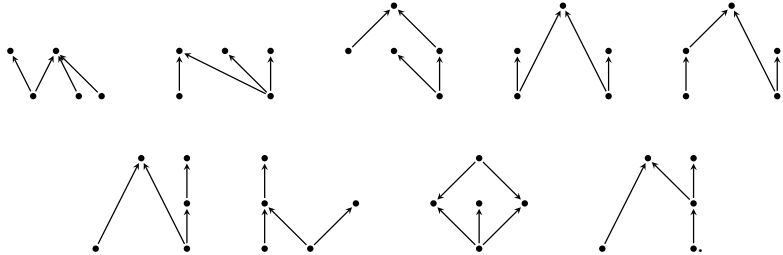
\begin{center}
\pushfive.
\end{center}

Of the remaining ten, one is $\thol\Delta^1$:

\begin{center}
\scalebox{.6}{\begin{tikzpicture}[every node/.style=my node, every path/.style=my path]
\node (a) at (1,0) {};
\node (b) at (2,0) {};
\node (c) at (0,0) {};
\node (d) at (1.5,-1) {};
\node (e) at (.5,-1) {};
\draw (a)--(d);
\draw (b)--(d);
\draw (a)--(e);
\draw (c)--(e);
\end{tikzpicture}}.
\end{center}

The remaining nine are the posets $P_1$ to $P_9$ from Section~\ref{sec:fivecof}.

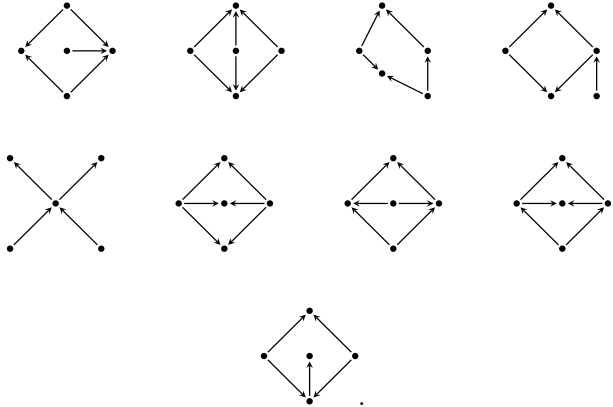
\begin{center}
	\potpn.
\end{center}

\section{Acknowledgments}
We would like to thank Viktoriya Ozornova, who has been working on the same topic and provided us with the general idea on how to prove Lemma~\ref{lem:chaincof}, as well as the proofs of Lemma~\ref{lem:poiscof} and Lemma~\ref{lem:peiscof} and supported us with several fruitful discussions about the general topic.

\bibliography{references}{}

\newcommand{\etalchar}[1]{$^{#1}$}
\providecommand{\bysame}{\leavevmode\hbox to3em{\hrulefill}\thinspace}
\providecommand{\MR}{\relax\ifhmode\unskip\space\fi MR }
\providecommand{\MRhref}[2]{%
  \href{http://www.ams.org/mathscinet-getitem?mr=#1}{#2}
}
\providecommand{\href}[2]{#2}
\begin{thebibliography}{BMO{\etalchar{+}}15}

\bibitem[AM14]{am:ncat}
Dimitri Ara and Georges Maltsiniotis, \emph{Vers une structure de cat\'egorie
  de mod\`eles \`a la {T}homason sur la cat\'egorie des {$n$}-cat\'egories
  strictes}, Adv. Math. \textbf{259} (2014), 557--654. \MR{3197667}

\bibitem[BMO{\etalchar{+}}15]{bmoopy:gcat}
Anna~Marie Bohmann, Kristen Mazur, Ang{\'e}lica~M. Osorno, Viktoriya Ozornova,
  Kate Ponto, and Carolyn Yarnall, \emph{A model structure on {$G Cat$}}, Women
  in topology: collaborations in homotopy theory, Contemp. Math., vol. 641,
  Amer. Math. Soc., Providence, RI, 2015, pp.~123--134. \MR{3380072}

\bibitem[{Bru}15]{bruckner:amsosac}
Roman {Bruckner}, \emph{{A Model Structure On The Category Of Small Acyclic
  Categories}}, ArXiv e-prints (2015).

\bibitem[Cis99]{cisinski:dnpr}
Denis-Charles Cisinski, \emph{La classe des morphismes de {D}wyer n'est pas
  stable par retractes}, Cahiers Topologie G\'eom. Diff\'erentielle Cat\'eg.
  \textbf{40} (1999), no.~3, 227--231. \MR{1716777 (2000j:18008)}

\bibitem[FP10]{fiorepaoli:nfcat}
Thomas~M. Fiore and Simona Paoli, \emph{A {T}homason model structure on the
  category of small {$n$}-fold categories}, Algebr. Geom. Topol. \textbf{10}
  (2010), no.~4, 1933--2008. \MR{2728481 (2011i:18011)}

\bibitem[MSZ16]{msz:htep}
Peter {May}, Marc {Stephan}, and Inna {Zakharevich}, \emph{The homotopy theory
  of equivariant posets}, ArXiv e-prints (2016).

\bibitem[Rap10]{raptis:hotopo}
George Raptis, \emph{Homotopy theory of posets}, Homology, Homotopy Appl.
  \textbf{12} (2010), no.~2, 211--230. \MR{2721035 (2011i:18031)}

\bibitem[Tho80]{thomason:ccmc}
Robert~W. Thomason, \emph{Cat as a closed model category}, Cahiers Topologie
  G\'eom. Diff\'erentielle \textbf{21} (1980), no.~3, 305--324. \MR{591388
  (82b:18005)}

\end{thebibliography}
\bibliographystyle{amsalpha}

\end{document}